\newif\ifarxiv
\newif\ifims
\arxivtrue

\ifims
\documentclass[aos]{imsart}
\fi

\ifarxiv
\documentclass[a4,11pt]{article}
\usepackage[margin=1in]{geometry}
\fi

\RequirePackage{amsthm,amsmath,amsfonts,amssymb}

\ifims
\RequirePackage[numbers]{natbib}
\fi

\RequirePackage[colorlinks,citecolor=blue,urlcolor=blue]{hyperref}
\RequirePackage{graphicx}

\ifims
\startlocaldefs
\fi

\numberwithin{equation}{section}
\theoremstyle{plain}
\newtheorem{theorem}{Theorem}[section]
\newtheorem{proposition}[theorem]{Proposition}
\newtheorem{lemma}[theorem]{Lemma}
\newtheorem{corollary}[theorem]{Corollary}
\newtheorem*{lemma*}{Lemma}

\theoremstyle{definition}

\newtheorem{assumption}{Assumption}[section]
\newtheorem{example}{Example}[section]

\usepackage[utf8]{inputenc}

\ifims
\fi

\usepackage[T1]{fontenc}
\usepackage{array}

\ifims

\fi

\usepackage{multirow}
\usepackage{booktabs}
\usepackage{afterpage}
\usepackage{enumerate}
\usepackage{mathtools}

\DeclarePairedDelimiter\set{\{}{\}}
\DeclarePairedDelimiter\abs{\lvert}{\rvert}
\DeclarePairedDelimiter\ceil{\lceil}{\rceil}
\DeclarePairedDelimiter\floor{\lfloor}{\rfloor}

\newcommand{\Z}{\mathbb{Z}}

\newcommand{\R}{\mathbb{R}}

\renewcommand{\Pr}{\mathbb{P}}
\newcommand{\E}{\mathbb{E}}
\newcommand{\Var}{\textup{Var}}

\DeclareMathOperator{\SD}{\textup{SD}}
\DeclareMathOperator{\mat}{\textup{mat}}
\newcommand{\I}{\mathbb{I}}
\newcommand{\diff}{\mathrm{d}}
\newcommand{\intdiff}{\,\diff}

\DeclareMathOperator*\argmin{\textup{arg\,min}}

\DeclareMathOperator{\Bin}{\textup{Bin}}
\DeclareMathOperator{\Ber}{\textup{Ber}}
\DeclareMathOperator{\Poi}{\textup{Poi}}

\newcommand\inner[2]{\langle #1, #2 \rangle}

\DeclarePairedDelimiter\norm{\lVert}{\rVert}

\DeclareMathOperator{\diag}{\textup{diag}}

\newcommand{\TBM}{\operatorname{TBM}}
\newcommand{\eqd}{\overset{d}{=}}

\usepackage{algorithm}
\usepackage[noend]{algpseudocode}
\algrenewcommand{\algorithmicdo}{}

\usepackage{mathrsfs}
\usepackage{pdfpages}
\usepackage{titlesec}

\usepackage{subfig}
\usepackage{hyperref}
\hypersetup{colorlinks=true, urlcolor=blue, linkcolor=blue, citecolor=blue}
\usepackage{subfiles}
\usepackage{datetime}
\usepackage{lslcomm}
\usepackage{lslabbrv}
\usepackage{lslnames}
\usepackage{lslhats}
\usepackage[normalem]{ulem}

\newcommand{\cA}{\mathcal{A}}

\newcommand{\cS}{\mathcal{S}}

\newcommand{\cX}{\mathcal{X}}
\newcommand{\cY}{\mathcal{Y}}

\newcommand{\bS}{\mathbb{S}}

\newcommand{\nhquad}{\!\!\!}
\newcommand{\spenorm}[1]{{\lVert#1\rVert}_{\rm sp}}
\newcommand{\fronorm}[1]{{\lVert#1\rVert}_{\rm F}}
\newcommand{\maxnorm}[1]{{\lVert#1\rVert}_{\rm max}}
\newcommand{\onenorm}[1]{{\lVert#1\rVert}_1}

\newcommand{\inftynorm}[1]{{\lVert#1\rVert}_{\infty\rightarrow\infty}}

\newcommand{\Bigabs}[1]{{\Big\lvert#1\Big\rvert}}

\newcommand{\indic}{\mathbb{I}}
\newcommand{\Ctrim}{C_{\rm{trim}}}

\renewcommand{\eps}{\varepsilon}

\newcommand{\pr}{\Pr}
\newcommand{\weq}{\ = \ }
\newcommand{\wapprox}{\ \approx \ }
\newcommand{\wle}{\ \le \ }
\newcommand{\wge}{\ \ge \ }
\newcommand{\wll}{\ \ll \ }

\newcommand{\iprod}[1]{\langle #1 \rangle}
\let\loss\ell
\renewcommand{\ell}{l}

\renewcommand{\mnote}[1]{\marginpar{\parbox{25mm}{\color{red} \tiny #1}}}

\usepackage{xstring}
\newcommand{\SPcite}[1]{%
  \IfStrEqCase{#1}{%
    {bounded 01}{\cite[Proposition~6.1:(ii)]{Leskela_Valimaa_2025}}%
    {bounded -11}{\cite[Proposition~6.1:(iv)]{Leskela_Valimaa_2025}}%
    {balanced}{\cite[Proposition~5.1:(ii)]{Leskela_Valimaa_2025}}%
    {independent sum}{\cite[Proposition~5.2]{Leskela_Valimaa_2025}}%
    {absolute value}{\cite[Proposition~5.3]{Leskela_Valimaa_2025}}%
    {examples}{\cite[Examples]{Leskela_Valimaa_2025}}%
    {Bennett}{\cite[Proposition~3.1]{Leskela_Valimaa_2025}}%
    {Bernstein1}{\cite[Proposition~3.1]{Leskela_Valimaa_2025}}%
    {Bernstein2}{\cite[Proposition~3.1]{Leskela_Valimaa_2025}}%
  }[\textbf{Unknown keyword}]%
}

\newdateformat{mydate}{\THEDAY \ \monthname[\THEMONTH] \THEYEAR}
\mydate

\ifims
\endlocaldefs
\fi

\begin{document}

\ifims
\begin{frontmatter}
\title{Consistent spectral clustering in sparse tensor block models}
\begin{aug}
\author[A]{\fnms{Ian}~\snm{\Valimaa}\ead[label=e1]{ian.valimaa@aalto.fi}}
\and
\author[A]{\fnms{Lasse}~\snm{\Leskela}\ead[label=e2]{lasse.leskela@aalto.fi}}
\address[A]{Department of Mathematics and Systems Analysis, Aalto University\printead[presep={,\ }]{e1,e2}}
\end{aug}
\fi

\ifarxiv
\title{Consistent spectral clustering in sparse \\ tensor block models}
\author{Ian Välimaa and Lasse Leskelä}
\date{\today}
\maketitle
\fi

\begin{abstract}
High-order clustering aims to classify objects
in multiway datasets that are prevalent in various fields such as bioinformatics, recommendation systems, and social network analysis.
Such data are often sparse and high-dimensional, posing significant statistical and computational challenges.
This paper introduces a tensor block model specifically designed for sparse integer-valued data tensors.
We propose a simple spectral clustering algorithm augmented with a trimming step to mitigate noise fluctuations,
and identify a density threshold that ensures the algorithm's consistency.  Our approach models sparsity using a sub-Poisson noise concentration framework, accommodating heavier than sub-Gaussian tails.
Remarkably, this natural class of
tensor block models is closed under aggregation across arbitrary modes.
Consequently, we obtain a comprehensive framework for
evaluating the tradeoff between signal loss and noise reduction
incurred by aggregating data.
The analysis is based on a novel concentration bound for sparse random Gram matrices.
The theoretical findings are illustrated through numerical experiments.
\end{abstract}

\ifims
\end{frontmatter}
\fi

\tableofcontents

\section{Introduction}

\subsection{Background}
Multiway clustering refers to the statistical problem of analyzing an array of size
$n_1 \times \cdots \times n_d$ by grouping similar elements along one or more of its $d$ modes. Such arrays arise in many applications, including multi-tissue gene expression data
where arrays have modes corresponding to individuals, genes, and tissues \cite{Hore_etal_2016, Mele_etal_2015, Wang_Fischer_Song_2019}; recommendation systems, where user–item interactions include additional modes for context \cite{Bu_etal_2010,LaGatta_etal_2023}; 
and computer vision, where 
multiway relationships among pixels, features, and object parts help capture complex visual patterns
\cite{Agarwal_etal_2005,Alaluusua_Avrachenkov_Kumar_Leskela_2025+}. In many of these settings, the observed data consist of nonnegative integers, with entries representing count data such as gene transcript abundances and interaction frequencies.
%

To study such settings, we focus on tensor block models for count-valued arrays.
A canonical example is a $d$-way array $\cY \in \Z_{\ge 0}^{n \times \cdots \times n}$
in which the entries $\cY(i_1, \dots, i_d)$ are Poisson distributed with mean $\rho \cS_{z(i_1) \dots z(i_d)}$.
Here $z\colon \{1, \dots, n\} \to \{1, \dots, r\}$ denotes the cluster membership vector,
$\rho > 0$ is a scaling parameter controlling sparsity, and
$\cS \in [0,1]^{r \times \cdots \times r}$ is
a $d$-way core tensor characterizing how cluster memberships affect interactions.
This framework generalizes the 
$d$-uniform hypergraph stochastic block model (HSBM)
\cite{Ghoshdastidar_Dukkipati_2014}
from binary data to integer-valued count data.
The aim is to recover the cluster membership vector $z$ from the observed
data tensor,
identifying clusters up to permutation of labels.
Sparsity is modeled by letting $\rho \to 0$ as $n \to \infty$,
while keeping the core tensor $\cS$ fixed.
The key theoretical objective is to identify
critical thresholds of $\rho$ that separate the regimes
in which consistent recovery is statistically possible
from those in which it is also computationally tractable.

In regular $d$-uniform HSBM models, the information-theoretic consistency threshold is
\begin{equation}
 \label{eq:threshold d-1}
 \rho \asymp n^{-(d-1)},
\end{equation}
meaning that
no algorithm can recover the clusters better than random guessing
when $\rho \ll n^{-(d-1)}$
\cite{Gu_Polyanskiy_2023,Mossel_Neeman_Sly_2015,Mossel_Neeman_Sly_2018},
whereas consistent clustering is possible when $\rho \gg n^{-(d-1)}$
\cite{Dumitriu_Wang_Zhu_2025}.
More importantly, polynomial-time clustering algorithms have been shown to 
achieve consistency
in sparse regimes with $\rho \gg n^{-(d-1)}$
\cite{Chien_Lin_Wang_2019,Dumitriu_Wang_Zhu_2025,Ghoshdastidar_Dukkipati_2017_AoS}.
Moreover, exact recovery of clusters is known to be achievable when 
$\rho$ exceeds the threshold \eqref{eq:threshold d-1}
by a logarithmic factor
\cite{Alaluusua_Avrachenkov_Kumar_Leskela_2023,Gaudio_Joshi_2023,Kim_Bandeira_Goemans_2018-07-08, Zhang_Tan_2023}.
These advances suggest that
there is no statistical--computational gap in clustering sparse uniform hypergraphs.
However, all of the aforementioned polynomial-time algorithms
rely on aggregating the observed data tensor into a matrix with entries
\[
 A_{ij}
 \weq \sum_{i_3,\dots,i_d}\cY_{iji_3\dots i_d}.
\]
The analytical accuracy guarantees for such algorithms essentially require
that the population-level aggregate matrix $\E A$
is sufficiently informative for cluster identification.
For $d$-uniform HSBMs with a core tensor of form
\[
 \cS_{z(i_1) \dots z(i_d)}
 \weq
 \begin{cases}
  a, &\quad z(i_1) = \cdots =z(i_d),\\
  b, &\quad \text{else},
 \end{cases}
\]
aggregation does not cause information loss at the population level, and, in regimes relevant for exact recovery, not even at the sample level
\cite{Bresler_Guo_Polyanskiy_2024}.
However, for more general core tensors $\cS$,
aggregation may result in information loss, rendering clustering based on
the aggregate data matrix impossible.

When the cluster membership vector $z$
is not identifiable from $\E A$, aggregation-based methods break down.
An alternative reduction from tensors to matrices is flattening,
a key component of low-rank tensor approximation methods such as
higher-order singular value decomposition
(HOSVD) \cite{DeLathauwer_DeMoor_Vandewalle_2000} and
higher-order orthogonal iteration (HOOI) \cite{DeLathauwer_DeMoor_Vandewalle_2000b}.
Unlike aggregation, flattening preserves all the information. Ke, Shi, and Xia \cite{Ke_Shi_Xia_2020} investigated flattening-based methods for
a degree-corrected block model for hypergraphs. They established consistency above a density threshold 
\begin{equation}
 \label{eq:threshold d/2 log}
 \rho\asymp n^{-d/2} \log^{1/2} n.
\end{equation}
While this threshold is significantly above the information-theoretic limit \eqref{eq:threshold d-1},
it is widely believed to be close to the optimal computational threshold for polynomial-time algorithms.
%
Indeed, Kunisky \cite{Kunisky_2025} recently showed that no low coordinate degree algorithm can detect the community structure below $\rho\asymp n^{-d/2}D(n)^{-(d-2)/2}$, where $D(n)$ is conjectured to correspond to $\omega(\log n)$ for polynomial-time algorithms. 

%
Despite these advances, theoretical understanding of clustering Poisson and other count-valued data in tensors remains limited.
Poisson distributions have heavier tails than sub-Gaussian ones, and thus require different tools for concentration analysis.
Yet Poisson models are natural in many applications, including ecology (species count per site \cite{Chiquet_Mariadassou_Robin_2018}) and political science (vote count per polling station \cite{Chiquet_Robin_Mariadassou_2019}). Because sub-Poisson tail behavior is stable under aggregation, it provides a natural and robust assumption for modeling counts. Moreover, the sub-Poisson class is broad, encompassing both Bernoulli and Poisson distributions.


\subsection{Main contributions}  

The main contributions of this paper are the following.
%
%
First, we propose a polynomial-time clustering method that avoids aggregation altogether. It flattens the data tensor into a matrix and then computes a hollow Gram matrix followed by spectral decomposition. This sidesteps the problem of an uninformative adjacency matrix. Unlike most previous methods, we include a trimming step to remove high-degree nodes.
%
%
Second, by developing a new concentration inequality for integer-valued Gram matrices, we prove that consistent clustering is possible in polynomial time when
\begin{equation}\label{eq:critical threshold}
\rho\gg n^{-d/2}.
\end{equation}
Notably, this matches the computational lower bound in \cite{Kunisky_2025} up to logarithmic factors and sharpens the threshold \eqref{eq:threshold d/2 log} by removing the logarithmic factor. This suggests a nearly optimal algorithm for this class of problems.
%
%
Third, we extend the stochastic block modeling framework from binary to integer-valued data tensors. This is enabled by applying concentration bounds that exploit the sub-Poisson tail via Bennett's inequality~\cite{Leskela_Valimaa_2025}. This becomes particularly useful when aggregating the data to a lower-order tensor, because it preserves sub-Poisson tails. 
Moreover, our algorithm supports mode-specific clustering with varying dimensions and cluster structures across modes, and our analysis establishes recovery guarantees in this general multiway setting.
%
%
For $n_1\times\dots\times n_d$-dimensional arrays, the critical threshold \eqref{eq:critical threshold} then generalizes to
\begin{equation}
 \label{eq:DensityMultiway}
\rho\gg(n_1\cdots n_d)^{-1/2}.    
\end{equation}

\subsection{Related work}  

Flattening-based multiway clustering methods have been studied in several recent works under closely related models. 
Particular instances are temporal and multilayer networks, modeled as binary arrays of dimensions $n\times n\times t$, where $t$ corresponds to the number of layers or time slots \cite{Alaluusua_Avrachenkov_Kumar_Leskela_2023,Avrachenkov_Dreveton_Leskela_2024,Avrachenkov_Dreveton_Leskela_2025+, Lei_Zhang_Zhu_2024,Lyu_Li_Xia_2023+}. 
In this setting, Lei and Lin~\cite{Lei_Lin_2022} established consistency above a density threshold $n^{-1} t^{-1/2} \log^{1/2}(n + t)$. More recently, Lei, Zhang, and Zhu \cite{Lei_Zhang_Zhu_2024} proved a computational lower bound $n^{-1}t^{-1/2}\log^{-1.4} n$, assuming a low-degree polynomial conjecture.
These thresholds align with ours up to logarithmic factors. However, these works do not address Poisson data.

Lyu, Li, and Xia \cite{Lyu_Li_Xia_2023+} studied clustering in mixture multilayer Poisson block models of size $n\times n\times t$. They provide a consistent clustering method achieving an asymptotically optimal misclassification rate. The results, however, apply only to dense models with $\rho$ being of constant order. 

Han et al.~\cite{Han_Luo_Wang_Zhang_2022} studied a real-valued tensor block model under sub-Gaussian noise. They showed that their initialization algorithm is weakly consistent when
$
\frac{\Delta^2}{\sigma^2} \gg (n_1 \cdots n_d)^{-1/2},
$
where \(\sigma\) denotes the maximum sub-Gaussian norm of the noise tensor entries, and in our asymptotic setting, \(\Delta \asymp \rho\).
In principle, their result could be applied to sparse integer-valued models as well. However, this approach leads to suboptimal density requirements. In particular, applying sub-Gaussian concentration to Bernoulli entries with mean \(\rho \ll 1\) yields \cite{Buldygin_Moskvichova_2013}
\[
 \sigma^2
 = \frac{1 - 2\rho}{2 \log \frac{1 - \rho}{\rho}}
 \sim \frac{1}{2 \log \frac{1}{\rho}},
\]
which implies a density condition of
$
\rho \gg (n_1 \cdots n_d)^{-1/4} \log^{1/2}(n_1 \cdots n_d),
$
a significantly stronger requirement than the one in \eqref{eq:DensityMultiway}.

\subsection{Organization}
This paper is structured as follows. Section \ref{sec:TBM} discusses the tensor formalism and the statistical model. Section \ref{sec:Main results} presents a spectral clustering algorithm and states the main theorem describing sufficient conditions for the consistency of the algorithm. Section \ref{sec:Numerical experiments} demonstrates the main result with numerical experiments. Section \ref{sec:Discussion} discusses related literature. Finally, Section \ref{sec:Technical overview} gives an overview of the intermediate results needed to prove the main theorem. Technical arguments are postponed to Appendices \ref{app:Preliminaries}, \ref{app:the:TrimmedMatrixConcentration}, \ref{app:lem:AX^T bound} and \ref{app:thm:hollow gram matrix bound}.
\section{Tensor block model}
\label{sec:TBM}

\subsection{Notation}

Let $\R$ and $\Z$ denote the sets of
real numbers and integers, respectively, with $\R_{\ge 0}$ and $\Z_{\ge 0}$ denoting their nonnegative elements. We write $[n] = \{1, \dots, n\}$.
The indicator of a statement $A$ is denoted $\I\{A\}$, with $\I\{A\} = 1$ if $A$ is true and $0$ otherwise. The natural logarithm is denoted by $\log$. We write $x \wedge y = \min\{x,y\}$ and $x \vee y = \max\{x,y\}$,
together with $(x)_+ = x \vee 0$.
Given nonnegative sequences $(x_n)$ and $(y_n)$, we write $x_n \ll y_n$ (or $x_n = o(y_n)$) if $x_n / y_n \to 0$, and $x_n \lesssim y_n$ (or $x_n = \mathcal{O}(y_n)$) if the ratio $x_n / y_n$ is bounded. We write $x_n \asymp y_n$ (or $x_n = \Theta(y_n)$) if both $x_n \lesssim y_n$ and $y_n \lesssim x_n$, and $x_n \sim y_n$ if $x_n / y_n \to 1$.

For vectors $u \in \R^n$, we write $\norm{u} = \sqrt{\sum_i u_i^2}$ and $\norm{u}_1 = \sum_i \abs{u_i}$.
The spectral norm of a matrix $A \in \R^{n\times m}$ is denoted by $\spenorm{A}=\sup_{\norm{u}\le 1}\norm{Au}$.
We also denote $\maxnorm{A} = \max_{ij} \abs{A_{ij}}$
and $\inftynorm{A} = \max_i \sum_j \abs{A_{ij}}$.
We write \( |A| = (|A_{ij}|) \) for the matrix of entrywise absolute values.
For a square matrix $A$, $\diag(A)$ denotes the diagonal matrix whose diagonal entries coincide with those of $A$.

For tensors, we primarily follow the standard notational conventions in \cite{DeLathauwer_DeMoor_Vandewalle_2000}.
A tensor of order $d$ is an array $\mathcal{A} \in \R^{n_1 \times \dots \times n_d}$, where each dimension $k = 1, \dots, d$ is called a mode. A slice along mode $k$ is denoted by $\mathcal{A}_{:\dots:i_k:\dots:} \in \R^{n_1 \times \dots \times n_{k-1} \times n_{k+1} \times \dots \times n_d}$. For $d = 2$, we write $\mathcal{A}_{i:}$ for the $i$th row and $\mathcal{A}_{:j}$ for the $j$th column. The Frobenius norm of a tensor is denoted by
\(
\norm{\mathcal{A}}_{\mathrm{F}} = \sqrt{ \sum_{i_1, \dots, i_d} \mathcal{A}_{i_1 \dots i_d}^2 },
\)
and the elementwise product of tensors with matching dimensions by $\mathcal{A} \odot \mathcal{B}$.
The mode-$k$ product of a tensor $\mathcal{A}\in\R^{n_1\times\dots\times n_d}$ by a matrix $B\in\R^{m\times n_k}$ is defined as a tensor $\mathcal{A}\times_kB\in\R^{n_1\times\dots\times n_{k-1}\times m\times n_{k+1}\times\dots\times n_d}$ with entries
\[
(\mathcal{A}\times_kB)_{i_1\dots i_{k-1}ji_{k+1}\dots i_d}
= \sum_{i_k}\mathcal{A}_{i_1\dots i_d}B_{ji_k}.
\]

A Tucker decomposition of a tensor \(\mathcal{A} \in \mathbb{R}^{n_1 \times \cdots \times n_d}\) expresses \(\mathcal{A}\) as
\[
\mathcal{A} = \mathcal{C} \times_1 F^{(1)} \times_2 \cdots \times_d F^{(d)},
\quad \text{with} \quad
\mathcal{A}_{i_1, \dots, i_d} = \sum_{j_1, \dots, j_d} \mathcal{C}_{j_1, \dots, j_d} F^{(1)}_{i_1 j_1} \cdots F^{(d)}_{i_d j_d},
\]
where \(\mathcal{C} \in \mathbb{R}^{m_1 \times \cdots \times m_d}\) is the core tensor and \(F^{(k)} \in \mathbb{R}^{n_k \times m_k}\) are the factor matrices.  
In the literature, including \cite{DeLathauwer_DeMoor_Vandewalle_2000}, the factor matrices \(F^{(k)}\) are often assumed to have orthonormal columns, i.e., \((F^{(k)})^\top F^{(k)} = I\). We do not impose this assumption here.

The mode-\(k\) matricization of a tensor
\(\mathcal{A} \in \mathbb{R}^{n_1 \times \cdots \times n_d}\)
is defined as a matrix
\(\mat_k(\mathcal{A}) \in \mathbb{R}^{n_k \times \frac{n_1 \cdots n_d}{n_k}}\),
where the rows correspond to indices of the \(k\)th mode and the columns correspond to the remaining modes combined. More precisely, the entries satisfy
\[
 \mathcal{A}_{i_1, \dots, i_d}
 \weq \mat_k(\mathcal{A})_{i_k, j(i_1, \dots, i_{k-1}, i_{k+1}, \dots, i_d)},
\]
where 
\(
 j \colon [n_1] \times \cdots \times [n_{k-1}] \times [n_{k+1}] \times \cdots \times [n_d]
 \to [ n_1 \cdots n_d / n_k]
\)
is the bijection corresponding to lexicographical ordering, which we use throughout.
As noted in \cite{DeLathauwer_DeMoor_Vandewalle_2000},
the mode-\(k\) product can be expressed via matricization as
$\mat_k(\mathcal{A}\times_k B) = B\mat_k\mathcal{A}$.
Moreover, successive products can be calculated as
$(\mathcal{A}\times_k B)\times_k C = \mathcal{A}\times_k CB$,
and $(\mathcal{A}\times_k B)\times_l C = (\mathcal{A}\times_l C)\times_k B$ for $k\ne l$.

\subsection{Statistical model}

The integer-valued \emph{tensor block model}
\(\TBM(\rho, \cS, z_1, \dots, z_d)\)
with dimensions $n_1,\dots,n_d$ and cluster counts $r_1,\dots,r_d$ 
represents the distribution of a data tensor
\(\cY \in \Z^{n_1 \times \cdots \times n_d}\)
whose entries are independent random variables with expectation
\begin{equation}
 \label{eq:TBM}
 \E \cY_{i_1, \dots, i_d}
 = \rho \, \cS_{z_1(i_1), \dots, z_d(i_d)},
\end{equation}
where \(\rho \ge 0\) is the \emph{density}, \(\cS \in [-1,1]^{r_1 \times \cdots \times r_d}\)
is the
\emph{core tensor}, and \(z_1 \in [r_1]^{n_1}, \dots, z_d \in [r_d]^{n_d}\) are the \emph{cluster membership vectors}%
\footnote{In network analysis contexts, clusters are often called \emph{communities} or \emph{blocks}.}.
An index \(i_k \in [n_k]\) is said to belong to cluster \(l_k\) along mode \(k\) if \(z_k(i_k) = l_k\).
The mean \(\cX = \E \cY\) is referred to as the \emph{signal tensor},
and the deviation \(\mathcal{E} = \cY - \cX\) as the \emph{noise tensor},
see Figure~\ref{fig:Ongelma}.

\begin{figure}[h]
 \centering
 \newcommand{\scalefactor}{0.40}
 \subfloat[\centering Data tensor $\cY$ with indices\\ ordered randomly.]
 {\includegraphics[trim={2cm 2cm 2cm 2cm},clip,scale=\scalefactor]{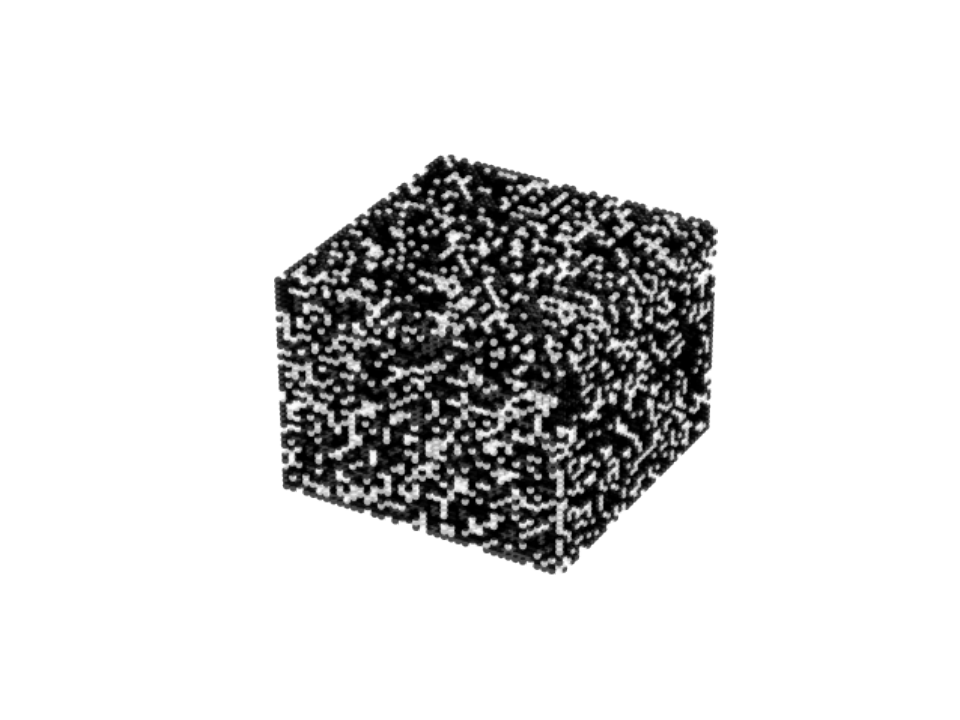}}
 \hspace{15mm}
 \subfloat[\centering Data tensor $\cY$ with indices\\ reordered by the clusters.]
 {\includegraphics[trim={2cm 2cm 2cm 2cm},clip,scale=\scalefactor]{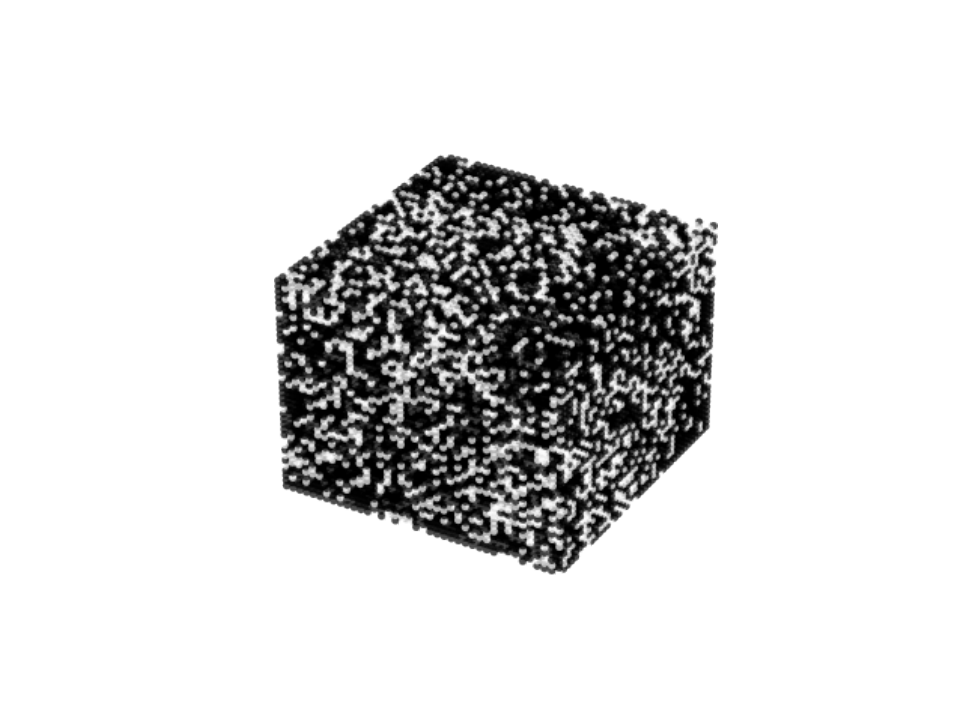}}
 \\
 \subfloat[\centering Signal tensor $\cX = \E \cY$ with indices ordered randomly.]
 {\includegraphics[trim={2cm 2cm 2cm 2cm},clip,scale=\scalefactor]{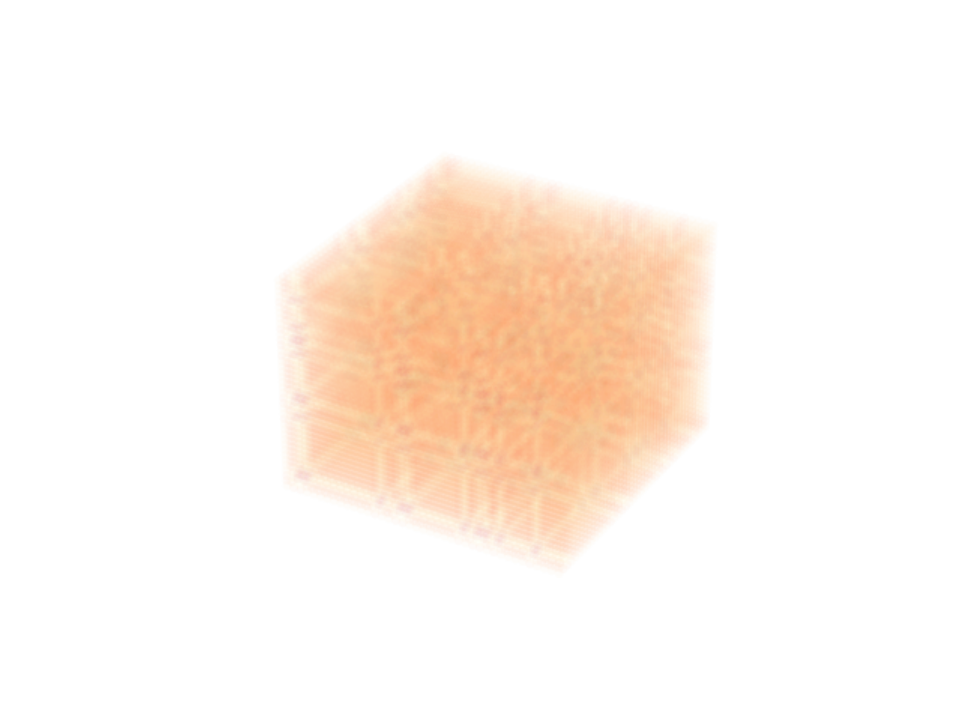}}
 \hspace{15mm}
 \subfloat[\centering Signal tensor $\cX = \E \cY$ with indices reordered by the clusters.]
 {\includegraphics[trim={2cm 2cm 2cm 2cm},clip,scale=\scalefactor]{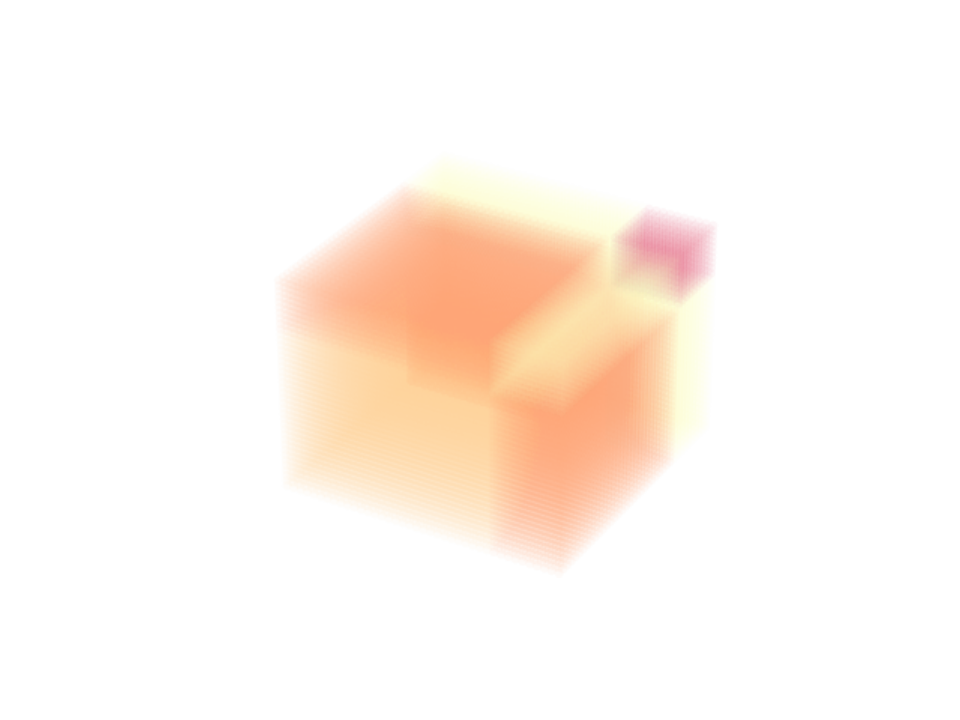}}
 \\
 \caption{Binary TBM of order $d=3$ 
 with dimensions $n_1=n_2=n_3=40$, cluster counts
 $r_1=r_2=r_3=2$,
 and identical clustering across modes (\(z_1 = z_2 = z_3\)).}
 \label{fig:Ongelma}
\end{figure}

In this work we focus on large and sparse data tensors.
This setting is captured by a sequence of models
\( \TBM(\rho^{(\nu)}, \cS^{(\nu)}, z_1^{(\nu)}, \dots, z_d^{(\nu)}) \),
with $z_k^{(\nu)} \in [r_k^{(\nu)}]^{n_k^{(\nu)}}$,
indexed by a scale parameter $\nu=1,2,\dots$ so
that
\begin{equation}
 \label{eq:LargeSparse}
 n_1^{(\nu)} \cdots n_d^{(\nu)} \to \infty
 \quad\text{and}\quad
 \rho^{(\nu)} \to 0
\end{equation}
as $\nu \to \infty$.
The number of modes \(d\) is assumed fixed and independent of the scale parameter throughout.
To avoid cluttering the notation, the scale parameter $\nu$ is omitted from the notation
whenever there is no risk of confusion.
Then the limits in \eqref{eq:LargeSparse} are simply expressed as 
$n_1 \cdots n_d \gg 1$ and $\rho \ll 1$.

In multiway clustering, the goal is to recover the underlying cluster
membership vectors \(z_1, \dots, z_d\) from the observed tensor \(\cY\).
The quality of an estimate \(\hat{z} \in [r]^n\) of a membership vector
\(z \in [r]^n\) is measured by the \emph{misclassification rate}
\begin{equation}
\label{eq:MisclassificationRate}
\loss(z, \hat{z})
=
\min_{\text{permutation } \pi: [r] \to [r]}
\frac{1}{n} \sum_{i=1}^n \mathbb{I}\{\pi(\hat{z}(i)) \ne z(i)\}.
\end{equation}
In a large-scale setting, an estimator $\hat z$
of a cluster membership vector \(z = z^{(\nu)}\),
based on input data \(\cY = \cY^{(\nu)}\),
is called
\emph{weakly consistent} if for any \(\eps > 0\),
\(\Pr(\loss(z, \hat{z}(\cY)) \ge \eps) \to 0\), and
\emph{strongly consistent} if
\(\Pr(\loss(z, \hat{z}(\cY)) = 0) \to 1\)
as \(\nu \to \infty\).

\subsection{Cluster balance and separation}
\label{sec:SignalSeparation}
The identifiability of clusters along mode $k$ requires that
the mode-$k$ matricization of the core tensor has distinct rows.
We also assume that these rows are separated by
\begin{equation}
 \label{eq:ClusterSeparation}
 \de_k
 \weq \min_{l\ne l'}
 \frac{\norm{(\mat_k \cS)_{l:}-(\mat_k\cS)_{l':}}}{\sqrt{r_1\cdots r_d/r_k}}
 \ > \ 0,
\end{equation}
in which we note that the denominator equals the square root of the width of $\mat_k \cS$.
The number $\de_k$ is called the \emph{mode-$k$ cluster separation}
of the model.

We will also assume that the cluster sizes along mode $k$ are of comparable
order, in the sense that
\begin{equation}
 \label{eq:ClusterBalance}
 \alpha_k
 \weq \min_{l \in [r_k]} \frac{\abs{z_k^{-1}\{l\}}}{n_k/r_k}
 \ > \ 0.
\end{equation}
The number $\alpha_k$ is called the \emph{mode-$k$ cluster balance coefficient}.
We note that $\al_k \le 1$, and the upper bound is achieved when
all cluster sizes along mode $k$ are exactly equal.

\subsection{Sub-Poisson data}
\label{sec:Regularity}

To analyze count data with unbounded entries, the classical sub-Gaussian assumption may be overly conservative.
The \emph{sub-Poisson} framework provides a natural extension of the sub-Gaussian paradigm to variables whose tails resemble those of a Poisson distribution, providing a variance-proxy interpretation that better reflects the behavior of such data \cite{Boucheron_Lugosi_Massart_2013,Leskela_Valimaa_2025}.
We say that an integrable real-valued random variable $X$ is sub-Poisson with \emph{variance proxy} $\sig^2 \ge 0$
if
\begin{equation}
 \label{eq:SubPoissonMain}
 \E e^{\la (X - \E X)}
 \wle \exp\bigl\{ \sig^2 (e^{\abs{\la}}-1-\abs{\la}) \bigr\}
 \quad \text{for all $\la \in \R$.}
\end{equation}
In some results, we also require a bound on the mean absolute deviation:
\begin{equation}
 \label{eq:NoiseL1}
 \E \abs{ X - \E X }
 \wle \rho.
\end{equation}
We say that $X$ is sub-Poisson with \emph{dispersion parameter} $\rho$
if \eqref{eq:SubPoissonMain}--\eqref{eq:NoiseL1} are valid with $\rho = \sig^2$.
Any binary, binomial, and Poisson random variable $X$ is sub-Poisson with
dispersion parameter $2 \E X$,
where the factor 2 accounts for the signed nature of the definition (see \cite{Leskela_Valimaa_2025} for details).

\begin{example}[Bernoulli TBM]
\label{ex:Bernoulli TBM}
In a binary TBM, the data entries are Bernoulli distributed
with mean 
$\E \cY_{i_1\dots i_d} = \cX_{i_1\dots i_d}$,
where $\cX_{i_1\dots i_d} = \rho \cS_{z_1(i_1)\dots z_d(i_d)}$.
We find that 
$\E \cY_{i_1\dots i_d} \asymp \rho$
and that the data tensor entries are sub-Poisson with
dispersion parameter $2\rho$.
When the dimensions are equal and all modes share the same 
membership vector ($z_1 = \cdots = z_d$),
this model corresponds to a directed version of the hypergraph stochastic block model \cite{Ghoshdastidar_Dukkipati_2014}.
\end{example}

\begin{example}[Poisson TBM]
Multiway nonnegative count data can be naturally represented by a 
tensor block model with Poisson distributed data entries
$\cY_{i_1\dots i_d} \eqd \Poi(\cX_{i_1\dots i_d})$,
with mean
$\E \cY_{i_1\dots i_d} =\rho \cS_{z_1(i_1)\dots z_d(i_d)} \asymp \rho$,
and the data entries are sub-Poisson with dispersion parameter $2\rho$.
\end{example}

\section{Main results}\label{sec:Main results}

\subsection{Algorithm}

Algorithm~\ref{alg:Spektraaliklusterointi} describes a spectral clustering procedure for recovering 
a latent cluster membership vector $z_k$ along mode $k$. To estimate the cluster membership vectors of all modes, the algorithm may be run $d$ times, once for each mode.

\begin{algorithm}
\caption{Hollow SVD Tensor Clustering}
\label{alg:Spektraaliklusterointi}
\begin{algorithmic}
\Require Data tensor $\cY \in \mathbb{R}^{n_1 \times \cdots \times n_d}$, mode $k$, cluster count $r_k$, trimming threshold $\tau$, relaxation constant $Q>1$
\Ensure Cluster membership vector $\hat{z}_k \in [r_k]^{n_k}$

\State Let $Y \gets \mat_k(\cY) \in \R^{n_k \times m_k}$
be the mode-$k$ matricization of $\cY$, with $m_k = \frac{n_1 \cdots n_d}{n_k}$
\State Let $|Y| \in \mathbb{R}^{n_k \times m_k}$ with entries $|Y|_{ij} = |Y_{ij}|$
\State Let $A \gets Y Y^\top - \mathrm{diag}(Y Y^\top)$
\State Let $\bar A \gets \abs{Y} \, \abs{Y}^\top - \mathrm{diag}(\abs{Y} \, \abs{Y}^\top)$

\State For $i, j \in [n_k]$: set $A_{ij} \gets 0$ if $\sum_{j'} \bar A_{i j'} > \tau$ or $\sum_{i'} \bar A_{i' j} > \tau$

%

\State Compute eigen-decomposition
$A = \sum_{i=1}^{n_k} \hat{\lambda}_i \hat{u}_i \hat{u}_i^\top$,
where $|\hat{\lambda}_1| \geq \cdots \geq |\hat{\lambda}_{n_k}|
$

\State Let
$
\hat{U} \gets [\hat{u}_1 \cdots \hat{u}_{r_k}] \in \mathbb{R}^{n_k \times r_k}$
and $
\hat{\Lambda} \gets \mathrm{diag}(\hat{\lambda}_1, \ldots, \hat{\lambda}_{r_k}) \in \mathbb{R}^{r_k \times r_k}
$

\State Cluster the rows of $\hat{U} \hat{\Lambda} \in \mathbb{R}^{n_k \times r_k}$ into $r_k$ clusters by finding
a membership vector $\hat{z}_k \in [r_k]^{n_k}$ and centroids $\hat{\theta}_1, \ldots, \hat{\theta}_{r_k} \in \mathbb{R}^{r_k}$
(for example with $k$-means++ \cite{Arthur_Vassilvitskii_2007})
such that
\[
\sum_{j=1}^{n_k} \|(\hat{U} \hat{\Lambda})_{j:} - \hat{\theta}_{\hat{z}_j}\|^2
\leq Q \min_{\breve{z}, \breve{\theta}} \sum_{j=1}^{n_k} \|(\hat{U} \hat{\Lambda})_{j:} - \breve{\theta}_{\breve{z}_j}\|^2
\]
\end{algorithmic}
\end{algorithm}

To motivate Algorithm~\ref{alg:Spektraaliklusterointi},
consider a sample $\cY$ from $\TBM(\rho,\cS,z_1,\dots,z_d)$ defined by~\eqref{eq:TBM}.
Define membership matrices $Z^{(k)}\in\set{0,1}^{n_k\times r_k}$ by $Z_{i_kj_k}^{(k)}=\I(z_k(i_k)=j_k)$ for each mode $k=1,\dots,d$. That is, $Z_{i_kj_k}^{(k)}$ indicates whether or not the cluster of the index $i_k$ along the mode $k$ is $j_k$. Then the signal tensor admits a Tucker decomposition
\[
\cX
= \rho\cS\times_1 Z^{(1)}\times_2\dots\times_d Z^{(d)}.
\]
From linear algebraic perspective, clustering amounts to inferring a low-rank Tucker decomposition from a perturbed data tensor. This motivates to estimate $\cX$ by approximating $\cY$ with a low-rank Tucker decomposition, i.e., to minimize the distance
\begin{equation}
 \label{eq:objective function}
 \fronorm{\cY - \breve{\cS}\times_1\breve{U}^{(1)}\times_2\dots\times_d\breve{U}^{(d)}},
\end{equation}
where the matrices $\breve{U}^{(k)}\in\R^{n_k\times r_k}$ are imposed to have orthonormal columns and $\breve{\cS}\in\R^{r_1\times\dots\times r_d}$ is a tensor. Unfortunately, minimizing \eqref{eq:objective function} is known to be NP-hard in general, already when the core dimensions are ones \cite{Hillar_Lim_2013}. Nonetheless, fast algorithms have been developed to approximately minimize \eqref{eq:objective function}. These include higher-order singular value decomposition (HOSVD) \cite{DeLathauwer_DeMoor_Vandewalle_2000}, higher-order orthogonal iteration (HOOI) \cite{DeLathauwer_DeMoor_Vandewalle_2000b} and simple variations of these such as sequentially truncated HOSVD \cite{Vannieuwenhoven_Vandebril_Meerbergen_2012}. These are based on observing that the matricization $\mat_k\cY$ is a noisy version of its expectation
\[
\mat_k\cX
\ =\ Z^{(k)}\mat_k\left(\rho\cS\times_1 Z^{(1)}\times_2\dots\times_{k-1}Z^{(k-1)}\times_{k+1}Z^{(k+1)}\times_{k+2}\dots\times_d Z^{(d)}\right)
\]
that has rank at most $r_k$. This motivates to estimate $\mat_k\cX$ with a rank-$r_k$ approximation of $\mat_k\cY$, or alternatively, estimate $(\mat_k\cX)(\mat_k\cX)^\top$ with a rank-$r_k$ approximation of  $(\mat_k\cY)(\mat_k\cY)^\top$. HOSVD repeats this rank approximation over each mode and HOOI iterates
this type of computation multiple times. Algorithm~\ref{alg:Spektraaliklusterointi}, in turn, essentially does this rank approximation only once before the final clustering step. Although simple, this step is crucial for HOSVD as it provides the very first initialization.

After dimension reduction, Algorithm~\ref{alg:Spektraaliklusterointi} clusters the low-dimensional vectors by solving $k$-means minimization task quasi-optimally. Exact minimization is known to be NP-hard in general \cite{Drineas_etal_2004}, but quasi-optimal solutions can be found with fast implementations such as $k$-means++ with quasi-optimality constant $\mathcal{O}(\log r)$, where $r$ is the number of clusters \cite{Arthur_Vassilvitskii_2007}. Since $k$-means++ is a randomized algorithm, it is guaranteed to be quasi-optimal only in expectation rather than always.

Two key differences distinguish Algorithm~\ref{alg:Spektraaliklusterointi} from standard spectral clustering. First, the diagonal entries of the Gram matrix $(\mat_k\cY)(\mat_k\cY)^\top$ are zeroed giving a hollow Gram matrix. This modification has been considered already in \cite{Lei_Lin_2022}. Second, the obtained hollow Gram matrix is trimmed by removing carefully selected rows and columns. In the case of nonnegative data, those with too large $L_1$-norms are selected.

\subsection{Consistency}

The following theorem presents the main result of the paper. It confirms the weak consistency of Algorithm~\ref{alg:Spektraaliklusterointi} for data sampled from the sparse integer-valued TBM
defined by~\eqref{eq:TBM}.

\begin{theorem}[Weak consistency]
\label{thm:Weak consistency}
Assume that $\cY \in \Z^{n_1 \times \dots \times n_d}$ is sampled from
$\TBM(\rho,\cS,z_1,\dots,z_d)$ 
with $r_k\ll n_k^{1/3}$ clusters in mode $k$, and density of order
\begin{equation}
 \label{eq:weak consistency regime}
 \frac{r_k^{3/2}}{\sqrt{n_1\cdots n_d}}\vee\frac{n_k\log n_k}{n_1\cdots n_d}
 \wll \rho
 \wll n_k^{-1}(n_1\cdots n_d/n_k)^{-\eps}
\end{equation}
for some constant $\eps>0$.
Assume that the mode-$k$ clusters are separated by
$\de_k \asymp 1$ and balanced by $\al_k \asymp 1$,
and the data entries are sub-Poisson with dispersion parameter $\rho$.
Then there exists a universal constant $C$ such that the estimated cluster membership vector $\hat{z}_k$ given by
Algorithm~\ref{alg:Spektraaliklusterointi} on mode $k$ with
trimming parameter $\tau = \Ctrim \rho^2 n_1 \cdots n_d$
where $\Ctrim \ge C \vee \eps^{-1}$
is weakly consistent.
Furthermore,
\begin{align*}
&\Pr\left(
\loss(z_k,\hat{z}_k)\ge CQ
\left(\frac{r_k\sqrt{\al_k}}{\de_k^2\prod_{k'}\al_{k'}}
\left(
\sqrt{s} + \frac{\varepsilon^{-3}+\Ctrim}{\sqrt{n_1\cdots n_d}\rho} + \frac{1}{n_k}
\right)
\right)^2
\right) \\
&\qquad\leq \frac{1}{s}\left(\sqrt{\Ctrim}e^{-n_1\cdots n_d\rho^2} + \frac{3}{n_k}\right) + \frac{C}{n_k},
\end{align*}
holds eventually, when $s$ is chosen to satisfy $e^{-n_1\cdots n_d\rho^2}+n_k^{-1} \ll s \ll r_k^{-3}$.
\end{theorem}

Under extra assumptions $n_1=\dots =n_d=n$ and $r_k\asymp 1$, Condition \eqref{eq:weak consistency regime} is equivalent to $d\ge 3$ and $n^{-d/2}\ll\rho\ll n^{-1-\eps(d-1)}$.

\subsection{Aggregation}

A common method in data analysis 
is to aggregate a tensor
$\cY \in \Z^{n_1\times\dots\times n_d}$ into a lower-order tensor 
$\cY' \in \Z^{n_1\times\dots\times n_{d'}}$ with entries
\begin{equation}
 \label{eq:AggregateTensor}
 \cY_{i_1\dots i_{d'}}'
 \weq \sum_{i_{d'+1},\dots,i_d}\cY_{i_1\dots i_d}.
\end{equation}
Theorem~\ref{thm:aggregated data} shows that the class of integer-valued TBMs is closed under aggregation.
More importantly, the Bennett-type variance proxy scales smoothly during aggregation.
As a consequence, Corollary~\ref{thm:Weak consistency, aggregated data} shows that
aggregating is beneficial for sparse data as long as the signal tensor remains well separated.

\begin{theorem}[Aggregation]\label{thm:aggregated data}
Let $\cY\in \Z^{n_1\times\dots\times n_d}$ be a sample from
$\TBM(\rho,\cS,z_1,\dots,z_d)$
having sub-Poisson entries with dispersion parameter $\rho$.
Then the aggregate data tensor $\cY' \in \Z^{n_1\times\dots\times n_{d'}}$
given by \eqref{eq:AggregateTensor}
is a sample from $\TBM(\rho',\cS',z_1,\dots,z_{d'})$ with density
\[
 \rho' = n_{d'+1}\cdots n_d\rho
\] 
and core tensor
\begin{equation}
 \label{eq:AggregateCoreTensor}
 \cS_{j_1\dots j_{d'}}'
 = \frac{1}{n_{d'+1} \cdots n_d} \sum_{i_{d'+1},\dots,i_d}
 \cS_{j_1\dots j_{d'}z_{d'+1}(i_{d'+1})\dots z_d(i_d)},
\end{equation} 
in which the entries are sub-Poisson with dispersion parameter $\rho'$.
\end{theorem}

\begin{proof}
The sub-Poisson property follows by 
\SPcite{independent sum}.
By the triangle inequality, $\E\abs{\mathcal{E}_{i_1\dots i_{d'}}'}=\E\abs{\sum_{i_{d'+1},\dots,i_d}\mathcal{E}_{i_1\dots i_d}}\le n_{d'+1}\cdots n_d\rho$. 
\end{proof}

\begin{corollary}[Weak consistency in Aggregated TBM]
\label{thm:Weak consistency, aggregated data}
Assume that $\cY \in \Z^{n_1 \times \dots \times n_d}$ is sampled from
$\TBM(\rho,\cS,z_1,\dots,z_d)$ 
with $r_k\ll n_k^{1/3}$ clusters in mode $k$, and density of order
\begin{equation}
 \label{eq:weak consistency regime, aggregated}
 \frac{r_k^{3/2}}{\sqrt{n_1\cdots n_{d'}}n_{d'+1}\cdots n_d}\vee\frac{n_k\log n_k}{n_1\cdots n_d}
 \wll \rho
 \wll \frac{1}{n_kn_{d'+1}\cdots n_{d}(n_1\dots n_{d'}/n_k)^{\eps}}
\end{equation}
for some constant $\eps>0$.
Assume that clusters of mode $k$ are balanced by $\al_k \asymp 1$,
and the aggregate core tensor $\cS'$ defined by \eqref{eq:AggregateCoreTensor}
has mode-$k$ separation $\de_k' \asymp 1$.
Assume that the entries of $\cY$ are sub-Poisson with dispersion parameter $\rho$.
Then there exists a universal constant $C$ such that the Algorithm~\ref{alg:Spektraaliklusterointi} applied to mode $k$ of the aggregate tensor $\cY'$ with
trimming parameter $\tau = \Ctrim \rho^2 n_1 \cdots n_d$
where $\Ctrim \ge C \vee \eps^{-1}$
is weakly consistent.
\end{corollary}

\begin{proof}
The claim follows from Theorem \ref{thm:Weak consistency} which is applicable by Theorem \ref{thm:aggregated data}.
\end{proof}

\begin{example}
Under extra assumptions $n_1=\dots =n_d=n$ and $r_k\asymp 1$, condition \eqref{eq:weak consistency regime, aggregated} is equivalent to $n^{-(d-d'/2)}\ll\rho\ll n^{-(1+d-d')-\eps(d'-1)}$, where the lower bound is orders of magnitude smaller than $n^{-d/2}$ for the nonaggregated model in Theorem \ref{thm:Weak consistency}. That is, aggregating allows to handle much sparser data as it improves the density threshold by a factor $n^{-1/2}$ per aggregated mode. However,
the cluster separation condition \eqref{eq:ClusterSeparation} for the aggregated TBM is stronger as it considers separation in $\cS'$.
\end{example}

\section{Numerical experiments}\label{sec:Numerical experiments}

This section presents numerical experiments to demonstrate the performance of Algorithm~\ref{alg:Spektraaliklusterointi} in various sparsity regimes.

\subsection{Setup}


We focus on a three-mode Bernoulli $\TBM(\rho,\cS,z,z,z)$ with $r=2$ clusters of equal size, where 
every mode has the same membership vector $z \in \set{1,2}^n$ (see Example~\ref{ex:Bernoulli TBM}).
We consider two core tensors (see Fig.~\ref{fig:S tensor})
\begin{align*}
\cS_{\rm uninformative}
= 
\begin{bmatrix}
	1 & 0 & 0 & 1 \\
	0 & 1 & 1 & 0
\end{bmatrix},
\qquad
\cS_{\rm informative} 
= 
\begin{bmatrix}
	1 & 0 & 0 & 0 \\
	0 & 0 & 0 & 1
\end{bmatrix},
\end{align*}
represented as $\cS = [ \cS_{::1} \ \cS_{::2}]$.
Both are symmetric, so that the mode chosen for matricization does not matter.
Recall that aggregating one mode yields $\TBM(\rho',\cS',z,z)$
with a density parameter $\rho'=n\rho$ and
a core tensor $\cS_{l_1l_2}' = \frac{1}{n}\sum_{i_3}\cS_{l_1l_2z(i_3)} = \frac{1}{2}(\cS_{l_1l_21} + \cS_{l_1l_22})$ (Theorem \ref{thm:aggregated data}).
In this case we obtain
\begin{align*} 
\cS_{\rm uninformative}'
= 
\begin{bmatrix}
	1/2 & 1/2 \\
	1/2 & 1/2
\end{bmatrix},
\qquad
\cS_{\rm informative}'
= 
\begin{bmatrix}
	1/2 & 0 \\
	0 & 1/2
\end{bmatrix},
\end{align*}
and we see that the former, having constant rows, is uninformative for clustering.

\begin{figure}
\centering
\includegraphics[trim={2cm 3cm 2cm 3cm}, clip,scale=0.9]{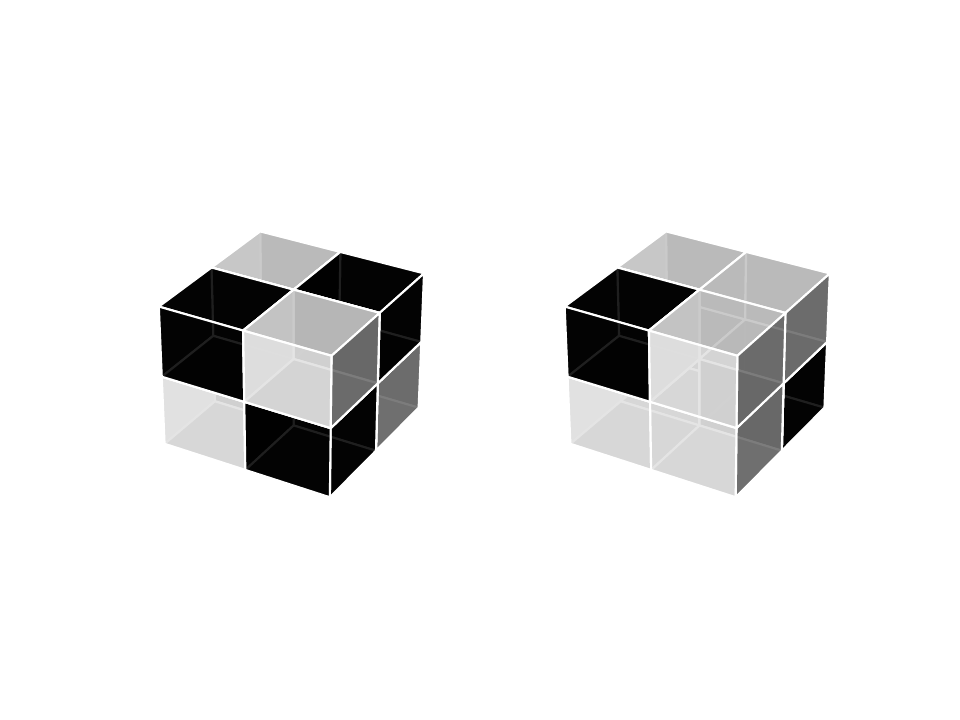}
\caption{
Core tensors $\cS_{\rm uninformative}$ (left) and $\cS_{\rm informative}$ (right) used in simulations, with black = 1, white = 0.
}\label{fig:S tensor}
\end{figure}

Different algorithms are evaluated over a grid of pairs $(n,\rho)$, with both parameters spaced logarithmically. For each pair, we record the proportion of correctly clustered nodes. Since the predicted phase transition follows $\rho = C n^{-\gamma}$, this appears as a line 
$
\log \rho = C - \gamma \log n
$
with slope $-\gamma$ in logarithmic coordinates.
To estimate the empirical transition, we fit a line in $(\log n,\log \rho)$-space using a logistic regression: we first label each run according to whether the clustering accuracy exceeds a chosen threshold (e.g.\ $0.9$), and then learn the decision boundary in log--log coordinates. Although this is not necessarily an optimal estimator of the true phase transition, it provides a visually clear proxy. For comparison, we also plot the theoretical line with slope $-\gamma$ to assess
the agreement between the simulated and predicted transitions.

The following four tensor clustering algorithms are compared:
\begin{enumerate}[(i)]
\item \textsc{Hollow SVD}. Algorithm~\ref{alg:Spektraaliklusterointi} with trimming threshold $3 n^3 \rho^2$.
By Theorem \ref{thm:Weak consistency}, we expect the phase transition line to have $\gamma=1.5$.

\item \textsc{Vanilla SVD}. The same as Algorithm~\ref{alg:Spektraaliklusterointi} but without trimming and diagonal resetting.
We expect the phase transition line to have $\gamma\approx 1.33$. Namely, the diagonal entries of the Gram matrix $(\mat_1\cY)(\mat_1\cY)^\top$ have standard deviations
\[
\begin{split}
	\SD(((\mat_1\cY)(\mat_1\cY)^\top)_{i_1i_1})
	&= \SD\left(\sum_{i_2,i_3}\cY_{i_1i_2i_3}^2\right) 
    \asymp n\sqrt{\rho},
\end{split}
\]
which dominate the spectral norm of the signal Gram matrix
\[
\begin{split}
	\spenorm{(\mat_1\cX)(\mat_1\cX)^\top}
	&\overset{(*)}{\asymp} \fronorm{(\mat_1\cX)(\mat_1\cX)^\top}
	\asymp n^3\rho^2,
\end{split}
\]
when $\rho\ll n^{-4/3}$. Here $(*)$ follows from $\spenorm{A}\leq\fronorm{A}\leq\sqrt{\operatorname{rank}(A)}\spenorm{A}$.

\item \textsc{HSC}. A simplification of high-order spectral clustering (HSC) proposed in \cite{Han_Luo_Wang_Zhang_2022}. First, HSC calculates a low-rank approximation of the data tensor with HOSVD~\cite{DeLathauwer_DeMoor_Vandewalle_2000} and essentially one iteration of HOOI \cite{DeLathauwer_DeMoor_Vandewalle_2000b}. In our simplification, HOSVD matricizes the data tensor and calculates its SVD only once since as the underlying signal tensor is symmetric.
This will slightly speed up the computations.
Then the algorithm clusters the first mode of the low-rank tensor with $k$-means++.
Since HSC is initialized with a simple rank approximation (via eigenvalue decomposition of $(\mat_1\cY)(\mat_1\cY)^\top$),
it is expected not to improve $\gamma$. That is, we expect $\gamma\approx 1.33$.

\item \textsc{Aggregate SVD}. Spectral clustering from an aggregate matrix, a simplified version of the algorithm proposed in \cite{Zhang_Tan_2023}. We compute the aggregate matrix $A_{ij}=\sum_{k}\cY_{ijk}$,
remove rows and columns with too large $L_1$ norms (trimming threshold $3 n^2 \rho$),
calculate the best rank-$2$ approximation,
and then cluster the rows with $k$-means++ \cite{Arthur_Vassilvitskii_2007}.
In the original paper, the data tensor is symmetric, which leads to minor differences in how $A$ is formed. The clustering step also does not use $k$-means++, and the algorithm includes an additional refinement stage. Nonetheless, the approaches are conceptually very similar. If the aggregate matrix is informative, we expect the phase transition to have $\gamma = 2$.
%
\end{enumerate}

\subsection{Results}

Figure \ref{fig:algoritmivertailu1} shows a comparison of the four algorithms when the aggregate matrix is uninformative.
The fitted phase transitions slopes are relatively close to their theoretical values.
We see that \textsc{Hollow SVD} has the highest accuracy, with estimated $\hat\ga = 1.43$ not far from the theoretical
value $\ga=1.5$. 
In contrast, \textsc{Vanilla SVD} and \textsc{HSC} have lower $\ga$ values close to 1.33.
The refinement steps in HSC do not improve the exponent $\gamma$ compared to the \textsc{Vanilla SVD}, but they do improve the clustering performance significantly. 
As expected, \textsc{Aggregate SVD} gives poor results overall.

\begin{figure}[h]
\centering
\includegraphics[scale=0.45]{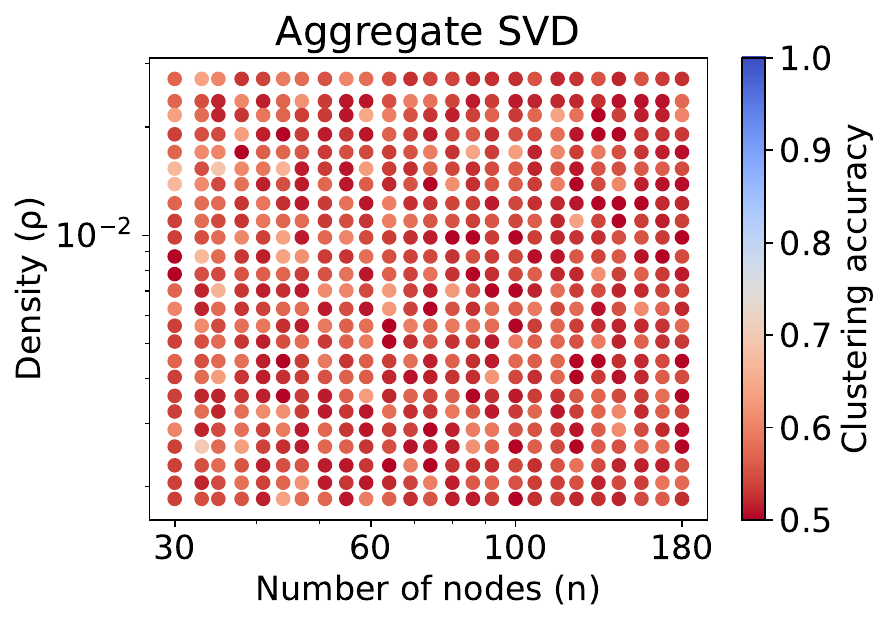}\hfill
\includegraphics[scale=0.45]{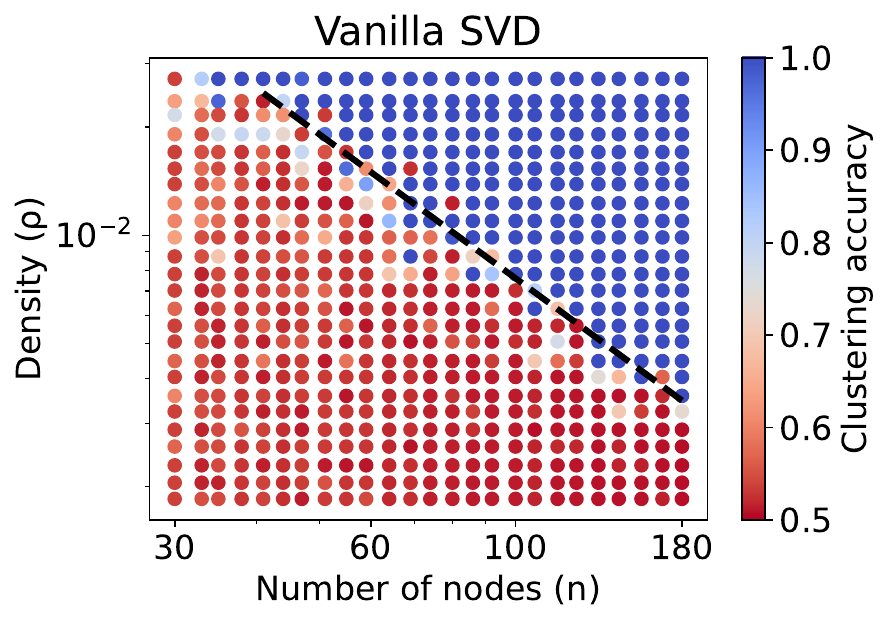} 
\newline\newline
\includegraphics[scale=0.45]{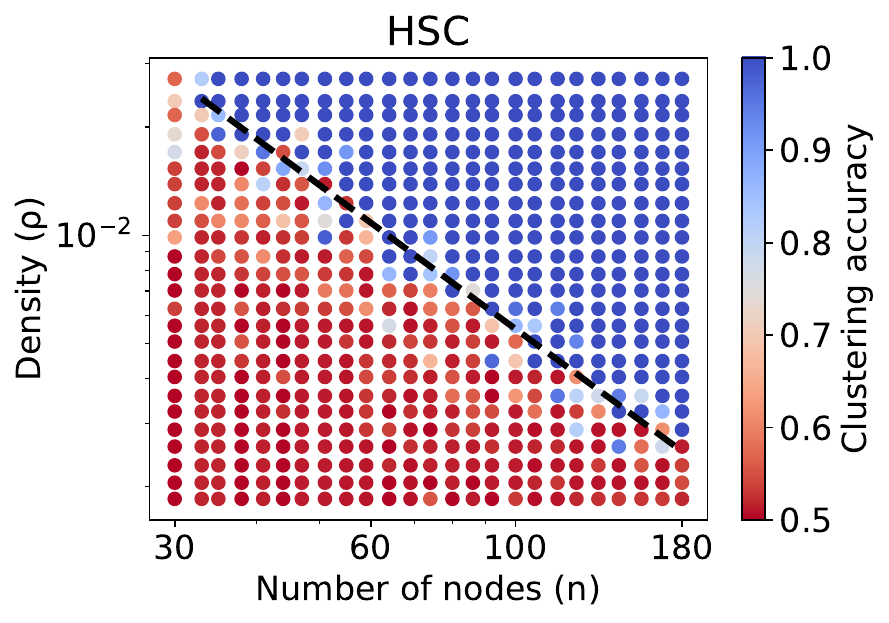}\hfill
\includegraphics[scale=0.45]{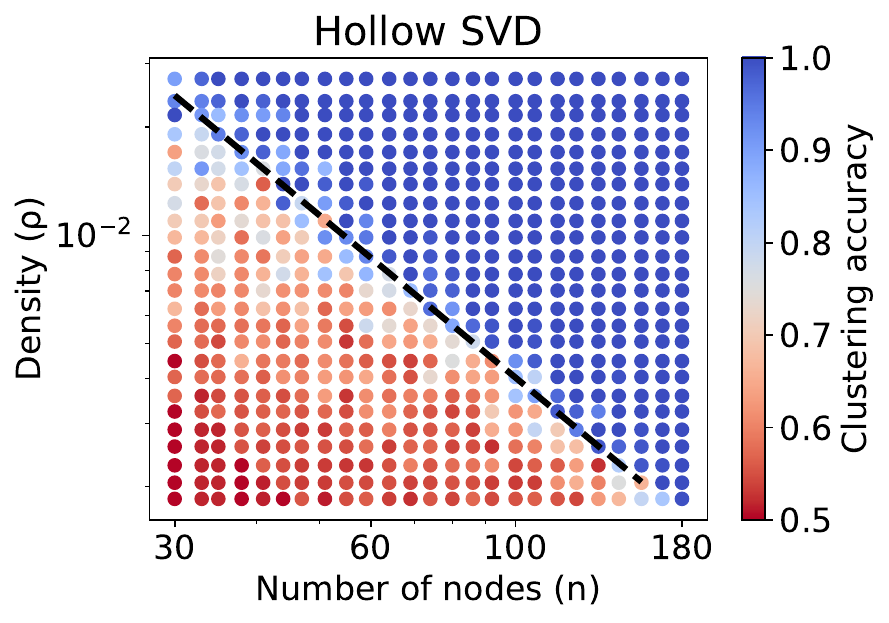}
\caption{Comparison of clustering algorithms when the aggregate matrix is uninformative. The number of nodes varies logarithmically between $30$ and $180$, and the density parameter varies logarithmically between $0.002$ and $0.027$. The dashed black line shows the theoretical phase-transition boundary corresponding to 
$\gamma = 1.33$, $1.33$, and $1.5$ for \textsc{Vanilla SVD}, \textsc{HSC}, and 
\textsc{Hollow SVD}, respectively. The corresponding estimated values are $1.29$, $1.26$, and $1.43$.
}\label{fig:algoritmivertailu1}
\end{figure}

Figure \ref{fig:algoritmivertailu2} shows a comparison of the four algorithms when the aggregate matrix is informative.
Again, the fitted $\ga$ values match well with the theoretical counterparts.
As expected from the above discussion, \textsc{Aggregate SVD} gives the best results in this case.

\begin{figure}[h]
\centering
\includegraphics[scale=0.45]{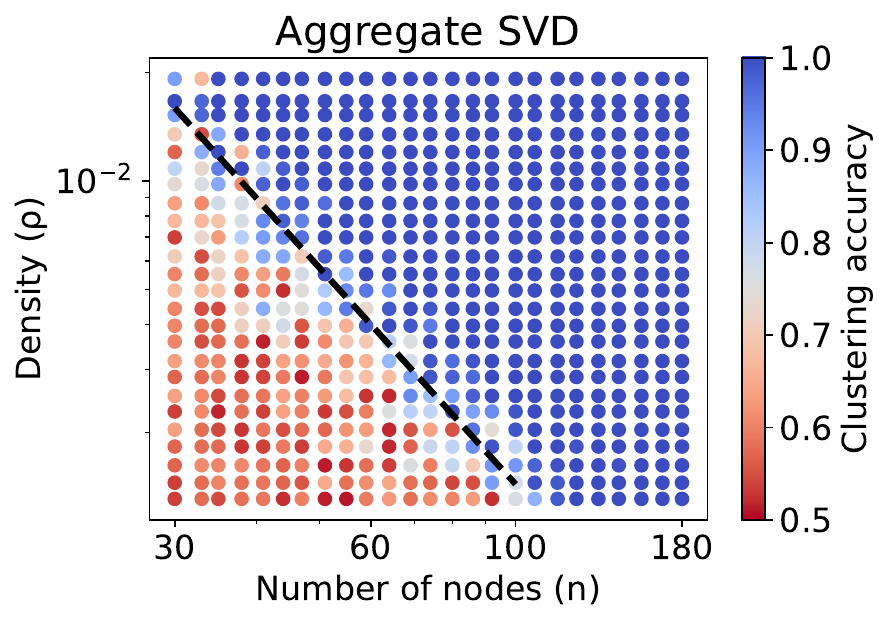}\hfill
\includegraphics[scale=0.45]{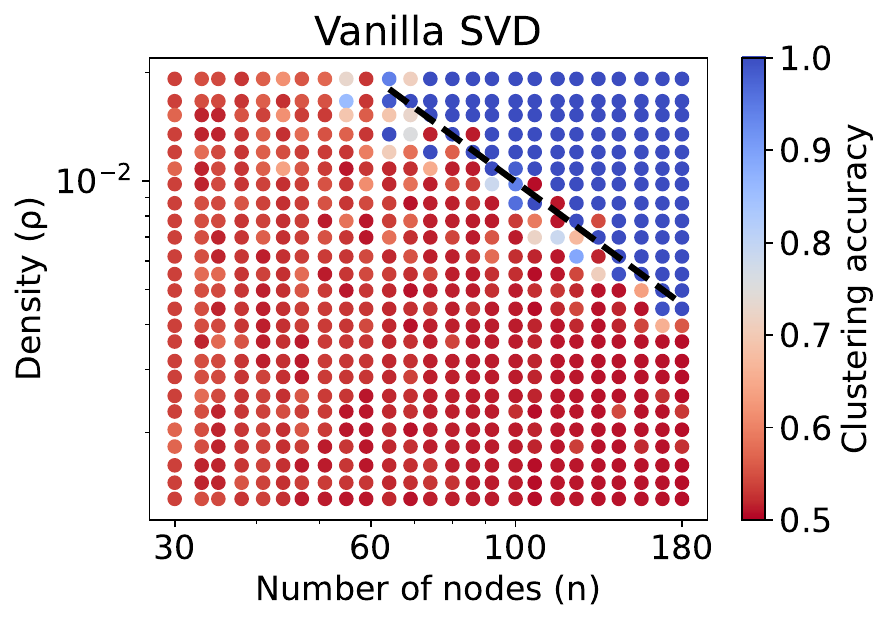} 
\newline\newline
\includegraphics[scale=0.45]{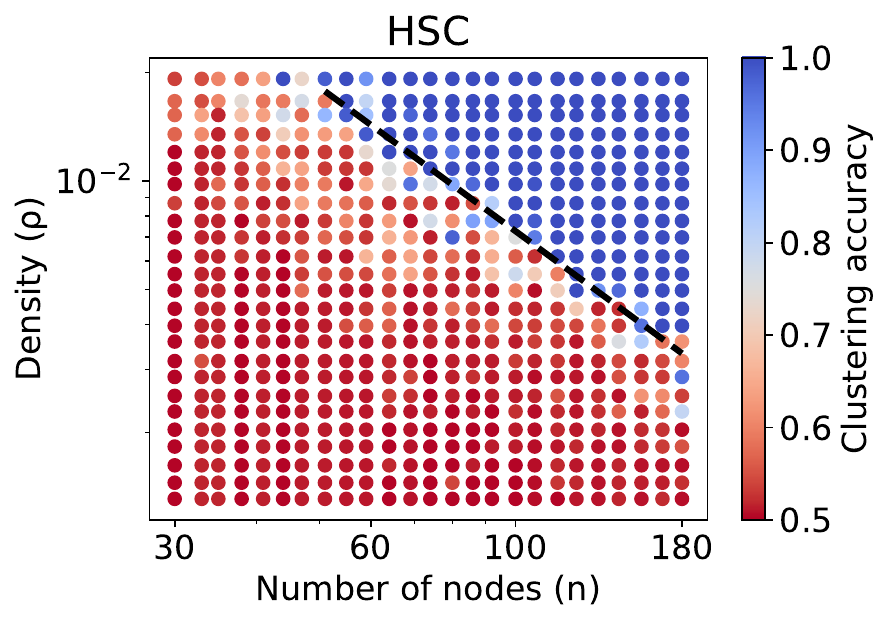}\hfill
\includegraphics[scale=0.45]{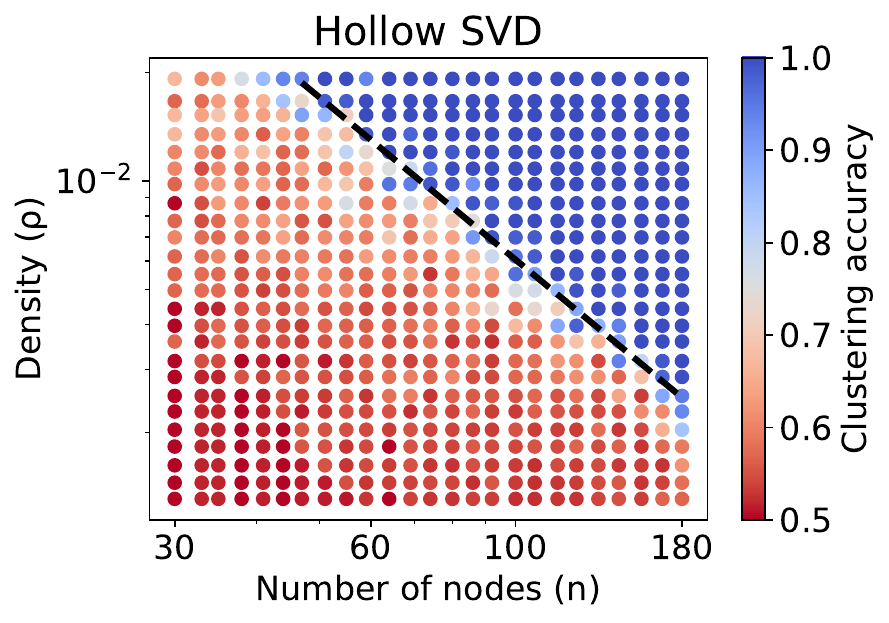}
\caption{Comparison of clustering algorithms when the aggregate matrix is uninformative. The number of nodes varies logarithmically between $30$ and $180$, and the density parameter varies logarithmically between $0.001$ and $0.019$. The dashed black line shows the theoretical phase-transition boundary corresponding to 
$\gamma = 2$, $1.33$, $1.33$ and $1.5$ for \textsc{Aggregate SVD}, \textsc{Vanilla SVD}, \textsc{HSC} and \textsc{Hollow SVD}, respectively. The corresponding estimated values are $2.00$, $1.31$, $1.24$, and $1.43$.
}\label{fig:algoritmivertailu2}
\end{figure}

\subsection{HSC initialization}

Previous simulations demonstrate that \textsc{HSC} does not improve the exponent $\gamma$ of its initialization algorithm. This subsection studies visually, how \textsc{HSC} depends on its initialization algorithm, and how it is possible that for some pairs of $n$ and $\rho$, the initialization algorithm does not cluster the nodes successfully while \textsc{HSC} does.

Similar to the experiment with an uninformative aggregate matrix, we consider a symmetric  core tensor
\[
\cS
= 
\begin{bmatrix}
1 & 1/2 & 1/2 & 1 \\
1/2 & 1 & 1 & 1/2
\end{bmatrix}.
\]
The only difference is that the entries with value $0$ have been changed to the value $1/2$. If the entries were zero, then any two row vectors of $\mat_k\cY$ from different clusters would be orthogonal. This, in turn, usually makes the projected row vectors to lie on two orthogonal lines representing different clusters, and the projected row vectors would overlap with each other in the visualization. Furthermore, when testing \textsc{HSC}, the initialization algorithm is changed from \textsc{Vanilla SVD} to the proposed \textsc{Hollow SVD}, because without any trimming step near the phase transition, there would be ``outliers'' with significantly larger norms, which would force the visualizations to focus on the outliers.

Figure \ref{fig:projektiot} shows the projections obtained from the initialization algorithm (\textsc{Hollow SVD}) and \textsc{HSC} for four different density parameters. As the density parameter decreases, the clusters merge in the projections. It may occur that the initialization provides somewhat meaningful projections, but the $k$-means algorithm cannot detect the clusters. However, such initialization might be informative enough for \textsc{HSC} to improve the projection and cluster the vectors. However, as expected, once the initialization is dominated by noise, \textsc{HSC} cannot cluster the vectors.

\begin{figure}[h!]
\centering
\subfloat[\textsc{Hollow SVD}]{
	\includegraphics[trim={0.2cm 0cm 0.2cm 0cm},clip,scale=0.43]{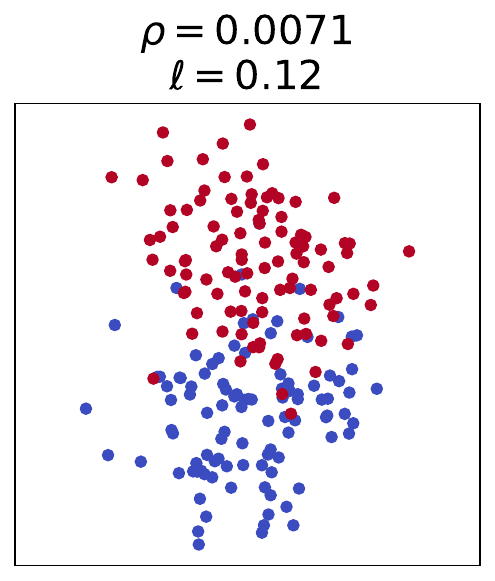}
	\includegraphics[trim={0.2cm 0cm 0.2cm 0cm},clip,scale=0.43]{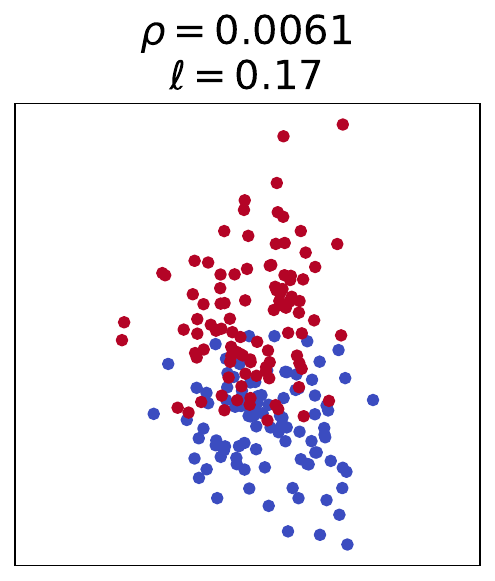}
	\includegraphics[trim={0.2cm 0cm 0.2cm 0cm},clip,scale=0.43]{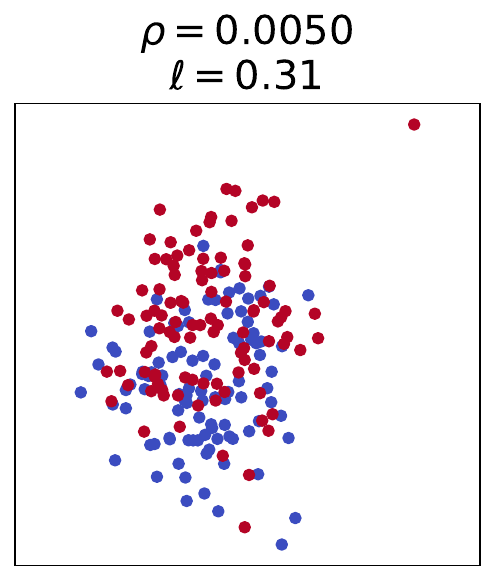}
	\includegraphics[trim={0.2cm 0cm 0.2cm 0cm},clip,scale=0.43]{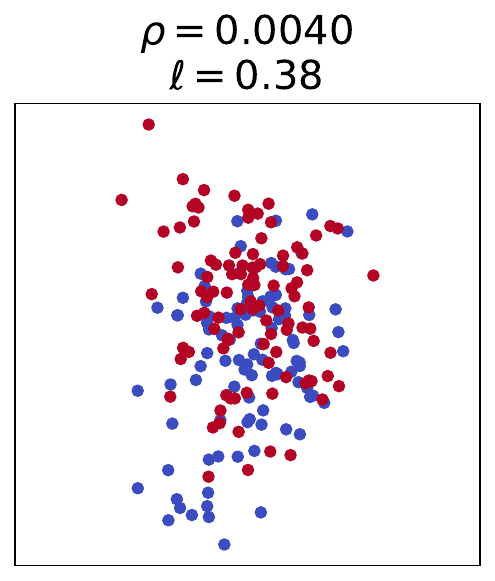}
} \\
\subfloat[\textsc{HSC} initialized with \textsc{Hollow SVD}]{
	\includegraphics[trim={0.2cm 0cm 0.2cm 0cm},clip,scale=0.43]{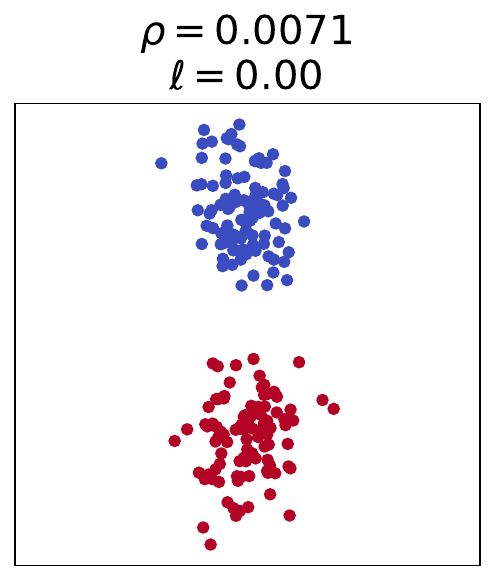}
	\includegraphics[trim={0.2cm 0cm 0.2cm 0cm},clip,scale=0.43]{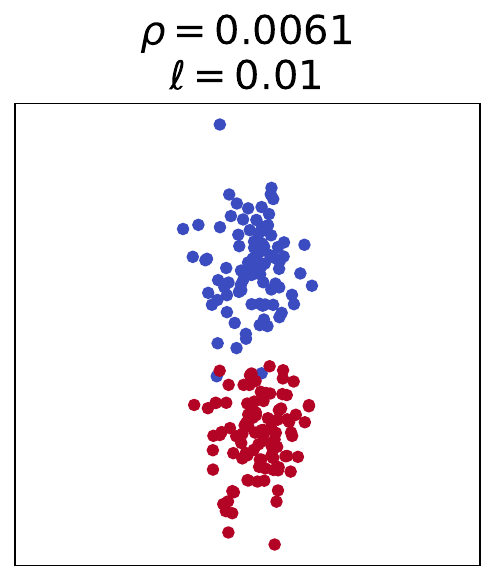}
	\includegraphics[trim={0.2cm 0cm 0.2cm 0cm},clip,scale=0.43]{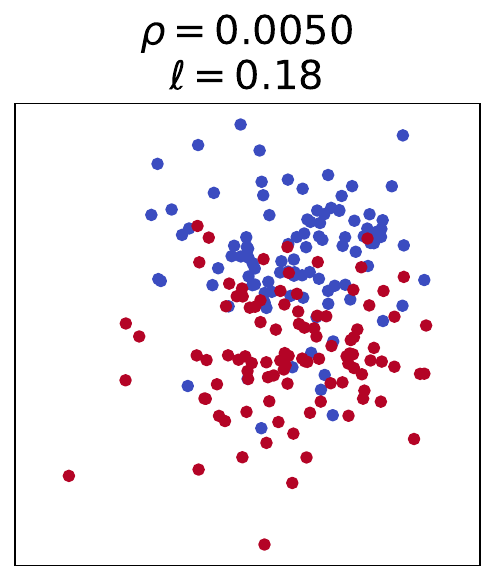}
	\includegraphics[trim={0.2cm 0cm 0.2cm 0cm},clip,scale=0.43]{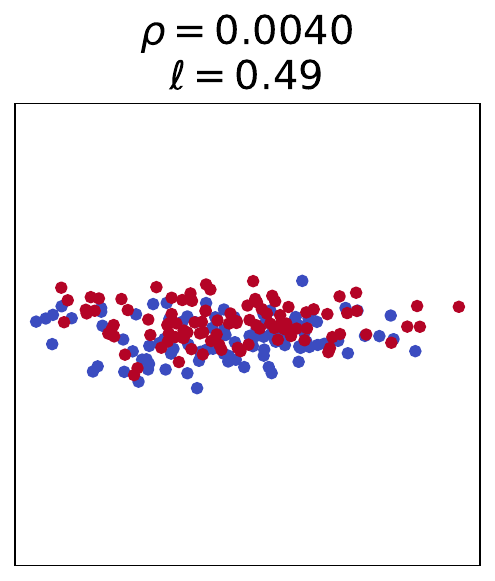}
} \\
\caption{Node embeddings computed with {\bf (a)} \textsc{Hollow SVD} and {\bf (b)} \textsc{HSC} initialized with \textsc{Hollow SVD}. The number of nodes is fixed to $n=200$ and the density parameter $\rho$ is varied around the phase transition.
}
\label{fig:projektiot}
\end{figure}

\section{Discussion}\label{sec:Discussion}

\subsection{Clustering}

This section discusses related theoretical research on multiway clustering with an emphasis on recent developments on hypergraphs and multilayer networks.

Zhang and Tan \cite{Zhang_Tan_2023} considered clustering of a $d$-uniform hypergraph on $n$ nodes, which is comparable to a TBM of order $d$ with Bernoulli distributed entries. They showed that strong consistency is impossible when $\rho\le c\frac{\log n}{n^{d-1}}$ for a sufficiently small constant $c$. Furthermore, they showed that a spectral clustering algorithm applied to an aggregate matrix $A$ with entries $A_{i_1i_2}=\sum_{i_3\dots i_d}\cY_{i_1\dots i_d}$ is strongly consistent when $\rho\ge C\frac{\log n}{n^{d-1}}$ for a sufficiently large constant $C$, assuming that $\E A$ is informative. Up to a constant factor, $A_{i_1i_2}$ counts the number of hyperedges shared by the nodes $i_1$ and $i_2$, effectively projecting the hypergraph onto a weighted graph.
Stephan and Zhu \cite{Stephan_Zhu_2024} showed that a spectral clustering based on non-backtracking walks achieves weak reconstruction (performs better than assigning all nodes to the largest cluster)
when $\rho\geq Cn^{-(d-1)}$ for a sufficiently large constant $C$, again assuming that $\E A$ is informative. However, these results do not address the case where $\E A$ is uninformative.

There has been research on multiway clustering with an uninformative adjacency matrix~$A$. Lei, Chen and Lynch \cite{Lei_Chen_Lynch_2020} analyzed an undirected multilayer network model, which can be formulated as a TBM with $\cY\in\set{0,1}^{n\times n\times t}$ with no cluster structure on the third mode (or $z_3(i)=i$, i.e., every index forms one cluster) and identical cluster structures on the first two modes $z_2=z_3=z$, $r_2=r_3=r$. Each slice $\cY_{::i}$ corresponds to a symmetric adjacency matrix of an undirected network layer in a multilayer network. They consider a clustering algorithm solving a least-squares problem
\[
 (\hat{Z},\hat{\cS})
 = \argmin_{\breve{Z},\ \breve{\cS}} \fronorm{\cY - \cS\times_1\breve{Z}\times_2\breve{Z}}^2
\]
with a membership matrix $\breve{Z}\in\set{0,1}^{n\times r}$ and a core tensor $\breve{S}\in[0,1]^{r\times r\times t}$. They did not give an algorithm finding an exact or provably approximate solution, but they did provide an algorithm trying to solve the optimization problem. They showed that $\hat{Z}$ achieves weak consistency, when $\rho\gg \frac{\log^{3/2}n}{nt^{1/2}}$. Lei and Lin \cite{Lei_Lin_2022} analyzed a similar problem and proposed an algorithm similar to Algorithm~\ref{alg:Spektraaliklusterointi}. Specifically, their algorithm masks out the diagonal (which they call debiasing or bias adjusting) but it does not remove the nodes with too large $L_1$-norms. Their algorithm is guaranteed to be weakly consistent for $\rho\gg\frac{\log^{1/2}(n+t)}{nt^{1/2}}$. Ke, Shi and Xia \cite{Ke_Shi_Xia_2020} proved an analogous result for hypergraph degree-corrected block models. Su, Guo, Chang and Yang \cite{Su_Guo_Chang_Yang_2024} extended these results and ideas to directed multilayer networks. The proof requires a matrix concentration inequality based on classical matrix Bernstein's inequalities (see for example \cite{Tropp_2012}). However, matrix Bernstein's inequalities, which may be based on a deep theorem by Lieb depending on the proof, suffer from a logarithmic factor leading to the threshold $\frac{\log^{1/2}(n+t)}{nt^{1/2}}$ instead of $\frac{1}{nt^{1/2}}$.   

Lei, Zhang and Zhu \cite{Lei_Zhang_Zhu_2024} studied computational and information theoretical limits of the multilayer network model corresponding to a TBM with $\cY\in\set{0,1}^{n\times n\times t}$ with no cluster structure on the third mode. They showed that weak consistency is impossible for $\rho\ll \frac{1}{nt}$ and the maximum likelihood estimator is weakly consistent for $\rho\gg\frac{1}{nt}$. Furthermore, based on a conjecture, they showed that a polynomial-time algorithm cannot be weakly consistent for $\rho\le \frac{1}{2nt^{1/2}\log^{1.4} n}$. Kunisky \cite{Kunisky_2025} studied the impossibility of detecting a presence of a cluster structure in symmetric tensors with low coordinate degree functions (LCDF). LCDFs are arbitrary linear combinations of functions depending on at most $D$ entries of a vector, in our case the data tensor. Here the coordinate degree $D$ is argued to roughly correspond to algorithms requiring computation time $e^{\Theta(D)}$. Although the developed theory is more general, in our special case Theorem 1.13 states that the detection task is impossible if $\rho\le cn^{-\frac{d}{2}}D^{-\frac{d-2}{2}}$ for a sufficiently small constant $c$ and $D\le cn$. Since polynomial-time algorithms correspond to $D\gg\log n$, the threshold corresponds roughly to $\rho\le c n^{-\frac{d}{2}}\log^{-\frac{d-2}{2}}n$. Moreover, Kunisky showed that the detection task becomes easier when an aggregated tensor remains informative, each aggregated mode decreasing the density threshold by $n^{-\frac{1}{2}}$. This agrees with our Corollary \ref{thm:Weak consistency, aggregated data}.

Similar computational gaps have been established under sub-Gaussian noise. Han, Luo, Wang, and Zhang \cite{Han_Luo_Wang_Zhang_2022} analyzed multiway clustering on a TBM with sub-Gaussian noise. Their algorithm is initialized with a slightly refined version of a HOSVD algorithm, which involves calculating eigenvalue decompositions of $(\mat_k\cY)(\mat_k\cY)^\top$ without removing the diagonal, and without removing rows and columns with too large $L_1$-norms. 
They showed that the initialization algorithm is weakly consistent for $\Delta^2/\sig^2\gg n^{-d/2}$, where $\sig$ is
the maximum sub-Gaussian norm of the noise tensor entries,
and $\Delta \asymp \rho$ in our asymptotic setting.
By combining the initialization with an iterative Lloyd algorithm,
they obtain an algorithm that is strongly consistent if
$\rho^2/\sig^2 \ge C n^{-d/2}$
for a sufficiently large constant $C$.
Notice that this threshold is similar to Theorem \ref{thm:Weak consistency}, namely both $\sig^2$ and $\rho$ are upper bounds of the variance of the data entries, and setting $\sig^2=\rho$ yields the same weak consistency condition $\rho\gg n^{-d/2}$. However, sub-Gaussian analysis with Bernoulli distributed entries does not give $\sig^2\asymp\rho$ but
\[
\sig^2
= \frac{1-2\rho}{2\log\frac{1-\rho}{\rho}}
\sim\frac{1}{2\log\frac{1}{\rho}}
\]
as $\rho\to 0$ by Theorem 2.1 in \cite{Buldygin_Moskvichova_2013}. This implies a density condition of $\rho \gg n^{-d/4} \log^{1/2}n$. Nonetheless, it seems that the analogous statistical-computational gap $\frac{\log n}{n^{d-1}}\lesssim\rho\ll\frac{1}{n^{d/2}}$ is present in clustering a binary TBM.

\subsection{Random matrix norm bounds}

Proving matrix concentration inequalities for a sparse Bernoulli matrix is more involved than for a sub-Gaussian matrix. Kahn and Szemer\'{e}di \cite{Friedman_Kahn_Szemeredi_1989} bound the second largest eigenvalue of an adjacency matrix of a random regular graph (random graphs with constant degree) and Feige and Ofek \cite{Feige_Ofek_2005} repeat the same arguments for Erd\H{o}s--R\'enyi graphs (symmetric adjacency matrix with lower diagonal having independent and identically distributed entries). Lei and Rinaldo \cite{Lei_Rinaldo_2015} extend Feige's and Ofek's argument to bound the spectral norm of a centered symmetric binary matrices with independent but possibly not identically distributed entries. With this they are able to show consistency of a spectral clustering algorithm under a stochastic block model, which is similar to a TBM with $d=2$. Chien, Lin, and Wang \cite{Chien_Lin_Wang_2019} extend this further to a centered adjacency matrix of a hypergraph stochastic block model (similar to binary TBM) to show strong consistency of a spectral clustering algorithm. Zhang and Tan \cite{Zhang_Tan_2023} relax some parametric assumptions made in \cite{Vannieuwenhoven_Vandebril_Meerbergen_2012}. We also apply the proof techniques developed in \cite{Friedman_Kahn_Szemeredi_1989} and \cite{Feige_Ofek_2005}. 

In the context of clustering multilayer networks, Lei, Chen and Lynch \cite{Lei_Chen_Lynch_2020} adapt Feige's and Ofek's argument to tensor setting. Lei and Lin \cite{Lei_Lin_2022} consider alternative approach by applying matrix concentration inequality based on classical matrix Bernstein's inequalities (see for example \cite{Tropp_2012}). However, matrix Bernstein's inequalities, which may be based on a deep theorem by Lieb depending on the proof, suffer from a logarithmic factor making the matrix bounds slightly suboptimal.

In the case of sub-Gaussian TBM, Han, Luo, Wang, and Zhang \cite{Han_Luo_Wang_Zhang_2022} approach by analyzing concentration of a singular subspace of a possibly wide random matrix (i.e., a vector subspace spanned by the first $r$ right or left singular vectors corresponding to the $r$ largest singular values of a random matrix). For this purpose, the classical Davis--Kahan--Wedin theorem is insufficient as it provides a common error bound for both left and right singular subspaces. Hence, they rely on more sophisticated perturbation bounds developed by Cai and Zhang \cite{Cai_Zhang_2018} providing different bounds for left and right singular subspaces which is relevant for particularly tall and wide matrices. The probabilistic analysis is based on $\eps$-nets and Hanson--Wright inequality (see for example Chapter 6 in \cite{Vershynin_2018}).


\section{Proofs}\label{sec:Technical overview}

This section presents the proof of the main theorem (Theorem \ref{thm:Weak consistency}) in three subsections. Section \ref{subsec:Deterministic analysis} analyzes spectral clustering with deterministic error bounds, Section \ref{subsec:Sparse matrices} studies concentration of sparse matrices and Section \ref{subsec:Weak consistency} proves the main theorem. As the analysis of matrix concentration is quite involved, the detailed proofs of Section \ref{subsec:Sparse matrices} are postponed to the appendix.

\subsection{Deterministic analysis}\label{subsec:Deterministic analysis}

The following theorem states that a low-rank approximation of a data matrix allows efficient clustering when the noise level measured with spectral norm is sufficiently small.
The proof is similar to the proofs of Lemma 4
in \cite{Zhang_Tan_2023}, Theorem 3 in \cite{Han_Luo_Wang_Zhang_2022}, Theorem~2.2 in \cite{Jin_2015} and Lemma 5.3 in \cite{Lei_Rinaldo_2015}. However, the first three included this step into consistency theorems assuming statistical models not appropriate for our purposes. The second part of the proof of Theorem~\ref{thm:k-means} can be replaced by Lemma 5.3 in \cite{Lei_Rinaldo_2015} assuming equal dimensions $n=m$. For completeness, the full proof is presented below.

\begin{theorem}
\label{thm:k-means}
Consider a matrix $X \in \R^{n\times m}$ with
rows $X_{i:}=S_{z(i):}$, where $z \in [r]^n$, and $S \in \R^{r\times m}$ has minimum row separation
$\Delta = \min_{i \ne j} \norm{S_{i:}-S_{j:}} > 0$. Let $Y\in\R^{n\times m}$ be another matrix and $\hat{X}$ be its best rank-$r$ approximation. Let $\hat{z}$ be the output of a quasi-optimal $k$-means clustering algorithm applied to the rows of
$\hat{X}$, i.e., find a vector $\hat{z} \in [r]^n$
and centroids $\hat{\theta}_1,\dots,\hat{\theta}_{r}\in\R^{m}$ satisfying
\[
\sum_{j}\norm{\hat{X}_{j:}-\hat{\theta}_{\hat{z}_j}}^{2}
\wle Q \min_{\breve{z},\breve{\theta}}\sum_{j}\norm{\hat{X}_{j:}-\breve{\theta}_{\breve{z}_j}}^{2}
\]
with some relaxation parameter $Q>1$.
If 
$\min_l \abs{z^{-1}\set{l}}> 128Q\frac{r\spenorm{Y - X}^{2}}{\Delta^{2}}$,
then the misclassification rate
has an upper bound
\[
 \loss(z,\hat{z})
 \wle 128 Q \frac{r\spenorm{Y - X}^{2}}{n\Delta^{2}}.
\]
\end{theorem}

\begin{proof}
Define $E=Y-X$. Calculate a singular value decomposition $Y=\hat{U}\hat{\Sigma}\hat{V}^\top+\hat{U}_{\perp}\hat{\Sigma}_{\perp}\hat{V}_{\perp}^\top$ and define $\hat{X}=\hat{U}\hat{\Sigma}\hat{V}^\top$, which is the best rank-$r$ approximation of $Y$ (in statistical applications $\hat{X}$ is used to estimate $X$). Since the rank of $\hat{X}$ is at most $r$ and the rank of $X$ is at most $r$, the rank of their difference $\hat{X}-X$ is at most $2r$ (every vector $u$ of the image of $\hat{X}-X$ can be written as $u=u_1+u_2$, where $u_1\in\operatorname{im}\hat{X}$ is in an $r$-dimensional subspace and $u_2\in\operatorname{im}X$ is in another $r$-dimensional subspace). By Corollary 2.4.3 in \cite{Golub_VanLoan_2013}, the Frobenius norm and the spectral norm of the rank-$2r$ matrix $\hat{X}-X$ are comparable according to $\fronorm{\hat{X}-X}^{2} \le 2r \spenorm{\hat{X}-X}^{2}$. Now by the Eckart--Young--Mirsky theorem \cite{Eckart_Young_1936} and Weyl's inequality \cite{Weyl_1912}, we have
\[
\begin{split}
	\fronorm{\hat{X}-X}^{2}
	&\le 2r\spenorm{\hat{X}-X}^{2} 
	\le 2r (\spenorm{\hat{X}-Y}+\spenorm{E})^{2} 
	\overset{\rm EYM}{\le} 2r (\sig_{r+1}(Y)+\spenorm{E})^{2} \\
	&\overset{\rm Weyl}{\le} 2r(\sig_{r+1}(X)+\spenorm{E}+\spenorm{E})^{2} 
	\le 8r\spenorm{E}^{2}. \\
\end{split}
\]
This shows that the estimate $\hat{X}$ is close to $X$ in Frobenius norm. Denote the estimated clusters by $\hat{z}_{j}\in[r]$, $j=1,\dots,n$ and their centers by $\hat{\theta}_{k}\in\R^{m}$, $k=1,\dots,r$. By quasi-optimality, we have
\[
\begin{split}
	\sum_{j}\norm{\hat{X}_{j:}-\hat{\theta}_{\hat{z}_j}}^{2}
	&\le Q\min_{\breve{z},\breve{\theta}}\sum_{j}\norm{\hat{X}_{j:}-\breve{\theta}_{\breve{z}_j}}^{2} 
	\le Q\sum_{j}\norm{\hat{X}_{j:}-X_{j:}}^{2} 
	= Q\fronorm{\hat{X}-X}^{2} 
	\le 8Qr\spenorm{E}^{2}.
\end{split}
\]
This shows that the row vectors of the estimate $\hat{X}$ are close to the estimated cluster means on average. Now we can estimate the distances between the true cluster means (rows of $S$) and the estimated cluster means ($\hat{\theta}_{\hat{z}_j}$):
\[
\begin{split}
	\sum_{j}\norm{X_{j:}-\hat{\theta}_{\hat{z}_j}}^{2}
	&\le 2\sum_{j}\norm{X_{j:}-\hat{X}_{j:}}^{2}+\norm{\hat{X}_{j:}-\hat{\theta}_{\hat{z}_j}}^{2}  
	\le 32Qr\spenorm{E}^{2}. \\
\end{split}
\]
For any $t>0$, Markov's inequality with respect to the counting measure gives
\[
\begin{split}
	\abs{\set{j\mid\norm{\hat{\theta}_{\hat{z}_j} - X_{j:}}^{2}\ge t^{2}}}
	&\le\frac{\sum_j\norm{\hat{\theta}_{\hat{z}_j} - X_{j:}}^{2}}{t^2} 
	\le\frac{32Qr\spenorm{E}^2}{t^2}.
\end{split}
\]
Consider the set $A=\set{j\in[n]\mid\norm{\hat{\theta}_{\hat{z}_j} - X_{j:}}<t}$. If $z_i\ne z_j$ for $i,j\in A$, then
\[
\begin{split}
	\norm{\hat{\theta}_{\hat{z}_i} - S_{z_j:}}
	&\ge \norm{S_{z_i:} - S_{z_j:}} - \norm{S_{z_i:} - \hat{\theta}_{\hat{z}_i}}
	>\Delta - t.
\end{split}
\]
If $z_i= z_j$ for some $i,j\in A$, then $\norm{\hat{\theta}_{\hat{z}_i} - S_{z_j:}} < t$. Consider $t$ satisfying $\Delta-t=t$, or equivalently, $t=\Delta/2$. The previous calculation shows that for every $i,j\in A$, $z_i=z_j$ if and only if $\norm{\hat{\theta}_{\hat{z}_i} - S_{z_j:}} < \Delta/2$. This gives a well-defined mapping $\pi\colon\hat{z}(A)\to[r]$, $\pi(\hat{z}_i)=z_i$ (if $\hat{z}_i=\hat{z}_j$, then $\norm{\hat{\theta}_{\hat{z}_i} - S_{z_j:}} < \Delta/2$ and hence $z_i=z_j$).

Next, let us show that $\pi$ is a bijection. Since $\abs{A^c}\leq\frac{32Qr\spenorm{E}^2}{t^2}< \min_l\abs{z^{-1}\set{l}}$, where the latter inequality holds by assumption, the intersection $A\cap z^{-1}\set{l}$ must be nonempty for all $l\in[r]$. Therefore, for every $l\in[r]$, there exists $i\in A$ such that $\pi(\hat{z}_i)=z_i=l$. This shows that $\pi\colon[r]\to[r]$ is a surjection, and hence a bijection.
The bound for the misclassification rate follows now from the estimate $\loss(z,\hat{z})\le \abs{A^{c}}/n$.
\end{proof}

As Theorem \ref{thm:k-means} shows, clustering from a low-rank approximation works well when the spectral norm of the noise matrix is sufficiently small and the clusters are sufficiently separated in the signal. The following result (Lemma~\ref{lem:separation of clusters}) addresses the latter issue by estimating the separation of clusters $\Delta$ in Theorem \ref{thm:k-means} under assumptions \eqref{eq:ClusterSeparation} and \eqref{eq:ClusterBalance}. The former issue is addressed in Section~\ref{subsec:Sparse matrices}.

\begin{lemma}
\label{lem:separation of clusters}
Let $X = \E Y$ where $Y = \mat_k \cY$ is
the mode-$k$ matricization of the data tensor
$\cY$ sampled from $\TBM(\rho, \cS, z_1,\dots,z_d)$. Then 
\[
 \min_{i,j \in [n_k] : z_k(i) \ne z_k(j)} \norm{(XX^\top)_{i:} - (XX^\top)_{j:}}
 \wge
 \frac{\prod_{k'}\al_{k'}}{\sqrt{2\al_k}} \frac{\de_k^2}{\sqrt{n_k r_k}} n_1 \cdots n_d \rho^2,
\]
where
$\de_k$ is the mode-$k$ cluster separation defined by \eqref{eq:ClusterSeparation},
and
$\alpha_k$ is the mode-$k$ cluster balance coefficient defined by \eqref{eq:ClusterBalance}.
\end{lemma}

\begin{proof}
Fix $k\in[d]$ and let $i_k,i_k'\in [n_k]$ be entities from different clusters
$z_k(i_k)\ne z_k(i_k')$. Denote $l_{-k}=(l_1,\dots,l_{k-1},l_{k+1},\dots,l_d)\in[r_1\dots r_d/r_k]$ and $j_{-k}=(j_1,\dots,j_{k-1},j_{k+1},\dots,j_d)\in[n_1\cdots n_d/n_k]$.
Define a diagonal matrix $D\in\R^{\frac{r_1\dots r_d}{r_k}\times\frac{r_1\dots r_d}{r_k}}$ with nonnegative diagonal entries
$D_{l_{-k}l_{-k}}
= \frac{\abs{z_1^{-1}\set{l_1}}\cdots\abs{z_{d}^{-1}\set{l_{d}}}}{\abs{z_k^{-1}\set{l_k}}}$.
Now
\[
\begin{split}
 &\norm{((\mat_k\cX)(\mat_k\cX)^\top)_{i_k:} - ((\mat_k\cX)(\mat_k\cX)^\top)_{i_k':}}^2 \\
 &= \sum_{j_k}\left(\sum_{j_{-k}}((\mat_k\cX)_{i_kj_{-k}} - (\mat_k\cX)_{i_k'j_{-k}}))(\mat_k\cX)_{j_kj_{-k}}\right)^2 \\
 &= \rho^4\sum_{l_k}\abs{z_k^{-1}\set{l_k}}
 \Bigg(\sum_{l_{-k}}\abs{z_1^{-1}\set{l_1}}\dots\abs{z_{k-1}^{-1}\set{l_{k-1}}}\abs{z_{k+1}^{-1}\set{l_{k+1}}}\dots\abs{z_d^{-1}\set{l_d}} \\
 &\quad \times((\mat_k\cS)_{z_k(i_k),l_{-k}} - (\mat_k\cS)_{z_k(i_k'),l_{-k}})(\mat_k\cS)_{l_k,l_{-k}}\Bigg)^2.
\end{split}
\]
By applying the cluster balance condition \eqref{eq:ClusterBalance},
the last term is bounded from below by
\[
\begin{split}
 &\ge \frac{\rho^4\alpha_k n_k}{r_k}\sum_{l_k}\Bigg(\sum_{l_{-k}}D_{l_{-k}l_{-k}}((\mat_k\cS)_{z_k(i_k),l_{-k}} - (\mat_k\cS)_{z_k(i_k'),l_{-k}})(\mat_k\cS)_{l_k,l_{-k}}\Bigg)^2 \\
 &= \frac{\rho^4\alpha_k n_k}{r_k}\norm{((\mat_k\cS)D(\mat_k\cS)^\top)_{z_k(i_k):} - ((\mat_k\cS)D(\mat_k\cS)^\top)_{z_k(i_k'):}}^2.
\end{split}
\]	
Next, we observe that for any matrix $A$ and indices $i\ne i'$,
the Cauchy--Schwarz inequality gives
\[
\begin{split}
	\norm{A_{i:} - A_{i':}}^2
	&= \norm{(e_i - e_{i'})^\top A}^2
	= (e_i - e_{i'})^\top AA^\top (e_i - e_{i'}) \\
	&\le \sqrt{2}\norm{(e_i - e_{i'})^\top AA^\top }
	= \sqrt{2}\norm{(AA^\top )_{i:} - (AA^\top )_{i':}}.
\end{split}
\]
Applying this inequality to matrix $(\mat_k\cS)\sqrt{D}$ now implies a lower bound
\[
\begin{split}
 &\norm{((\mat_k\cX)(\mat_k\cX)^\top )_{i_k:} - ((\mat_k\cX)(\mat_k\cX)^\top )_{i_k':}}^2 \\
 &\wge \frac{\rho^4\alpha_k n_k}{2r_k}
 \norm{((\mat_k\cS)\sqrt{D})_{z_k(i_k):} - ((\mat_k\cS)\sqrt{D})_{z_k(i_k'):}}^4 \\
 &\wge \frac{\rho^4\alpha_k n_k}{2r_k}\left(\prod_{k'\ne k}\frac{\alpha_{k'} n_{k'}}{r_{k'}}\right)^2\norm{(\mat_k\cS)_{z_k(i_k):} - (\mat_k\cS)_{z_k(i_k'):}}^4.
\end{split}
\]
By the cluster separation assumption \eqref{eq:ClusterSeparation}, we obtain
\[
 \norm{((\mat_k\cX)(\mat_k\cX)^\top )_{i_k:} - ((\mat_k\cX)(\mat_k\cX)^\top )_{i_k':}}
 \wge \frac{\prod_{k'}\al_{k'}}{\sqrt{2\al_k}} \frac{\de_k^2}{\sqrt{n_kr_k}}n_1\cdots n_d\rho^2.
\]
\end{proof}

\subsection{Sub-Poisson random matrices}
\label{subsec:Sparse matrices}

Theorem \ref{the:TrimmedMatrixConcentration} bounds the spectral norm of a random matrix, from which rows and columns with too large $L_1$-norms are removed. This result along with its proof is a mild generalization of Theorem 1.2 in \cite{Feige_Ofek_2005}.
Namely, we relax the assumption of binary entries to sub-Poisson entries.
In a similar spirit, Theorem \ref{thm:hollow gram matrix bound} analyzes concentration of a trimmed product $XX^\top $, where the entries of the random matrix $X\in \Z^{n\times m}$ are independent and centered. The focus is only on the off-diagonal entries, because the diagonal and off-diagonal entries concentrate at different rates. The proof is based on decoupling and an observation that the entries of $XX^\top $ are nearly independent, eventually allowing us to apply Theorem \ref{the:TrimmedMatrixConcentration}. This argument, however, relies on having a sufficiently sparse matrix $X$. Finally, Lemma \ref{lem:hollow gram matrix row L1 bound} asserts that only a small fraction of rows and columns are removed.

\begin{theorem}
\label{the:TrimmedMatrixConcentration}
Let $X \in \R^{n\times m}$ be a random matrix with
independent centered 
sub-Poisson entries with variance proxy $\sig^2$.
Define a trimmed matrix $X' \in \R^{n\times m}$ by
\begin{equation}
	\label{eq:TrimmedMatrix}
	X_{ij}'
	\weq
	\begin{cases}
		X_{ij}, & \norm{X_{i:}}_1\vee\norm{X_{:j}}_1\le \Ctrim(n\vee m) \sig^2, \\
		0, & \text{otherwise}.
	\end{cases}
\end{equation}
Then
\[
\pr \left( \spenorm{X'}\ge 9 ( \Ctrim + 66 t) \sig\sqrt{n\vee m} \right)
\wle 2 \left( e(n\vee m) \right)^{-2t}
\qquad \text{for all $t\ge 1$}.
\]
\end{theorem}

\begin{proof}
See Appendix \ref{app:the:TrimmedMatrixConcentration}.
\end{proof}

Given a random matrix $X \in \Z^{n \times m}$, we show that a trimmed version of the Gram matrix $X X^\top$ concentrates around a similarly trimmed version of $(\E X)(\E X)^\top$.
Below we recall that $\abs{X}$ denotes the matrix of entrywise absolute values of $X$.

\begin{theorem}
\label{thm:hollow gram matrix bound}
Let $X \in \Z^{n\times m}$ be a random matrix with independent sub-Poisson entries
with variance proxy $\sig^2$, and such that
$\E\abs{X_{ij}}\le \sig^2$.
Let $\Ctrim \ge 0$
and define
indicator matrices $M,N \in \{0,1\}^{n \times n}$ by
$M_{ij} = \indic(i \ne j)$ and
\[
 N_{ij}
 \weq \indic \bigg( \onenorm{ (\abs{X} \abs{X}^\top) \odot M)_{i:} }
 \vee \onenorm{ (\abs{X} \abs{X}^\top) \odot M)_{:j} }
 \le \Ctrim n m \sig^4 \bigg).
\]
There exists an absolute constant $C$ such that
\[
 \Pr\left( \spenorm{((XX^\top )\odot M - \E X\E X^\top ) \odot N}
    \ge C(t+\Ctrim)\sqrt{nm}\sig^2 + m\sig^4
 \right)
 \wle C n^{-1} e^{-t^{1/3}}
\]
whenever
$8\log en\le m\sig^2$,
$6e^2n\sig^2\le 1$,
and
$t \ge C \left( 1\vee\frac{\log^3 m}{\log^3(1/(6en\sig^2))} \right)$.
\end{theorem}
\begin{proof}
See Appendix~\ref{sec:hollow gram matrix bound}.
\end{proof}

In particular, if $n\sig^2\lesssim m^{-\eps}$ for some arbitrarily small constant $\eps>0$, then the bound
$\spenorm{((XX^\top )\odot M - \E X\E X^\top )\odot N}\lesssim \sqrt{nm}\sig^2 + m\sig^4$ holds with high probability as $n\to\infty$, and if $n\sig^2\lesssim 1$, then $\spenorm{((XX^\top )\odot M - \E X\E X^\top )\odot N} \lesssim \log^3(m)\sqrt{nm}\sig^2 + m\sig^4$ with high probability as $n\to\infty$.

The next result guarantees that
only a small fraction of rows and columns is masked out.

\begin{lemma}
\label{lem:hollow gram matrix row L1 bound}
Let $X\in \Z^{n\times m}$, $n\le m$, be a random matrix with
independent sub-Poisson entries with variance proxy $\sig^2$,
and
$\E\abs{X_{ij}}\le \sig^2$.
Denote $M_{ij} = \indic(i \ne j)$, and let
\[
 \xi_i
 \weq \indic \bigg( \norm{ (\abs{X} \abs{X}^\top) \odot M)_{i:} }_1
 \le (4+5t) nm \sig^4 \bigg).
\]
If $8\log en\le m\sig^2$ and $6e^2n\sig^2\le 1$, then
\[
 \Pr\left( \frac{1}{n}\sum_{i=1}^{n} \xi_i \le 1-s \right)
 \wle \frac{1}{s} \left( \sqrt{t} \exp \left(-\frac{\sqrt{t}nm\sig^4}{10\sqrt{6}}\right)
  + 3n^{-\sqrt{t}/2\sqrt{6}}\right)
\]
for all $s>0$ and
$t\geq 6\vee \frac{96\log^2 m}{\log^2(1/6en\sig^2)}$.
\end{lemma}

\begin{proof}
See Appendix~\ref{sec:hollow gram matrix row L1 bound}.
\end{proof}

In particular, if $nm\sig^4\to\infty$ and $n\to\infty$, then one can choose $s\to0$ to approach zero sufficiently slowly to obtain a high-probability event.

\subsection{Proof of Theorem~\ref{thm:Weak consistency}}
\label{subsec:Weak consistency}
Let $Y = \mat_k \cY$ be the mode-$k$ matricization of the data tensor,
and denote its width by $m_k = n_1 \cdots n_d / n_k$.
Observe first that $n_k \gg 1$ due to $n_k \gg r_k^{1/3}$.
The scaling assumptions~\eqref{eq:weak consistency regime} then imply that
$m_k \gg m_k \rho \gg \log n_k \gg 1$ and $n_k \rho \ll 1$.
Therefore, the inequalities
$8 \log e n_k \le m_k \rho$ and
$n_k \rho \le 6^{-1} e^{-2}$ needed in
Theorem~\ref{thm:hollow gram matrix bound} and
Lemma~\ref{lem:hollow gram matrix row L1 bound} hold eventually.
Algorithm~\ref{alg:Spektraaliklusterointi} starts by
computing a trimmed Gram matrix
\[
 A
 \weq Y Y^\top \odot M \odot N,
\]
using binary matrices $M,N \in \R^{n_k\times n_k}$, where
$M_{ij}=\I\set{i\ne j}$,
and $N = \xi \xi^\top$ is
determined by the indicator vector $\xi \in \R^{n_k}$ with entries
\[
 \xi_i
 \weq \I\set*{ \norm{ (\abs{Y} \abs{Y}^\top) \odot M)_{i:} }_1 \le \Ctrim \rho^2 n_k m_k}.
\]
Lemma~\ref{lem:hollow gram matrix row L1 bound} ensures that 
$N$ masks out only a small fraction of the Gram matrix. 
In detail, the needed sub-Poisson variance proxy is $2\rho$ since the lemma also requires that $\E\abs{Y_{ij}}\leq\E\abs{Y_{ij}-\E Y_{ij}}+\abs{\E Y_{ij}}\leq 2\rho$. Consequently, the trimming constant is $\Ctrim/4$. Hence, there exists an absolute constant $C_1$ such that 
setting $\Ctrim\ge C_1(1\vee \frac{\log^2 m_k}{\log^2(1/12en_k\rho)})$ gives
\[
\begin{split}
\Pr\left(\frac{1}{n_k}\norm{1-\xi}_1 \ge s\right)
&\wle \frac{1}{s}\left(\sqrt{\Ctrim}e^{-n_1\cdots n_d\rho^2} + 3n_k^{-1}\right).
\end{split}
\]
Notice that eventually $12en_k\rho\leq m_k^{-\varepsilon}$ and hence $\frac{\log m_k}{\log(1/12en_k\rho)}\leq\varepsilon^{-1}$. Here we may choose $e^{-n_1\cdots n_d\rho^2}+n_k^{-1}\ll s\ll r_k^{-3}$ since $r_k^{3}e^{-n_1\cdots n_d\rho^2}\leq\frac{r_k^3}{n_1\cdots n_d\rho^2}\ll 1$ and $\frac{r_k^{3}}{n_k}\ll 1$. The lower bound of $s$ ensures a high-probability event and the upper bound will be needed later.
The number of zeroed entries is
\[
 \fronorm{1-N}^2
 \weq n_k^2 - \norm{\xi}_1^2
 \wle n_k^2 - ((1-s)n_k)^2
 \wle 2sn_k^2.
\]
Denote $X = \E Y$.
The resulting  signal loss is bounded from above by
\begin{align}
 \label{eq:signal loss from trimming}
 &\spenorm{(X X^\top) \odot (1-N)} 
 \le \fronorm{(X X^\top) \odot (1-N)} \nonumber \\
 &\le \max_{i,j}\abs{(XX^\top )_{ij}} \fronorm{1-N}
 \leq \sqrt{2s}n_1\cdots n_d\rho^{2}.
\end{align}
By combining inequality \eqref{eq:signal loss from trimming} with Theorem \ref{thm:hollow gram matrix bound} with sub-Poisson variance proxy $2\rho$ and trimming constant $\Ctrim/4$, there exists an absolute constant $C_2$ such that
\[
 \begin{split}
 &\Pr\left(\spenorm{(Y Y^\top \odot M \odot N - XX^\top } \ge u\right) \\
 &\leq C_2n_k^{-1}e^{-t^{1/3}}
 \quad \text{for all }t\geq C_2\left(1\vee \frac{\log^3 m_k}{\log^3(1/12en_k\rho)}\right),
 \end{split}
\]
where
\[
u
= \sqrt{2s}n_1\cdots n_d\rho^{2} + C_2(t+\Ctrim)\sqrt{n_1\cdots n_d}\rho + \frac{4n_1\cdots n_d\rho^2}{n_k}.
\]
Notice that eventually $12en_k\rho\leq m_k^{-\varepsilon}$ and hence $\frac{\log m_k}{\log(1/12en_k\rho)}\leq\varepsilon^{-1}$. Set $t=C_2(1\vee \eps^{-3})$ to be a constant. 

Next, define 
\[
\Delta
= \min_{i_k,i_k':z_k(i)\neq z_k(i')}\norm{(XX^\top )_{i_k:} - (XX^\top )_{i_k':}}.
\]
After trimming, Algorithm~\ref{alg:Spektraaliklusterointi} proceeds by computing a best rank-$r_k$ approximation $\hat{U}\hat{\Lambda}\hat{U}^\top$ of $YY^\top \odot M\odot N$. Notice that Algorithm \ref{alg:Spektraaliklusterointi} applies $k$-means clustering to the rows of $\hat{U}\hat{\Lambda}$ instead of the rows of $\hat{U}\hat{\Lambda}\hat{U}^\top$ as in Theorem \ref{thm:k-means}. However, given the centroids $\hat{\theta}_1,\dots,\hat{\theta}_{r_k}\in\R^{r_k}$ (treated as row vectors) and the membership vector $\hat{z}_k\in[r_k]^{n_k}$ computed in Algorithm \ref{alg:Spektraaliklusterointi}, the transformed centroids $\hat{\theta}_{1}\hat{U}^\top ,\dots,\hat{\theta}_{r_k}\hat{U}^\top \in\R^{n_k}$ and the same membership vector $\hat{z}_k\in[r_k]^{n_k}$ solve the $k$-means problem quasi-optimally as assumed in Theorem \ref{thm:k-means}. To prove this, first observe that
\[
\begin{split}
&\sum_{j}\norm{(\hat{U}\hat{\Lambda}\hat{U}^\top)_{j:} - \hat{\theta}_{\hat{z}_k(j)}\hat{U}^\top }^{2}
= \sum_{j}\norm{((\hat{U}\hat{\Lambda})_{j:} - \hat{\theta}_{\hat{z}_k(j)})\hat{U}^\top }^{2}
= \sum_{j}\norm{(\hat{U}\hat{\Lambda})_{j:} - \hat{\theta}_{\hat{z}_k(j)}}^{2} \\
&\leq Q\min_{\breve{z},\breve{\theta}}\sum_{j}\norm{(\hat{U}\hat{\Lambda})_{j:} - \breve{\theta}_{\breve{z}(j)}}^{2}
= Q\min_{\breve{z},\breve{\theta}}\sum_{j}\norm{((\hat{U}\hat{\Lambda})_{j:} - \breve{\theta}_{\breve{z}(j)})\hat{U}^\top}^{2}. 
\end{split}
\]
By minimizing over centroids of the form $\breve{\theta}_{j}\hat{U}$ with $\breve{\theta}_j\in\R^{n_k}$ instead of $\breve{\theta}_j\in\R^{r_k}$, we get a further upper bound
\[
\begin{split}
 &Q\min_{\breve{z},\breve{\theta}}\sum_{j}\norm{((\hat{U}\hat{\Lambda}\hat{U}^\top)_{j:} - \breve{\theta}_{\breve{z}(j)})\hat{U}\hat{U}^\top}^{2} 
 \wle Q\min_{\breve{z},\breve{\theta}}\sum_{j}\norm{(\hat{U}\hat{\Lambda}\hat{U}^\top)_{j:} - \breve{\theta}_{\breve{z}(j)}}^{2}.
\end{split}
\]
This justifies the applicability of Theorem~\ref{thm:k-means} by which there exists an absolute constant $C_3$ satisfying
\[
\begin{split}
&\Pr\left(\loss(z_k,\hat{z}_k)\geq C_3Q\frac{r_ku^2}{n_k\Delta^2}\right) \\
&\leq \Pr\left(\frac{1}{n_k}\sum_{i=1}^{n_k}\I\set*{\sum_{j',l:j'\neq i}\abs{Y_{il}Y_{j'l}} \ge \Ctrim n_1\cdots n_d\rho^2}\geq s\right) \\&\quad + \Pr\left(\spenorm{(YY^\top )\odot M\odot N - XX^\top }\geq u\right) \\
&\leq \frac{1}{s}\left(\sqrt{\Ctrim}e^{-n_1\cdots n_d\rho^2} + \frac{3}{n_k}\right) + \frac{C_2}{n_k}
\end{split}
\]
when $\min_l\abs{z_k^{-1}\set{l}}\geq C_3 Q \frac{r_ku^{2}}{\Delta^2}$. By Lemma \ref{lem:separation of clusters}, there exists a constant $C_{4}$ such that
$\Delta \ge C_4\frac{\prod_{k'}\al_{k'}}{\sqrt{\al_k}} \frac{\de_k^2}{\sqrt{n_k r_k}} n_1 \cdots n_d \rho^2$. Now
\[
\begin{split}
 \frac{\sqrt{r_k}u}{\sqrt{n_k}\Delta}
 &\le \frac{\sqrt{2s}n_1\cdots n_d\rho^{2} + C_2(t+\Ctrim)\sqrt{n_1\cdots n_d}\rho
 + \frac{4n_1\cdots n_d\rho^2}{n_k}}{C_4\frac{\prod_{k'}\al_{k'}}{\sqrt{\al_k}} \frac{\de_k^2}{r_k} n_1 \cdots n_d \rho^2} \\
 &= \frac{r_k\sqrt{\al_k}}{C_4\de_k^2\prod_{k'}\al_{k'}}\left(\sqrt{2s} + \frac{C_2(t+\Ctrim)}{\sqrt{n_1\cdots n_d}\rho} + \frac{4}{n_k}\right).
\end{split}
\]
Now $\min_l\abs{z_k^{-1}\set{l}}\geq C_3Q\frac{r_ku^{2}}{\Delta^2}$ holds eventually by assumptions $s\ll r_k^{-3}$, $\rho\gg r_k^{3/2}(n_1\cdots n_d)^{-1/2}$ and $n_k\gg r_k^{3}\geq r_k^{3/2}$, namely
\[
\begin{split}
\frac{C_3Q\frac{r_ku^{2}}{\Delta^2}}{\min_l\abs{z_k^{-1}\set{l}}}
&\leq \frac{C_3Qn_k\left(\frac{r_k\sqrt{\al_k}}{C_4\de_k^2\prod_{k'}\al_{k'}}\left(\sqrt{2s} + \frac{C_2(t+\Ctrim)}{\sqrt{n_1\cdots n_d}\rho} + \frac{4}{n_k}\right)\right)^2}{\frac{\alpha_k n_k}{r_k}} \\
&= \frac{C_3Qr_k^3}{C_4^2\de_k^4\prod_{k'}\al_{k'}^2}\left(\sqrt{2s} + \frac{C_2(t+\Ctrim)}{\sqrt{n_1\cdots n_d}\rho} + \frac{4}{n_k}\right)^2
\ll 1.
\end{split}
\]
Recall that $e^{-n_1\cdots n_d\rho^2}+n_k^{-1}\ll s\ll r_k^{-1}$ and $t=C_2(1\vee \eps^{-3})$. Now we have
\[
\begin{split}
&\Pr\Bigg(\loss(z_k,\hat{z}_k)\geq \underbrace{C_3Q\left(\frac{r_k\sqrt{\al_k}}{C_4\de_k^2\prod_{k'}\al_{k'}}\left(\sqrt{2s} + \frac{C_2(C_2(1\vee \eps^{-3})+\Ctrim)}{\sqrt{n_1\cdots n_d}\rho} + \frac{4}{n_k}\right)\right)^2}_{\ll 1}\Bigg) \\
&\leq \underbrace{\frac{1}{s}\left(\sqrt{\Ctrim}e^{-n_1\cdots n_d\rho^2} + \frac{3}{n_k}\right) + \frac{C_2}{n_k}}_{\ll 1},
\end{split}
\]
that is, $\loss(z_k,\hat{z}_k)\to 0$ in probability. This concludes the proof of Theorem \ref{thm:Weak consistency}. \qed

\ifarxiv
\addcontentsline{toc}{section}{References}
\bibliographystyle{plain}
\fi

\bibliography{lslReferences-2025-12-04}

\clearpage

\begin{appendix}

\section{Preliminaries}\label{app:Preliminaries}

A random variable $X$ and its probability distribution is called \emph{upper sub-Poisson} if there exists $\sigma^2 \ge 0$ such that
\[
\E e^{\lambda (X - \E X)} \le e^{\sigma^2(e^\lambda - 1 - \lambda)}
\qquad \text{for all $\la \ge 0$}.
\]
Such a number $\sig^2$ is called an upper sub-Poisson variance proxy of $X$. The random variable $X$ is called
\emph{lower sub-Poisson} if $-X$ is upper sub-Poisson, and \emph{sub-Poisson} if it is both upper and lower sub-Poisson.
The following result summarizes key tail bounds for upper sub-Poisson random variables \cite{Leskela_Valimaa_2025}.
Inequality \eqref{eq:Bennett} is known as Bennett's inequality and inequalities \eqref{eq:Bernstein1} and \eqref{eq:Bernstein2} are known as Bernstein's inequalities.

\begin{lemma}
\label{lem:Chernoff}
If a random variable $X$ has an upper sub-Poisson variance proxy $\sigma^2$,
then for all $t\ge 0$,
\begin{align}
 \Pr\left(X\ge t\right)
 &\wle e^{-\sig^2}\left(\frac{e\sig^2}{\sig^2 + t}\right)^{\sig^2 + t} \label{eq:Bennett} \\
 &\wle \exp\left(-\frac{t^2/2}{\sig^2+t/3}\right) \label{eq:Bernstein1}\\
 &\wle \exp\left(-\left(\frac{t^2}{4\sig^2}\wedge\frac{3t}{4}\right)\right). \label{eq:Bernstein2}
\end{align}
\end{lemma}

The following result summarizes key properties of the function
\begin{equation}
\label{eq:BennettBound}
\beta(t, \sig^2)
\weq e^{-\sig^2} \left( \frac{e \sig^2}{\sig^2 + t} \right)^{\sig^2 + t}.
\end{equation}
appearing in Bennett's inequality \eqref{eq:Bennett}.

\begin{lemma}
\label{the:BennettBijection}
For any $\sig^2 > 0$, the function $t \mapsto \beta(t,\sig^2)$ defined by \eqref{eq:BennettBound}
is a strictly decreasing bijection from $[0,\infty)$ into $(0,1]$, and bounded by
$\beta(t, \sig^2) \le \left( \frac{e \sig^2}{t} \right)^t$ for all $t > 0$.
\end{lemma}
\begin{proof}
We note that $\beta(t, \sig^2) = e^{\sig^2 \ga(t/\sig^2)}$ where
$\ga(s) = s - (1+s)\log(1+s)$.
Because $\ga'(s) = - \log(1+s)$, we find that
$\frac{d}{dt} \beta(t,\sig^2) = \beta(t,\sig^2) \ga'(t/\sig^2) = - \beta(t,\sig^2) \log(1+t/\sig^2)$.
Because $\beta(t,\sig^2) > 0$, it follows that $\frac{d}{dt} \beta(t,\sig^2) < 0$ for all $t$.

To prove the upper bound, a simple computation yields the identity
\[
\log \beta(t,\sig^2)
\weq t - (\sig^2+t) \log(1+t/\sig^2),
\]
from which we find that
$\log \beta(t,\sig^2) \le t - t \log(1+t/\sig^2) = t \log \frac{e}{1+t/\sig^2} \leq \log \left( \frac{e \sig^2}{t} \right)^t$.
\end{proof}

The following result is used in obtaining a cleaner version of a Bernstein's inequality in Lemma~\ref{lem:row L1 bound}.

\begin{lemma}
\label{lem:inverse of minimum}
If $f,g:\R_{\ge 0}\to\R_{\ge 0}$ are strictly increasing  bijections, then $f\wedge g$ is a strictly increasing bijection with an inverse given by $(f\wedge g)^{-1} = f^{-1}\vee g^{-1}$.
\end{lemma}

\begin{proof}
If $0\le x_1<x_2$, then $f(x_1)<f(x_2)$ and $g(x_1)<g(x_2)$ giving $f(x_1)\wedge g(x_1)\le f(x_1)<f(x_2)$ and $f(x_1)\wedge g(x_1)\le g(x_1)<g(x_2)$, that is, $f(x_1)\wedge g(x_1)<f(x_2)\wedge g(x_2)$ and $f\wedge g$ is strictly increasing. 

Suppose $f(x)\wedge g(x) = y$. Then either $f(x)=y$ or $g(x)=y$ and hence $f^{-1}(y)=x$ or $g^{-1}(y)=x$. If $f(x)\ge g(x) = y$, then $f^{-1}(y)\le x = g^{-1}(y)$ (since $f$ is increasing) and hence $x = f^{-1}(y)\vee g^{-1}(y)$. Similarly, if $y = f(x)\le g(x)$, then $g^{-1}(y)\le x = f^{-1}(y)$ (since $g$ is increasing) and hence $x = f^{-1}(y)\vee g^{-1}(y)$. This shows that $x = f^{-1}(y)\vee g^{-1}(y)$.
\end{proof}

\section{Proof of Theorem \ref{the:TrimmedMatrixConcentration}}\label{app:the:TrimmedMatrixConcentration}

The proofs mainly follow \cite{Feige_Ofek_2005}, which focuses on Bernoulli distributed entries (random graphs). The proof technique originates from \cite{Friedman_Kahn_Szemeredi_1989}, where an analogous result is given for random regular graphs.

The following lemma in its original form \cite{Feige_Ofek_2005} bounds sums of the form $\sum_{i\in S}\sum_{j\in T}\abs{X_{ij}}$, where $X\in\R^{n\times m}$ is a random matrix and $S\subset[n]$ and $T\subset[m]$. When $X$ corresponds to an adjacency matrix of a graph, then this sum counts the number of (directed) edges from the vertex set $S$ to the vertex set $T$. However, here the statement is modified to draw a clearer connection to the spectral norm $\spenorm{X}=\sup_{\norm{u},\norm{v}\le 1}\sum_{i,j}u_i X_{ij} v_j$.

\begin{lemma}
\label{the:SubmatrixSums}
For any random matrix $X\in\R^{n\times m}$ with independent centered entries satisfying
$\E e^{\la X_{ij}} \le e^{\sig^2(e^{\abs{\la}}-1-\abs{\la})}$ for all $\la \in \R$, and any $t \ge 1$,
\[
\pr\bigg( \bigcup_{S,T} \Big\{ - \log \beta(X_{S,T}, \abs{S} \abs{T} \sig^2) \ge (1+t) \tau_{\abs{S},\abs{T}} \Big\} \bigg)
\wle (e^2nm)^{-t},
\]
where the union is taken over nonempty sets $S \subset [n]$ and $T \subset [m]$,
\[
X_{S,T}
\weq \max_{u \in \{\pm 1\}^S} \max_{v \in \{\pm 1\}^T} \sum_{i \in S} \sum_{j \in T} u_i X_{ij} v_j,
\qquad
\tau_{k,\ell}
\weq k \log \frac{2en}{k} + \ell \log \frac{2em}{\ell},
\]
and the function $\beta$ is defined by \eqref{eq:BennettBound}.
\end{lemma}

\begin{proof}
For any nonempty sets $S \subset [n]$ and $T \subset [m]$ of sizes $k = \abs{S}$ and $\ell = \abs{T}$,
and any $u \in \{\pm 1\}^S$ and $v \in \{\pm 1\}^T$,
\SPcite{independent sum} implies
that
$\E \exp( \la \sum_{i \in S} \sum_{j \in T} u_i v_j X_{ij} ) \le \exp( k\ell \sig^2 ( e^{\abs{\la}}-1-\abs{\la}) ) $
for all $\la \in \R$. The union bound and Bennett's inequality \eqref{eq:Bennett} then imply that for any $r_{k,\ell} \ge 0$,
\begin{equation}
	\label{eq:SubmatrixSums1}
	\pr \big( X_{S,T} \ge r_{k,\ell} \big)
	\wle \sum_{u \in \{\pm 1\}^S} \sum_{v \in \{\pm 1\}^T} \pr \Big( \sum_{i \in S} \sum_{j \in T} u_i v_j X_{ij} \ge r_{k,\ell} \Big)
	\wle 2^k 2^\ell \beta(r_{k,\ell}, k \ell \sig^2).
\end{equation}
By the union bound, and the bounds
$\binom{n}{k} \le \left( \frac{en}{k}  \right)^k$,
and
$\binom{m}{\ell} \le \left( \frac{em}{\ell} \right)^\ell$,
it follows that
\begin{align*}
	\pr\bigg( \bigcup_{S,T} \left\{ X_{S,T} \ge r_{\abs{S},\abs{T}} \right\} \bigg)
	&\wle \sum_{k=1}^n \sum_{\ell = 1}^m \binom{n}{k} \binom{m}{\ell} 2^k 2^\ell \beta(r_{k,\ell}, k \ell \sig^2) \\
	&\wle \sum_{k=1}^n \sum_{\ell = 1}^m \Big( \frac{2en}{k} \Big)^k \Big( \frac{2em}{\ell} \Big)^\ell \beta(r_{k,\ell}, k \ell \sig^2).
\end{align*}
Denote by $\beta^{-1}( \cdot, \sig^2)$ the inverse of $\beta(\cdot, \sig^2)$ with respect to the first input variable, which is well defined
due to Lemma~\ref{the:BennettBijection}.
Let us plug in $r_{k,\ell} = \beta^{-1}( e^{-(1+t) \tau_{k,\ell}}, k\ell \sig^2)$.
Then
$\beta(r_{k,\ell}, k \ell \sig^2) = \left( \frac{k}{2en} \right)^{k(t+1)} \left( \frac{\ell}{2em} \right)^{\ell(t+1)}$,
and it follows that
\begin{align*}
	\pr\bigg( \bigcup_{S,T} \left\{ X_{S,T} \ge r_{\abs{S},\abs{T}} \right\} \bigg)
	&\wle \sum_{k=1}^n \sum_{\ell = 1}^m \Big( \frac{k}{2en} \Big)^{kt} \Big( \frac{\ell}{2em} \Big)^{\ell t}.
\end{align*}
Because the derivative of the log of
$\left( \frac{x}{2en} \right)^x$ satisfies
$\log(x/2en) + 1 \le \log(1/2e) + 1 < 0$ for $x \le n$,
we see that $\left( \frac{x}{2en} \right)^x$ is decreasing with respect to $x \in (0,n]$.
Because $t \ge 1$,
we see that
\[
\sum_{k=2}^n \left(\frac{k}{2en}\right)^{kt}
\wle \sum_{k=2}^n \left(\frac{1}{en}\right)^{2t}
\wle \sum_{k=2}^n \left(\frac{1}{en}\right)^{t+1}
\wle e^{-1} (en)^{-t}
\wle \frac12 (en)^{-t},
\]
and
\[
\sum_{k=1}^n \left(\frac{k}{2en}\right)^{kt}
\wle 2^{-t} (en)^{-t} + \frac12 (en)^{-t}
\wle (en)^{-t}.
\]
Similarly, $\sum_{\ell=1}^m \left( \frac{\ell}{2em} \right)^{\ell t} \le \frac{1}{(em)^\top}$,
and we conclude that
\begin{align*}
	\pr\bigg( \bigcup_{S,T} \left\{ X_{S,T} \ge r_{\abs{S},\abs{T}} \right\} \bigg)
	\wle (e^2nm)^{-t}.
\end{align*}
Finally, by recalling  the definition of $r_{k,\ell}$ and the fact that
$\beta$ is decreasing in its first input variable,
we see that
$X_{S,T} \ge r_{\abs{S},\abs{T}}$
iff $\beta(X_{S,T}, \abs{S}\abs{T} \sig^2) \le e^{-(1+t) \tau_{\abs{S},\abs{T}}}$.
Hence the claim follows from the above inequality.
\end{proof}

The following two simple lemmas are used a few times in the proof of Theorem~\ref{the:TrimmedMatrixConcentration}.

\begin{lemma}
\label{the:QuadraticFormSign}
For any $A \in \R^{n \times m}$, $u \in \R^n$, and $v \in \R^m$,
\[
\max_{x \in \{\pm 1\}^n} \max_{ y \in \{\pm 1\}^m } \sum_i \sum_j x_i u_i A_{ij} v_j y_j
\wle \max_{ x \in \{\pm 1\}^n } \max_{ y \in \{\pm 1\}^m } \sum_i \sum_j x_i \hhu_i A_{ij} \hhv_j y_j
\]
whenever
$\abs{u_i} \le \abs{\hhu_i}$ for all $i$
and 
$\abs{v_j} \le \abs{\hhv_j}$ for all $j$.
\end{lemma}
\begin{proof}
Denote $z_i = \sum_j A_{ij} v_j y_j $ and observe by applying $\abs{u_i} \le \abs{\hhu_i}$ that
\[
\sum_i \sum_j x_i u_i A_{ij} v_j y_j
\weq \sum_i x_i u_i z_i
\wle \sum_i \abs{x_i u_i z_i}
\wle \sum_i \abs{x_i \hhu_i z_i}.
\]
When maximizing the right side above with respect to $x \in \{\pm 1\}^n$, 
we may restrict to vectors in which $x_i$ has the same sign as $\hhu_i z_i$.
Therefore,
\[
\max_{x \in \{\pm 1\}^n} \sum_i \sum_j x_i u_i A_{ij} v_j y_j
\wle \max_{x \in \{\pm 1\}^n} \sum_i x_i \hhu_i z_i
\weq \max_{x \in \{\pm 1\}^n} \sum_i \sum_j x_i \hhu_i A_{ij} v_j y_j.
\]
Because this is true for all $y \in \{ \pm 1 \}^m$, we conclude that
\[
\max_{x \in \{\pm 1\}^n} \max_{ y \in \{\pm 1\}^m } \sum_i \sum_j x_i u_i A_{ij} v_j y_j
\wle \max_{ x \in \{\pm 1\}^n } \max_{ y \in \{\pm 1\}^m } \sum_i \sum_j x_i \hhu_i A_{ij} v_j y_j.
\]
The claim follows by repeating the same argument with the roles of $x$ and $y$ interchanged.
\end{proof}

\begin{lemma}
\label{the:GeometricSumBound}
For any real numbers $a, b > 0$, $x \in (0,1)$, and $y > 1$,
\[
\sum_{k \in \Z: a x^k \le b} \nhquad a x^k
\wle \frac{b}{1-x}
\qquad\text{and}\qquad
\sum_{k \in \Z: a y^k \le b} \nhquad a y^k
\wle \frac{b}{1-1/y}.
\]
\end{lemma}
\begin{proof}
Let $\ell \in \Z$ be the smallest integer such that $x^\ell \le b/a$.
Then
\[
\sum_{k \in \Z: a x^k \le b} a x^k
\weq \sum_{k=\ell}^\infty a x^k
\weq a x^\ell \sum_{k=0}^\infty x^k
\weq \frac{a x^\ell}{1-x}
\wle \frac{b}{1-x}.
\]
This confirms the first inequality. The second inequality follows by applying the first inequality with $x = 1/y$.
\end{proof}

\begin{proof}[Proof of Theorem~\ref{the:TrimmedMatrixConcentration}]
Let $X \in \R^{n \times m}$ be a random matrix with independent centered entries with sub-Poisson variance proxy $\sig^2$.
Because the spectral norm of a submatrix is less than the spectral norm of the whole matrix, without loss of generality, we may assume that $n=m$ (the smaller dimension can be extended to the larger dimension by padding with zeros, which does not affect the claimed bound of $\spenorm{X}$ nor the probability bound).
Fix a constant $\Ctrim > 0$, and denote
\begin{equation}
	\label{eq:TrimmingSymmetric}
	X_{ij}'
	\weq
	\begin{cases}
		X_{ij}, & \norm{X_{i:}}_1 \vee \norm{X_{:j}}_1\le \Ctrim  n\sig^2, \\
		0, & \text{otherwise}.
	\end{cases}
\end{equation}

\subsection{Discretizing the unit sphere}
Let $0< \eps < 1/2$.
Recall that an $\eps$-cover of $\bS$ is a finite set $\bS_\eps' \subset \bS$
such that for every $x \in \bS$ there exists a point $x' \in \bS_\eps'$ such that $\norm{x'-x} \le \eps$.
By \cite[Corollary 4.2.13]{Vershynin_2018}, 
the sphere $\bS = \{x \in \R^n \colon \norm{x} = 1\}$
admits an $\eps$-cover $\bS_\eps'$
with size bounded by $\abs{\bS_{\eps}'} \le (2\eps^{-1}+1)^n$.
We observe that also the set
\[
 \bS_\eps
 \weq \big\{ (x_1u_1,\dots, x_n u_n) \colon u \in \bS_{\eps}', \, x \in \{\pm 1\}^n \big\}
\]
is an $\eps$-cover of $\bS$, and
has size bounded by $\abs{\bS_{\eps}} \le 2^n (2\eps^{-1}+1)^n$.
We note that $\bS_\eps$ is \emph{sign-symmetric} in the sense that any sign change
of coordinates of a point in $\bS_\eps$ leaves the point in $\bS_\eps$.
The spectral norm of $X'$ can be approximated \cite[Exercise 4.4.3]{Vershynin_2018} by
\begin{equation}
	\label{eq:Cover}
	\spenorm{X'}
	\wle (1-2\eps)^{-1} \max_{u, v \in \bS_{\eps}} \iprod{u,X'v}.
\end{equation}

\subsection{Discretizing the unit interval}

Fix numbers $a, b \in (0,1)$ and partition the interval $(0,1]$ into bins
$(a_{k+1},a_k]$ and
$(b_{\ell+1},b_\ell]$
indexed by $k=0,1,\dots$ and $\ell=0,1,\dots$
using $a_k = a^k$ and $b_\ell = b^\ell$.
We further split the bin pairs into
\emph{light pairs} $L = \{(k,\ell) \colon a_k b_\ell \le \tau \}$ and
\emph{heavy pairs} $H = \{(k,\ell) \colon a_k b_\ell > \tau\}$
using a threshold parameter $\tau > 0$.
Then we split the inner products appearing in \eqref{eq:Cover}
according to 
\begin{equation}
	\label{eq:InnerProductSplit}
	\iprod{u,X'v}
	\weq \underbrace{\sum_{(k,\ell) \in L} \sum_{i \in S_k(u)} \sum_{j \in T_\ell(u)} u_iX_{ij}' v_j}_{\iprod{u,X'v}_L}
	\ + \underbrace{\sum_{(k,\ell) \in H} \sum_{i \in S_k(u)} \sum_{j \in T_\ell(u)} u_iX_{ij}' v_j}_{\iprod{u,X'v}_H},
\end{equation}
where
\begin{equation}
	\label{eq:Bins}
	S_k(u) = \set{i \colon \abs{u_i} \in (a_{k+1},a_k]}
	\qquad \text{and} \qquad
	T_\ell(v) = \set{j \colon \abs{v_j} \in (b_{\ell+1},b_\ell]}.
\end{equation}

We will next verify that
\begin{equation}
	\label{eq:IgnoreTrimming}
	\max_{u,v \in \bS_\eps} \iprod{u, X' v}_L
	\wle  \max_{u,v \in \bS_\eps} \iprod{u, X v}_L.
\end{equation}
To see this, observe that
\[
\iprod{u, X' v}_L
\weq \sum_{(k,\ell) \in L} \sum_{i \in S_k(u)} \sum_{j \in T_\ell(u)} u_i X'_{ij} v_j
\weq \sum_{i,j} u_i X'_{ij} M_{ijuv} v_j,
\]
where $M_{ijuv} = \sum_{(k,\ell) \in L} \indic(i \in S_k(u)) \indic(j \in T_\ell(u))$.
By writing $X'_{ij} = \xi_i X_{ij} \eta_j$ with $\xi_i, \eta_j \in \{0,1\}$, 
recalling that $\bS_\eps$ is sign-symmetric,
and applying Lemma~\ref{the:QuadraticFormSign}, we find that
\begin{align*}
	\max_{u,v \in \bS_\eps} \iprod{u, X' v}_L
	&\weq \max_{u,v \in \bS_\eps} \max_{x,y \in \{\pm 1\}^n} \sum_{i,j} x_i u_i \xi_i X_{ij} M_{ijuv} \eta_j v_j y_j \\
	&\wle \max_{u,v \in \bS_\eps}\max_{x,y \in \{\pm 1\}^n} \sum_{i,j} x_i u_i X_{ij} M_{ijuv} v_j y_j \\
	&\weq \max_{u,v \in \bS_\eps} \iprod{u, X v}_L,
\end{align*}
so that \eqref{eq:IgnoreTrimming} is valid.  By combining \eqref{eq:Cover} and \eqref{eq:InnerProductSplit} with \eqref{eq:IgnoreTrimming}, we then find that
\begin{equation}
	\label{eq:LightHeavySplit}
	\spenorm{X'}
	\wle (1-2\eps)^{-1} \Big( \max_{u, v \in \bS_{\eps}} \iprod{u, Xv}_L + \max_{u, v \in \bS_{\eps}} \iprod{u,X'v}_H \Big).
\end{equation}

\subsubsection{Light weights}

Recalling the definitions in \eqref{eq:InnerProductSplit}, we see that for any $u,v \in \bS_\eps$,
\[
\tau^{-1} \iprod{u,X v}_L
\weq \sum_{i,j} w_{ij} X_{ij},
\]
where
$w_{ij} = \frac{u_i v_j}{\tau} M_{ijuv}$ with $M_{ijuv} = \sum_{(k,\ell) \in L} \indic( i \in S_k(u) ) \indic( j \in T_\ell(v) )$.
The right side above is a sum of independent random variables weighted by $w_{ij}$.
Note that
$M_{ijuv} = 1$
if
there exists $(k,\ell)$ such that $a_k b_\ell \le \tau$ and
$\abs{u_i} \in (a_{k+1}, a_k]$ and 
$\abs{v_\ell} \in (b_{\ell+1}, b_\ell]$, and zero otherwise.
Therefore, the weights are bounded by $\abs{w_{ij}} \le 1$.
Consequently, \SPcite{independent sum} and \SPcite{balanced} then imply that the random variable
$\tau^{-1} \iprod{u,X v}_L$ has a sub-Poisson variance proxy
\[
\sig^2 \sum_{ij} w_{ij}^2
\weq \frac{\sig^2}{\tau^2}
\sum_{ij} u_i^2 v_j^2 M_{ijuv}^2
\wle \frac{\sig^2}{\tau^2} \sum_{ij} u_i^2 v_j^2
\weq (\sig/\tau)^2.
\]
Bernstein's inequality \eqref{eq:Bernstein2} then implies that
\[
\pr\left( \iprod{u,Xv}_L \ge \de \right) \\
\weq \pr\left( \tau^{-1} \iprod{u,Xv}_L \ge \frac{\de}{\tau}\right) \\
\wle \exp\left( -\left( \frac{(\de/\tau)^2}{4 (\sig/\tau)^2} \wedge \frac{3(\de/\tau)}{4} \right) \right)
\weq \exp \left( -\frac{3\de}{4\tau} \right)
\]
for all $\de \ge 3\sig^2/\tau$. By setting $\de = 3 t \sig^2/\tau$,
applying the union bound,
and recalling that $\abs{\bS_{\eps}} \le (4\eps^{-1}+2)^n$, it follows that for all $t \ge 1$,
\begin{equation}
	\label{eq:LightPairBernstein}
	\pr \left( \max_{u,v \in \bS_\eps} \iprod{u,X v}_L \ge \frac{3 t \sig^2}{\tau}\right)
	\wle \abs{\bS_\eps}^2 \exp \left( -\frac{3\de}{4\tau} \right)
	\wle \exp\left( 2 n \log(4\eps^{-1}+2) - \frac{9 t \sig^2}{4\tau^2} \right).
\end{equation}

\subsubsection{Heavy weights}

The treatment of heavy weight pairs is based on analyzing block maxima
\[
X_{S,T}
\weq \max_{u \in \{\pm 1\}^S} \max_{v \in \{\pm 1\}^T} \sum_{i \in S} \sum_{j \in T} u_i X_{ij} v_j
\]
defined for index sets $S,T \subset [n]$.
Fix an integer $t \ge 1$, and 
denote by $E$ the event that 
\begin{equation}
	\label{eq:SubmatrixSumsSymm}
	- \log \beta(X_{S,T}, \abs{S} \abs{T} \sig^2)
	\wle (1+t) \left( \abs{S} \log \frac{2en}{\abs{S}} + \abs{T} \log \frac{2en}{\abs{T}} \right)
\end{equation}
for all nonempty sets $S,T \subset [n]$,
where the function $\beta$ is defined by \eqref{eq:BennettBound}.
By Lemma~\ref{the:SubmatrixSums}, we know that
$\pr(E^c) \le (e n)^{-2t}$.

Fix a realization of $X \in \R^{n \times n}$ satisfying \eqref{eq:SubmatrixSumsSymm}.
Fix unit vectors $u,v \in \R^n$.
We abbreviate the sets in \eqref{eq:Bins} by $S_k = S_k(u)$ and $T_\ell = T_\ell(v)$, and 
denote $s_k = \abs{S_k}$ and $t_\ell = \abs{T_\ell}$.
Also define
\begin{align*}
	Y_{k\ell} &\weq \max_{x \in \{\pm 1\}^{S_k}} \max_{y \in \{\pm 1\}^{T_\ell}} \sum_{i \in S_k} \sum_{j \in T_\ell} x_i X_{ij} y_j, \\
	Y'_{k\ell} &\weq \max_{x \in \{\pm 1\}^{S_k}} \max_{y \in \{\pm 1\}^{T_\ell}} \sum_{i \in S_k} \sum_{j \in T_\ell} x_i X'_{ij} y_j.
\end{align*}
By writing $X'_{ij} = \xi_i X_{ij} \eta_j$ where $\xi_i, \eta_j \in \{0,1\}$, we find by Lemma~\ref{the:QuadraticFormSign} that
\begin{align}
	\label{eq:HeavyCond1}
	Y'_{k\ell}
	&\weq
	\max_{x \in \{\pm 1\}^{S_k}} \max_{y \in \{\pm 1\}^{T_\ell}}
	\sum_{i \in S_k} \sum_{j \in T_\ell} x_i X'_{ij} y_j
	\wle \max_{x \in \{\pm 1\}^{S_k}} \max_{y \in \{\pm 1\}^{T_\ell}}
	\sum_{i \in S_k} \sum_{j \in T_\ell} x_i X_{ij} y_j
	\weq Y_{k\ell}.
\end{align}
By noting that $\abs{u_i} \le a_k$ and $\abs{v_j} \le b_\ell$ for all $i \in S_k$
and $j \in T_\ell$, and
applying Lemma~\ref{the:QuadraticFormSign},
we see that
\begin{align*}
	\sum_{i \in S_k} \sum_{j \in T_\ell} u_i X_{ij}' v_j
	\wle \max_{x \in \{\pm 1\}^{S_k}} \max_{y \in \{\pm 1\}^{T_\ell} }
	\sum_{i \in S_k} \sum_{j \in T_\ell} x_i u_i X_{ij}' y_j v_j
	\wle a_k Y'_{k\ell} b_\ell.
\end{align*}
By summing both sides of the above inequality with respect to $(k,\ell) \in H$, we find that
\begin{equation}
	\label{eq:BinnedNorm}
	\iprod{u,X'v}_H
	\wle \sum_{(k,\ell) \in H} a_k Y'_{k\ell} b_\ell.
\end{equation}
We also note that
$\norm{Y'_{k:}}_1 \le \sum_{i \in S_k} \norm{X'_{i:}}_1$
and
$\norm{Y'_{:\ell}}_1 \le \sum_{j \in T_\ell} \norm{X'_{:j}}_1$, so that by \eqref{eq:TrimmingSymmetric},
\begin{equation}
	\label{eq:HeavyCond2}
	\norm{Y'_{k:}}_1 \wle \Ctrim  s_k n \sig^2
	\qquad \text{and} \qquad
	\norm{Y'_{:\ell}}_1 \wle \Ctrim  t_\ell n \sig^2.
\end{equation}
Furthermore, by \eqref{eq:SubmatrixSumsSymm}, we see that for all $k,\ell$,
\begin{align}
	\label{eq:HeavyCond3}
	- \log \beta(Y_{k\ell}, \sig^2 s_k t_\ell )
	&\wle (1+t) \left( s_k \log \hhs_k + t_\ell \log \hht_\ell \right),
\end{align}
where
$\hhs_k = \frac{2en}{s_k}$,
$\hht_\ell = \frac{2en}{t_\ell}$.

Fix a constant $c > 0$. 
We consider combinations of the following inequalities:
\begin{gather}
	\label{eq:Heavy1}
	\frac{a_k}{b_\ell} \vee \frac{b_\ell}{a_k} \ \le \ \sig \sqrt{n}, \\
	\label{eq:Heavy2}
	\frac{Y_{k\ell}}{e \sig^2 s_k t_\ell} \ > \ \frac{e a_k b_\ell}{\tau}, \\
	\label{eq:Heavy3}
	\Big( \frac{Y_{k\ell}}{e \sig^2 s_k t_\ell\hhs_k^{1/3}} \vee \frac{c a_k^2 n}{\hhs_k^{1/3} \log \hhs_k} \Big)
	\wedge
	\Big( \frac{Y_{k\ell}}{e \sig^2 s_k t_\ell\hht_\ell^{1/3}} \vee \frac{c b_\ell^2 n}{\hht_\ell^{1/3} \log \hht_\ell} \Big)
	\ \le \ 1.
\end{gather}
We partition the pairs in \( H = \{(k,\ell) \colon a_k b_\ell > \tau\} \) into four disjoint sets:
\( H_1 \) contains the pairs not satisfying \eqref{eq:Heavy1};
\( H_2 \) the pairs satisfying \eqref{eq:Heavy1} but not \eqref{eq:Heavy2};
\( H_3 \) the pairs satisfying \eqref{eq:Heavy1}--\eqref{eq:Heavy2} but not \eqref{eq:Heavy3};
and \( H_4 \) contains the pairs satisfying \eqref{eq:Heavy1}--\eqref{eq:Heavy3}.

(i) Fix $(k,\ell) \in H_1$.  Because $u_i^2 \ge a_{k+1}^2$ for all $i \in S_k$,
we find that
\[
\sum_k a_k^2 s_k
\weq \sum_k a_k^2 \sum_{i \in S_k} \frac{a_{k+1}^2}{a_{k+1}^2}
\wle \sum_k a_k^2 \sum_{i \in S_k} \frac{u_i^2}{a_{k+1}^2}
\weq a^{-2} \sum_k \sum_{i \in S_k} u_i^2
\weq a^{-2}.
\]
By repeating a similar computation for $t_\ell$, we conclude that
\begin{equation}
	\label{eq:BinBoundSquareSum}
	\sum_k a_k^2 s_k \wle a^{-2}
	\qquad \text{and} \qquad
	\sum_\ell b_\ell^2 t_\ell \wle b^{-2}.
\end{equation}
By noting that
$a_k b_\ell (\frac{a_k}{b_\ell} \vee \frac{b_\ell}{a_k})
= a_k^2 \vee b_\ell^2
\le a_k^2 + b_\ell^2$,
we see that
$a_k b_\ell < \frac{a_k^2 + b_\ell^2}{\sig \sqrt{n}}$ whenever
$\frac{a_k}{b_\ell} \vee \frac{b_\ell}{a_k} > \sig \sqrt{n}$.
Hence
\[
\sum_{(k, \ell) \in H_1} a_k b_\ell Y'_{k\ell}
\wle \sum_{(k, \ell) \in H_1} \frac{a_k^2 + b_\ell^2}{\sig \sqrt{n}} \abs{Y'_{k\ell}}
\wle \sum_k \frac{a_k^2}{\sig \sqrt{n}} \norm{Y'_{k:}}_1
+ \sum_\ell \frac{b_\ell^2}{\sig \sqrt{n}} \norm{Y'_{:\ell}}_1.
\]
By plugging in the bounds \eqref{eq:HeavyCond2}
and then applying \eqref{eq:BinBoundSquareSum}, it follows that
\begin{equation}
	\label{eq:H1Conclusion}
	\sum_{(k, \ell) \in H_1} a_k b_\ell Y'_{k\ell}
	\wle \Ctrim (a^{-2}+b^{-2}) \sig \sqrt{n}.
\end{equation}


(ii) Observe that
$Y_{k\ell} \le (e^2 \sig^2/\tau) s_k t_\ell a_k b_\ell$ for $(k,\ell) \in H_2$, so
that
\[
\sum_{(k,\ell) \in H_2} a_k Y_{k\ell} b_\ell
\wle \frac{e^2 \sig^2}{\tau} \sum_{(k,\ell) \in H_2} a_k^2 s_k b_\ell^2 t_\ell
\wle \frac{e^2 \sig^2}{\tau} \sum_k a_k^2 s_k \sum_\ell b_\ell^2 t_\ell.
\]
By \eqref{eq:BinBoundSquareSum}, it follows that 
\begin{equation}
	\label{eq:H2Conclusion}
	\sum_{(k,\ell) \in H_2} a_k Y_{k\ell} b_\ell
	\wle \frac{e^2 \sig^2}{a^2 b^2 \tau}.
\end{equation}

(iii)
Let us now fix a pair $(k,\ell) \in H_3$.
Then the inequalities $\frac{Y_{k\ell}}{e \sig^2 s_k t_\ell} > \frac{e a_k b_\ell}{\tau}$
and
$a_k b_\ell > \tau$
imply that $\log \frac{Y_{k\ell}}{e \sig^2 s_k t_\ell} \ge 1$.
The bound $\beta(x, \sig^2) \le \left( \frac{e \sig^2}{x} \right)^x$ in Lemma~\ref{the:BennettBijection} implies that
$-\log \beta(Y_{k\ell}, \sig^2 s_k t_\ell) \ge Y_{k\ell} \log \frac{Y_{k\ell}}{e \sig^2 s_k t_\ell}$.
Hence by \eqref{eq:HeavyCond3},
we conclude that
\begin{equation}
	\label{eq:StratifiedSumOnBennettConcentration}
	\sum_{(k,\ell) \in H_3} a_k Y_{k\ell} b_\ell
	\wle (1+t) \left( \sum_{(k,\ell) \in H_3} a_k b_\ell  \frac{s_k \log \hhs_k} {\log \frac{Y_{k\ell}}{e \sig^2 s_k t_\ell}}
	+ \sum_{(k,\ell) \in H_3} a_k b_\ell  \frac{t_\ell \log \hht_\ell} {\log \frac{Y_{k\ell}}{e \sig^2 s_k t_\ell}} \right).
\end{equation}

We will derive an upper bound for the first sum in \eqref{eq:StratifiedSumOnBennettConcentration}.
The condition
\[
\frac{Y_{k\ell}}{e \sig^2 s_k t_\ell \hhs_k^{1/3}} \vee \frac{c a_k^2 n}{\hhs_k^{1/3} \log \hhs_k} > 1
\]
in the definition of $H_3$ implies that we may partition $H_3 = H_{31} \cup H_{32}$ using
\begin{align*}
	H_{31}
	&\weq \Big\{ (k,\ell) \in H_3 \colon \log \hhs_k < 3 \log \frac{Y_{k\ell}}{e \sig^2 s_k t_\ell} \Big\}, \\
	H_{32}
	&\weq \Big\{ (k,\ell) \in H_3 \colon 3\log \frac{Y_{k\ell}}{e \sig^2 s_k t_\ell} \le \log \hhs_k < c \frac{a_k^2 n}{\hhs_k^{1/3}} \Big\}.
\end{align*}
Because $b_\ell \le a_k \sig \sqrt{n}$ for all $(k,\ell) \in H_3$,
a geometric sum bound (Lemma~\ref{the:GeometricSumBound}) shows that
\[
\sum_{\ell: (k,\ell) \in H_{31}} \nhquad b_\ell
\wle \sum_{\ell: b_\ell \le a_k \sig \sqrt{n}} \nhquad b_\ell
\wle \frac{a_k \sig \sqrt{n}}{1-b}.
\]
Therefore, by \eqref{eq:BinBoundSquareSum},
\begin{equation}
	\label{eq:H31Conclusion}
	\sum_{(k,\ell) \in H_{31}} \nhquad a_k b_\ell \frac{s_k \log \hhs_k} {\log \frac{Y_{k\ell}}{e \sig^2 s_k t_\ell}}
	\wle 3 \nhquad \sum_{(k,\ell) \in H_{31}} \nhquad s_k a_k b_\ell
	\wle 3 \sum_k s_k a_k \sum_{\ell: (k,\ell) \in H_{31}} \nhquad b_\ell
	\wle \frac{3 \sig \sqrt{n}}{a^2(1-b)}.
\end{equation}

Consider then a pair $(k,\ell) \in H_{32}$.
In this case 
$\hhs_k^{1/3} \ge \frac{Y_{k\ell}}{e \sig^2 s_k t_\ell}$.
Recall also that every pair in $H_3$ satisfies
$\frac{Y_{k\ell}}{e \sig^2 s_k t_\ell} > \frac{e a_k b_\ell }{\tau} \ge e$.
Therefore,
\begin{equation}
	\label{eq:ConditionH32}
	\hhs_k^{1/3}
	\wge e \frac{a_k b_\ell}{\tau}.
\end{equation}
On the other hand, $ \log \hhs_k < c \frac{a_k^2 n}{\hhs_k^{1/3}}$ implies that
\[
\sum_{(k,\ell) \in H_{32}} \nhquad a_k b_\ell \frac{s_k \log \hhs_k} {\log \frac{Y_{k\ell}}{e \sig^2 s_k t_\ell}}
\wle \sum_{(k,\ell) \in H_{32}} \nhquad a_k b_\ell s_k c \frac{a_k^2 n}{\hhs_k^{1/3}}
\wle c n \sum_{(k,\ell) \in H_{32}} \nhquad a_k b_\ell \frac{s_k a_k^2}{\hhs_k^{1/3}}.
\]
Inequality \eqref{eq:ConditionH32} implies that $b_\ell \le \frac{\tau \hhs_k^{1/3}}{e a_k}$
for all $(k,\ell) \in H_{32}$. Therefore, by applying
a geometric sum bound (Lemma~\ref{the:GeometricSumBound}), and \eqref{eq:BinBoundSquareSum},
we see that
\[
\sum_{(k,\ell) \in H_{32}} \nhquad a_k b_\ell \frac{s_k a_k^2}{\hhs_k^{1/3}}
\wle \sum_k a_k \frac{s_k a_k^2}{\hhs_k^{1/3}}
\sum_{\ell: b_\ell \le \frac{\tau \hhs_k^{1/3}}{e a_k}} \nhquad  b_\ell 
\wle \sum_k a_k \frac{s_k a_k^2}{\hhs_k^{1/3}}
\frac{\tau \hhs_k^{1/3}}{e a_k (1-b)}
\wle \frac{\tau}{e a^2 (1-b)}.
\]
It follows that
\begin{equation}
	\label{eq:H32Conclusion}
	\sum_{(k,\ell) \in H_{32}} \nhquad a_k b_\ell \frac{s_k \log \hhs_k} {\log \frac{Y_{k\ell}}{e \sig^2 s_k t_\ell}}
	\wle \frac{c \tau n}{e a^2 (1-b)}.
\end{equation}
By combining \eqref{eq:H31Conclusion} and \eqref{eq:H32Conclusion},
we conclude that
\[
\sum_{(k,\ell) \in H_3} a_k b_\ell  \frac{s_k \log \hhs_k} {\log \frac{Y_{k\ell}}{e \sig^2 s_k t_\ell}}
\wle \frac{3 \sig \sqrt{n} + (c/e) \tau n}{a^2(1-b)}.
\]

A symmetric argument can be carried out for the second sum in
\eqref{eq:StratifiedSumOnBennettConcentration}. We conclude that
\begin{equation}
	\label{eq:H3Conclusion}
	\sum_{(k,\ell) \in H_3} a_k Y_{k\ell} b_\ell
	\wle (1+t) \left( \frac{1}{a^2(1-b)} + \frac{1}{(1-a)b^2} \right) ( 3 \sig \sqrt{n} + (c/e) \tau n ).
\end{equation}

(iv)
Consider a pair $(k,\ell) \in H_4$. In this case,
\begin{equation}
	\label{eq:H4Key}
	\frac{ Y_{k\ell}}{e \sig^2 s_k t_\ell \hhs_k^{1/3}} \vee \frac{c a_k^2 n}{\hhs_k^{1/3} \log \hhs_k}  \le 1
	\qquad \text{or} \qquad
	\frac{ Y_{k\ell}}{e \sig^2 s_k t_\ell\hht_\ell^{1/3}} \vee \frac{c b_\ell^2 n}{\hht_\ell^{1/3} \log \hht_\ell}
	\le 1.
\end{equation}
Assume that the first inequality in \eqref{eq:H4Key} is valid.
Then
\[
Y_{k\ell}
\wle e \sig^2 s_k t_\ell \hhs_k^{1/3}
\weq 2 e^2 \sig^2 t_\ell n \hhs_k^{-2/3}.
\]
In light of the numeric inequality $\log x = 3 \log x^{1/3} \le 3 x^{1/3}$, we see that
the first inequality in \eqref{eq:H4Key} also implies that
\(
c a_k^2 n
\le \hhs_k^{1/3} \log\hhs_k
\le 3 \hhs_k^{2/3}.
\)
By combining this with the above inequality, it follows that
\[
Y_{k\ell}
\wle \frac{6 e^2 \sig^2 t_\ell}{c a_k^2}.
\]
A symmetric argument shows that $Y_{k\ell} \le \frac{6 e^2 \sig^2 s_k}{c b_\ell^2}$ whenever the second inequality in \eqref{eq:H4Key} is valid.
We conclude that for all $(k,\ell) \in H_4$,
\[
Y_{k\ell}
\wle \frac{6 e^2 \sig^2}{c} \left( \frac{s_k}{b_\ell^2} \vee \frac{t_\ell}{a_k^2} \right)
\wle \frac{6 e^2 \sig^2}{c} \left( \frac{s_k}{b_\ell^2} + \frac{t_\ell}{a_k^2} \right),
\]
so that
\begin{align*}
	\sum_{(k,\ell) \in H_4} a_k Y_{k\ell} b_\ell
	\wle \frac{6 e^2 \sig^2}{c}
	\left( \sum_{(k,\ell) \in H_4} s_k \frac{a_k}{b_\ell} + \sum_{(k,\ell) \in H_4} t_\ell \frac{b_\ell}{a_k} \right).
\end{align*}

By applying Lemma~\ref{the:GeometricSumBound} to the geometric terms $b_\ell = b^\ell$,
and then \eqref{eq:BinBoundSquareSum}, we find that
\begin{align*}
	\sum_{(k,\ell) \in H_4} s_k \frac{a_k}{b_\ell}
	\wle \sum_k s_k a_k \sum_{\ell: \frac{1}{b_\ell} \le \frac{a_k}{\tau}} \frac{1}{b_\ell}
	\wle \frac{1}{1-b} \sum_k \frac{s_k a_k^2}{\tau}
	\wle \frac{1}{a^2(1-b)\tau}.
\end{align*}
By a similar argument,
we conclude that
\begin{equation}
	\label{eq:H4Conclusion}
	\sum_{(k,\ell) \in H_4} a_k Y_{k\ell} b_\ell
	\wle \frac{6 e^2\sigma^2}{c \tau}
	\left( \frac{1}{a^2(1-b)} + \frac{1}{b^2(1-a)} \right).
\end{equation}

Let us now select $a=b=\frac12$.
By combining the bounds
\eqref{eq:H1Conclusion},
\eqref{eq:H2Conclusion},
\eqref{eq:H3Conclusion},
\eqref{eq:H4Conclusion},
and recalling that $Y'_{k\ell} \le Y_{k\ell}$ due to \eqref{eq:HeavyCond1},
we conclude that
\[
\sum_{(k,\ell) \in H} a_k Y'_{k\ell} b_\ell
\wle 8 \Ctrim \sig \sqrt{n} 
+ 16 e^2\frac{\sigma^2}{\tau}
+ (1+t)\left(48 \sig \sqrt{n}+ \frac{16 c}{e} \tau n\right)
+ \frac{96 e^2}{c}\frac{\sigma^2}{\tau}.
\]
By plugging in
$c = \frac{3e\sig}{\tau\sqrt{n}}$,
and recalling \eqref{eq:BinnedNorm}, we may now conclude that
\begin{equation}
	\label{eq:HeavyPairBound}
	\iprod{u,X'v}_H
	\wle \left( 8 \Ctrim + \frac{16 e^2 \sig}{\tau\sqrt{n}} + 96 (1+t) + 32 e \right) \sig \sqrt{n}
\end{equation}
for all unit vectors $u,v$, all $t \ge 1$, and all matrices $X \in \R^{n \times n}$
such that \eqref{eq:SubmatrixSumsSymm} is valid.

\subsubsection{Conclusion}

By \eqref{eq:LightPairBernstein},
\[
\max_{u,v \in \bS_\eps} \iprod{u,X v}_L
\wle 3 (1+t) \sig^2/\tau.
\]
with complementary probability $p_L \le \exp\left( 2 n \log(4\eps^{-1}+2) - \frac{9 (1+t)\sig^2}{4\tau^2} \right)$.
By \eqref{eq:HeavyPairBound} and Lemma~\ref{the:SubmatrixSums},
\[
\max_{u,v \in \bS_\eps} \iprod{u,X'v}_H
\wle\left( 8 \Ctrim + \frac{16 e^2 \sig}{\tau\sqrt{n}} + 96 (1+t) + 32 e \right) \sig \sqrt{n}
\]
with complementary probability $p_H \le (e n)^{-2t}$.
In light of \eqref{eq:LightHeavySplit}, we conclude that
\[
\spenorm{X'}
\wle (1-2\eps)^{-1} \left( 8 \Ctrim + 16 e^2 \frac{\sig}{\tau\sqrt{n}} + 3(1+t) \frac{\sig}{\tau\sqrt{n}} + 96 (1+t) + 32 e \right) \sig \sqrt{n}
\]
with complementary probability $p \le p_L + p_H$.

Choose
$\tau = \frac{3 \sig}{\sqrt{ 8 n\log(4\eps^{-1}+2)} }$ and $\eps = \frac{1}{20}$.
The inequality $\log(en) = 1 + \log n \le n \le n \log(4\eps^{-1}+2)$ then implies that
\[
p_L
\wle \exp \left( -2 t n \log(4\eps^{-1}+2) \right)
\wle (en)^{-2t}.
\]
Then $(1-2\eps)^{-1} = \frac{10}{9}$, 
$\frac{\sig}{\tau \sqrt{n}} = \frac{\sqrt{8 \log (82)}}{3} \le 2$, and
\begin{align*}
	\spenorm{X'}
	&\wle \frac{10}{9} \left( 8 \Ctrim + 32 e^2 + 102 (1+t) + 32 e \right) \sig \sqrt{n} \\
	&\wle 9 \left( \Ctrim + 66 t \right) \sig \sqrt{n}
\end{align*}
with complementary probability at most $2 (en)^{-2t}$.
The claim of Theorem~\ref{the:TrimmedMatrixConcentration} follows.
\end{proof}

\section{Proof of Lemma  \ref{lem:AX^T bound}}
\label{app:lem:AX^T bound}

Lemma \ref{lem:AX^T bound} provides a specialized norm bound needed in a small part of an argument in Appendix \ref{app:thm:hollow gram matrix bound}. The proof is based on a chaining argument developed in \cite{Dai_Su_Wang_2024}.

\begin{lemma}\label{lem:AX^T bound}
Let $X\in\R^{n\times m}$ be a random matrix with
independent centered sub-Poisson entries with variance proxy $\sig^2$,
and let $A\in\R^{n\times m}$ be a deterministic matrix with uniformly bounded entries $\abs{A_{ij}}\le\sig^2$.
Then
\[
 \Pr\left( \fronorm{AX^\top} \ge 34(1+t)n\sig^2\left(\sqrt{m\sig^2}\vee\frac{2}{3}\right) \right)
 \wle \frac{n}{(en)^t},
 \qquad t \ge 1.
\]
Especially, if $m\sig^2\gtrsim 1$, then
$\fronorm{AX^\top} \lesssim n\sig^2\sqrt{m\sig^2}$ with high probability.
\end{lemma}

\begin{proof}
{\bf Discretization.} For a vector $\norm{u}\le 1$, define sets $S_k=\set{i\mid 2^{-k-1}<\abs{u_i}\le 2^{-k}}$, $k=0,\dots,K-1$, $S_K=\set{i\mid \abs{u_i}\le 2^{-K}}$ of size $s_k=\abs{S_k}$ and values $a_k=2^{-k}\ge\max\set{\abs{u_i}\mid i\in S_k}$. Then
\[
\begin{split}
	s_ka_k^2
	&= \sum_{i\in S_k}\frac{a_k^2}{u_i^2}u_i^2
	\le 4\sum_{i\in S_k}u_i^2
	\le 4, 
	\qquad k<K, \\
	s_Ka_K^2
	&\le n4^{-K}.
\end{split}
\]
Choose $K=\floor{\frac{\log n}{\log 4}}$ so that also $s_Ka_K^2\le n4^{-K}\le 4$ and $n/4\le 4^K\le n$. This implies $s_k\le 4a_k^{-2}=4^{k+1}$ for all $k$ and
\[
\sum_{k=0}^{K}s_ka_k^2
= \sum_{k=0}^{K}\sum_{i\in S_k}a_k^2
\le s_Ka_K^2 + 4\sum_{k=0}^{K-1}\sum_{i\in S_k}u_i^2
\le 4 + 4
\le 8.
\]
Now for any vector $x\in\R^n$, we have
\[
\sum_ix_iu_i
= \sum_k\sum_{i\in S_k}x_iu_i
\le \sum_k a_k\underbrace{\max_{v\in\set{-1,1}^{S_k}}\sum_{i\in S_k}x_iv_i}_{= x(S_k)}
= \sum_{k}a_kx(S_k).
\]
Since $\norm{x}= \sup_{\norm{u}\le 1}\sum_i x_iu_i$, we have
\begin{equation}\label{eq:AX^T discretization}
	\norm{x}
	\le \max_{S_0,\dots,S_K}\sum_{k}a_kx(S_k)
\end{equation}
where $S_0,\dots,S_K\subset[n]$ are assumed to satisfy $\sum_ks_ka_k^2\le 8$, $s_ka_k^2\le 4$.

{\bf Bounding $(AX^\top )_{i:}(S)$.} For any nonempty subset $S\subset[n]$, we have
\[
(AX^\top )_{i:}(S)= \max_{u\in\set{-1,+1}^{S}}\sum_{j}\sum_{i'\in S}A_{ij}X_{i'j}u_{i'},
\]
where $\sum_{j}\sum_{i'\in S}A_{ij}X_{i'j}u_{i'}$ is a sum of independent random variables. By Bernstein's inequality~\eqref{eq:Bernstein2}, we have a tail bound
\[
\begin{split}
	\Pr\left(\bigcup_{i=1}^{n}\bigcup_{\emptyset\ne S\subset[n]}\set*{(AX^\top )_{i:}(S)\ge t_{\abs{S}}}\right) 
	&\le \sum_{i=1}^{n}\sum_{s=1}^{n}\sum_{S\in\binom{[n]}{s}}\sum_{v\in\set{-1,+1}^{S}}\Pr\left(\sum_{i'\in S}\sum_{j=1}^{m}\underbrace{v_{i'}\frac{A_{ij}}{\sig^2}}_{\in[-1,1]}X_{i'j}\ge \frac{t_s}{\sig^2}\right) \\
	&\le n\sum_{s=1}^{n}\binom{n}{s}2^{s}\exp\left(-\left(\frac{(t_s/\sig^2)^2}{4ms\sig^2}\wedge\frac{3(t_s/\sig^2)}{4}\right)\right) \\
	&\le n\sum_{s=1}^{n}\exp\left(s\log\frac{2en}{s}-\left(\frac{t_s^2}{4ms\sig^6}\wedge\frac{3t_s}{4\sig^2}\right)\right).
\end{split}
\]
Now $t_s=\sqrt{4ms^2\sig^6\log\frac{2en}{s}}\vee\frac{4\sig^2s\log\frac{2en}{s}}{3}$ satisfies $s\log\frac{2en}{s}=\frac{t_s^2}{4ms\sig^6}\wedge\frac{3t_s}{4\sig^2}$ so that
\[
\begin{split}
	\Pr\left(\bigcup_{i=1}^{n}\bigcup_{\emptyset\ne S\subset[n]}\set*{(AX^\top )_{i:}(S)\ge (1+t)t_s}\right) 
	&\le n\sum_{s=1}^{n}\exp\left(s\log\frac{2en}{s}-(1+t)\left(\frac{t_s^2}{4ms\sig^6}\wedge\frac{3t_s}{4\sig^2}\right)\right) \\
	&= n\sum_{s=1}^{n}\exp\left(-ts\log\frac{2en}{s}\right) \\
	&= n\sum_{s=1}^{n}\left(\frac{s}{2en}\right)^{ts} \\
	&\overset{(*)}{\le} n\left(\left(\frac{1}{2en}\right)^t + (n-1)\left(\frac{2}{2en}\right)^{2t}\right) \\
	&\overset{t\ge 1}{\le} n\left(\frac{1}{2}\left(\frac{1}{en}\right)^{t} + (n-1)\left(\frac{1}{en}\right)^{1+t}\right) \\
	&\le n\left(\frac{1}{2}\left(\frac{1}{en}\right)^{t} + \frac{1}{e}\left(\frac{1}{en}\right)^{t}\right) \\
	&\le \frac{n}{(en)^t},
\end{split}
\]
where inequality $(*)$ holds, because the function $s\mapsto\left(\frac{s}{2en}\right)^{s}$ is decreasing (the derivative of its logarithm $s\log\frac{s}{2en}$ is negative $\log\frac{s}{2en} + 1=\log\frac{s}{2n}$, $s\le n$). 

{\bf Norm bound.} Define a rare event
\[
E
= \bigcup_{i=1}^{n}\bigcup_{\emptyset\ne S\subset[n]}\set*{(AX^\top )_{i:}(S)\ge (1+t)t_{\abs{S}}}.
\]
The aim is to bound the norms of the row vectors $\norm{(AX^\top )_{i:}}$ under a high-probability event $E^c$. When $E^c$ holds, inequality \eqref{eq:AX^T discretization} gives
\[
\begin{split}
 \fronorm{AX^\top}
 &= \sqrt{\sum_{i=1}^n\norm{(AX^\top )_{i:}}^2} \\
 &\le\sqrt{n}\max_i \max_{S_0,\dots,S_K}\sum_{k}a_k (AX^\top )_{i:}(S_k) \\
 &\le \sqrt{n}\max_{S_0,\dots,S_K}(1+t)\sum_{k}a_k t_{s_k}.
\end{split}
\]
Evaluating $t_{s_k}=\sqrt{4ms_k^2\sig^6\log\frac{2en}{s_k}}\vee\frac{4\sig^2s_k\log\frac{2en}{s_k}}{3}$ gives an upper bound
\[
\begin{split}
 \fronorm{AX^\top}
 &= (1+t)\sqrt{n}\max_{S_0,\dots,S_K}\left(\sum_{k}\left(\sqrt{4ms_k^2\sig^6\log\frac{2en}{s_k}}\vee\frac{4\sig^2s_k\log\frac{2en}{s_k}}{3}\right)a_k\right) \\
 &\le 2(1+t)\sqrt{n}\sig^2\max_{S_0,\dots,S_K}
 \left( \sum_{k}\left(\sqrt{m\sig^2}\vee\frac{2}{3}\right)a_ks_k\log\frac{2en}{s_k} \right) \\
 &= 2(1+t)\sqrt{n}\sig^2\left(\sqrt{m\sig^2}\vee\frac{2}{3}\right)\max_{S_0,\dots,S_K}\sum_{k}a_ks_k\log\frac{2en}{s_k}.
\end{split}
\]
The above sum is estimated by recognizing it as an inner product $\inner{x}{y}=\sum_{k}s_kx_ky_k$, namely now Cauchy--Schwarz inequality and subadditivity of the square root yield
\[
\begin{split}
	\sum_{k=0}^{K}s_ka_k\log\frac{2en}{s_k}
	&\le\sqrt{\sum_{k=0}^{K}s_ka_k^2}\sqrt{\sum_{k=0}^{K} s_k\log^2\frac{2en}{s_k}} \\
	&\le \sqrt{8}\sqrt{\sum_{k=0}^{K} s_k\log^2\frac{4e^2n}{s_k}}\\
	&\overset{(*)}{\le} \sqrt{8}\sqrt{\sum_{k=0}^{K} 4^{k+1}\log^2\frac{4e^2n}{4^{k+1}}} \\
	&\le \sqrt{32}\sum_{k=0}^{K} 2^k\log\frac{e^2n}{4^k} \\
	&= \sqrt{32}\left(\log(e^2n)\sum_{k=0}^{K} 2^k - \log(4)\sum_{k=0}^{K} k2^k\right),
\end{split}
\]
where the inequality $(*)$ follows from increasingness of a function $s\mapsto s\log^2\frac{4e^2n}{s}$ as it has a nonnegative derivative $\log^2\frac{4e^2n}{s} - 2\log\frac{4e^2n}{s}=(\log\frac{4e^2n}{s})\log\frac{4n}{s}\ge 0$ for $s\le 4n$. With a help of Wolfram, the last term is equal to
\[
\begin{split}
	\sqrt{32}\left(\log(e^2n)(2^{K+1}-1) - \log(4)2(2^K K-2^K+1)\right).
\end{split}
\]
Hence, we conclude that
\[
\begin{split}
	\sum_{k=0}^{K}s_ka_k\log\frac{2en}{s_k}
	&\le 2^{K+1}\sqrt{32}\left(\log(e^2n) - \log(4)(K-1)\right) \\
	&= 2^{K+1}\sqrt{32}\log\frac{e^2n}{4^{K-1}} \\
	&= 2^{K+1}\sqrt{32}\log\frac{4e^2n}{4^{K}} \\
	&\overset{(**)}{\le} 2\sqrt{32n}\log\frac{4e^2n}{n} \\
	&= 2\sqrt{32n}\log(4e^2),
\end{split}
\]
where the inequality $(**)$ follows from increasingness of a function $s\mapsto s\log\frac{4e^2n}{s^2}$ as it has a nonnegative derivative $\log\frac{16e^2n}{s^2}-2=\log\frac{16n}{s^2}\ge 0$ for $s\le\sqrt{n}$.

{\bf Summary.} Combining these results gives that the high-probability event $E^c$ implies
\[
\begin{split}
 \fronorm{AX^\top}
 &\le 2(1+t)\sqrt{n}\sig^2\left(\sqrt{m\sig^2}\vee\frac{2}{3}\right)2\sqrt{32n}\log(4e^2) \\
 &= 16\sqrt{2}\log(4e^2)(1+t)n\sig^2\left(\sqrt{m\sig^2}\vee\frac{2}{3}\right) \\
 &\le 34(1+t)n\sig^2\left(\sqrt{m\sig^2}\vee\frac{2}{3}\right). \\
\end{split}
\]
\end{proof}

Notice that in previous lemma, if $\abs{A_{ij}}=\Var(X_{ij})=\sig^2$, then the variance of a single entry of $AX^\top $ is $\Var((AX^\top )_{ij})=m\sig^6$ and the expected squared Frobenius norm is $\E\fronorm{AX^\top}^2 = n^2\sig^4m\sig^2$, which makes the high-probability bound rate-optimal.

\section{Proof of Theorem~\ref{thm:hollow gram matrix bound}}
\label{app:thm:hollow gram matrix bound}

This appendix analyzes concentration of a product $XX^\top $, where the entries of the random matrix $X\in \Z^{n\times m}$ are independent. The focus is only on the off-diagonal entries, because the diagonal and off-diagonal entries concentrate at different rates, as the following simple proposition suggests.

\begin{proposition}\label{lem:Expected Bernoulli Gram matrix}
Let $X\in\R^{n\times m}$ be a random matrix with independent Bernoulli-distributed entries $X_{ij}\eqd\Ber(p)$. Define a mask $M\in\R^{n\times n}$ by $M_{ij}=\I(i\neq j)$ and denote the elementwise matrix product by $\odot$. Then $A=XX^\top -\E XX^\top $ satisfies
\[
\begin{split}
	\sqrt{\E\spenorm{A\odot I}^2}
	&\ge \SD((XX^\top )_{ii})
	= \sqrt{mp(1-p)}, \\
	\sqrt{\E\spenorm{A\odot M}^2}
	&\ge \sqrt{\E\norm{(A\odot M)_{i:}}^2}
	= \sqrt{(n-1)m(1-p^2)}p.\\
\end{split}
\]
\end{proposition}

\begin{proof}
To estimate $\spenorm{A\odot I}$, first notice that $A_{ii}=\sum_{j}X_{ij}^2-\E X_{ij}^2$ is a sum of independent centered entries. Then
\[
\E\spenorm{A\odot I}^2
= \E\max_iA_{ii}^2
\ge \E A_{ii}^2
= \Var A_{ii} 
= \sum_{j}\Var(X_{ij}^2-\E X_{ij}^2)
= \sum_{j}\Var X_{ij}^2.
\]
Since $X_{ij}$ is a binary random variable, $X_{ij}^2=X_{ij}$ is simply a Bernoulli distributed random variable with variance $p(1-p)$. Hence, we conclude $\E\spenorm{A\odot I}^2\geq mp(1-p)$.

To estimate $\spenorm{A\odot M}$, first notice that $A_{ii'} = \sum_{j}X_{ij}X_{i'j} - \E X_{ij}X_{i'j}$ is a sum of independent centered entries. Then
\[
\begin{split}
	\E\spenorm{A\odot M}^2
	&\ge \E\norm{(A\odot M)_{i:}}^2 
	= \E\sum_{i':i'\neq i} A_{ii'}^2 
	= \sum_{i':i'\neq i} \Var A_{ii'} \\
	&= \sum_{i':i'\neq i}\sum_{j}  \Var(X_{ij}X_{i'j} - \E X_{ij}X_{i'j})
	= \sum_{i':i'\neq i}\sum_{j}  \Var(X_{ij}X_{i'j}).
\end{split}
\]
Since $X_{ij}$ and $X_{i'j}$ are independent binary random variables, their product is also Bernoulli distributed random variable with parameter $p^2$. Hence $\Var(X_{ij}X_{i'j}) = p^2(1-p^2)$ and $\E\spenorm{A\odot M}^2\geq (n-1)mp^2(1-p^2)$.
\end{proof}

This suggests that $XX^\top  - \E XX^\top $ consists of two parts, a diagonal and an off-diagonal part. For small $p\ll 1$, we expect the norm of the diagonal to be of order $\sqrt{mp}$ and the norm of the off-diagonal to be of order $\sqrt{nm}p$. If this holds, then the diagonal dominates the off-diagonal, when $\sqrt{mp}\gg\sqrt{nm}p$, or equivalently, $p\ll n^{-1}$. When we consider square matrices $m=n$, we rarely consider $p\ll n^{-1}$, but when $m\gg n$, we may be interested in the regime $p\ll n^{-1}$ and this phenomenon will become significant. The rest of this subsection aims at proving that the off-diagonal part is bounded from above by $\mathcal{O}(\sqrt{nm}p)$, or in a slightly more general setting, $\mathcal{O}(\sqrt{nm}\sig^2)$.

As the first step, Lemma \ref{lem:row L1 bound} shows that the $L_1$-norm of almost every row of $X-\E X$ concentrates well with high probability. Lemma \ref{lem:product of centered and integer matrices} bounds a product $X_1X_2^\top $ of two independent matrices $X_1,X_2$ with independent entries. The key observation is that the entries of $X_1X_2^\top $ are conditionally independent entries given $X_2$, when $X_2$ is sufficiently sparse. This allows us to apply Theorem \ref{the:TrimmedMatrixConcentration}. Lemma \ref{lem:Decoupling} applies a well-known decoupling technique (see for example Chapter 6 in \cite{Vershynin_2018}) to analyze a product $(X-\E X)X^\top $ as if $X$ and $X^\top $ were independent. The decoupling technique applies to expectations of convex functions of $(X-\E X)X^\top $ which is then converted to a tail bound in the proof. Theorem \ref{thm:hollow gram matrix bound} finally bounds the off-diagonal part of the Gram matrix $XX^\top $. The desired norm bound is obtained by removing certain rows and columns, and Lemma \ref{lem:hollow gram matrix row L1 bound} asserts that only a few rows and columns are removed. The bound presented here is derived with a simple application of Markov's inequality, and consequently the probability bound is worse than the bound in Lemma \ref{lem:row L1 bound}.

\begin{lemma}
\label{lem:row L1 bound}
Let $X\in\R^{n\times m}$ be a random matrix with independent centered
sub-Poisson entries with variance proxy $\sig^2$.
Define $S_i = \onenorm{X_{i:}} - \E \onenorm{X_{i:}}$.
If $m \sig^2 \le 8 \log en$, then
\[
 \Pr\left(\sum_{i=1}^{n}\I\set*{S_i\ge (1+t)m\sig^2}\ge \frac{en}{e^{m\sig^2/8}}\right)
 \wle (en)^{-t}
 \qquad\text{for all $t \ge 0$}.
\]
If $m\sig^2\ge 8\log en$, then
\begin{align*}
 \Pr\left(\max_i S_i \ge (1+t)\sqrt{8m\sig^2\log en}\right)
 &\wle (en)^{-t}
 \qquad \text{for all $t \ge 0$}.
\end{align*}
\end{lemma}

That is, if $m\sig^2\ge 8\log en$, then the largest $L_1$-norm is bounded by $\max_i\norm{X_{i:}}_1\le\max_i\E\norm{X_{i:}}_1 + C\sqrt{m\sig^2\log en}$ with high probability as $n\to\infty$, and if $1\ll m\sig^2\le 8\log en$, then the fraction of ``large'' $L_1$-norms $\frac{1}{n}\sum_{i=1}^{n}\I\set*{\norm{X_{i:}}_1\ge \E\norm{X_{i:}}_1 + 2m\sig^2}\ll 1$ tends to zero with high probability as $n\to\infty$, and if $C\le m\sig^2\le 8\log en$ for some constant $C$, then the fraction of ``large'' $L_1$-norms $\frac{1}{n}\sum_{i=1}^{n}\I\set*{\norm{X_{i:}}_1\ge \E\norm{X_{i:}}_1 + 2m\sig^2}\le e^{1-C/8}$ is at most of constant order with high probability as $n\to\infty$, but this fraction can be made arbitrarily small by setting $C$ large enough. Under an extra assumption $\E\abs{X_{ij}} \lesssim \sig^2$, we have $\E\norm{X_{i:}}_1\lesssim m\sig^2$ and the $L_1$-norm bounds are in the order of $m\sig^2$.

\begin{proof}
By \SPcite{absolute value}, the random variables $\abs{X_{ij}}$ are upper sub-Poisson with variance proxy $2\sig^2$. By \SPcite{independent sum}, the sums $S_i$ are upper sub-Poisson with variance proxy $2m\sig^2$. Fix an integer $k\geq 1$ and a real number $t\geq 0$. By Bernstein's inequality \eqref{eq:Bernstein2} and the inequality $\binom{n}{k}\leq(en/k)^k$, we have
\[
\begin{split}
 \Pr\left(\sum_{i=1}^{n}\I\set*{S_i\ge t}\ge k\right)
 &\wle \sum_{A\in\binom{[n]}{k}}\prod_{i\in A}\Pr\set{S_i\ge t}
 \wle \binom{n}{k}\left(\exp\left(-\left(\frac{t^2}{8m\sig^2}\wedge\frac{3t}{4}\right)\right)\right)^{k} \\
 &\wle \exp\left(k\left(\log\frac{en}{k}-\left(\frac{t^2}{8m\sig^2} \wedge \frac{3t}{4}\right)\right)\right).
\end{split}
\]
By substituting $t = (1+s)a$ with $a,s \ge 0$ and noting that $1+s \le (1+s)^2$, we find that
\[
 \Pr\left(\sum_{i=1}^{n}\I\set*{S_i\ge (1+s)a}\ge k\right)
 \wle \exp\left(k\left(\log\frac{en}{k}-(1+s)\left(\frac{a^2}{8m\sig^2} \wedge \frac{3a}{4}\right)\right)\right).
\]
For a clean bound, we choose $a$ so that $\log\frac{en}{k} = \frac{a^2}{8m\sig^2} \wedge \frac{3a}{4}$.
With the help of 
Lemma~\ref{lem:inverse of minimum},
we find that $a = \sqrt{8 m \sig^2 \log \frac{en}{k}} \vee \frac{4\log\frac{en}{k}}{3}$. 
We conclude that
\[
\Pr\left(\sum_{i=1}^{n}\I\set*{S_i\ge (1+s)\left(\sqrt{8m\sig^2\log\frac{en}{k}}\vee\frac{4\log\frac{en}{k}}{3}\right)}\ge k\right)
\wle\exp\left(-k\log\left(\frac{en}{k}\right)s\right).
\]
Notice that the function $k\mapsto k\log\frac{en}{k}$ is increasing for $1\le k\le n$ since it has a nonnegative derivative $\log\frac{en}{k}-1=\log\frac{n}{k}\ge 0$. Furthermore, since it is impossible for a sum of $n$ indicators to exceed $n$, we obtain the claimed probability bound $(en)^{-s}$. 

Choose $k=\ceil{\frac{en}{e^{m\sig^2/8}}}$. In the case $m\sig^2\le 8\log en$ we estimate $k\geq\frac{en}{e^{m\sig^2/8}}$. This in turn implies
\[
\begin{split}
	\sqrt{8m\sig^2\log\frac{en}{k}}\vee\frac{4\log\frac{en}{k}}{3}
	&\wle m\sig^2\vee\frac{m\sig^2}{6}
	\weq m\sig^2, 
\end{split}
\] 
and
\[
\Pr\left(\sum_{i=1}^{n}\I\set*{S_i\ge (1+s)m\sig^2}\ge \frac{en}{e^{m\sig^2/8}}\right)
\wle (en)^{-s}.
\]
In the case $m\sig^2\ge 8\log en$ we observe that $k=1$. This in turn implies 
\[
\begin{split}
	\sqrt{8m\sig^2\log\frac{en}{k}}\vee\frac{4\log\frac{en}{k}}{3}
	&\wle \sqrt{8m\sig^2\log en}\vee\frac{4}{3}\sqrt{\frac{m\sig ^2\log en}{8}} 
	\weq \sqrt{8m\sig^2\log en},
\end{split}
\]
and
\[
\Pr\left(\max_iS_i\ge (1+s)\sqrt{8m\sig^2\log en}\right) 
\weq \Pr\left(\sum_{i=1}^{n}\I\set*{S_i\ge (1+s)\sqrt{8m\sig^2\log en}}\ge 1\right) 
\wle(en)^{-s}.
\]
\end{proof}

\begin{lemma}
\label{lem:product of centered and integer matrices}
Fix positive integers $m,n_1,n_2$ such that $n=n_1+n_2$
satisfies $8 \log en \le m \sig^2$, and $3 e^2 n \sig^2 \le 1$.
Let $X_1 \in \R^{n_1\times m}, X_2 \in \Z^{n_2\times m}$
be independent random matrices
having independent sub-Poisson entries with variance proxy $\sig^2$,
and such that
$\E (X_1)_{ij} = 0$, 
$\E \abs{(X_1)_{ij}} \le \sig^2$,
and $\E \abs{(X_2)_{ij}} \le \sig^2$.
\begin{enumerate}[(i)]
\item Then
\[
 \begin{split}
 \Pr\left( \onenorm{( \abs{X_1} \, \abs{X_2}^\top )_{i:}} \ge t n m \sig^4 \right) 
  &\wle \frac{\sqrt{t}}{\sqrt{6}}\exp\left(-\frac{\sqrt{t}nm\sig^4}{40\sqrt{6}}\right) + 2n^{-1}e^{-\sqrt{t}/2\sqrt{6}}
 \end{split}
\]
for all
$
 t \ge  6 \vee \frac{96\log^2 m}{\log^2(1/3en\sig^2)}.
$
	
\item There exists an absolute constant $C$ such that
\begin{equation}
 \label{eq:YPrime}
 \Pr\left(\spenorm{ X_1 X_2^\top \odot N} \ge C \left(t + \Ctrim\right)\sqrt{nm}\sig^2\right)
 \wle 8n^{-1}e^{-t^{1/3}/3}
\end{equation}
for all
$ t\ge 5^{3/2}\vee\frac{27\log^{3}m}{\log^{3}(1/3en\sig^2)}$ 
where
$N \in \{0,1\}^{n_1 \times n_2}$ is defined by
\[
 N_{ij}
 \weq \indic \Big(
 \onenorm{(\abs{X_1} \, \abs{X_2}^\top )_{i:}} \vee
 \onenorm{(\abs{X_1} \, \abs{X_2}^\top )_{:j}} \le \Ctrim nm\sig^4
 \Big). 
\]
\end{enumerate}
\end{lemma}

In particular,
$\spenorm{X_1 X_2^\top \odot N} \lesssim (\log m)^3\sqrt{nm}\sig^2$ with high probability
when $n \gg 1$ and $n\sig^2 \lesssim 1$;
and
$\spenorm{X_1 X_2^\top \odot N}\lesssim \sqrt{nm}\sig^2$ with high probability
when $n \gg 1$ and $n\sig^2\lesssim m^{-c}$ for some constant $c>0$.

\begin{proof}
Denote $Y = X_1 X_2^\top$.
Decompose the integer matrix as $X_2 = \sum_{k=1}^{\infty}X_2^{(k)}$, where matrices $X_2^{(k)}$ are defined recursively by 
\[
X_2^{(k)}
= \argmin_{X\in\set{0,\pm 1}^{n_2\times m}:\max_j\norm{X_{:j}}_1\le 1}\sum_{i,j}\abs*{\left(X_2 - \left(X + \sum_{k'=1}^{k-1}X_2^{(k')}\right)\right)_{ij}},
\qquad k\ge 1
\]
with arbitrary tie breaks as there might not be unique minimizers. As an example of such a decomposition, consider the following small integer matrix
\[
\begin{bmatrix}
	3&-2&1 \\
	0&0&-1 \\
	0&1&0
\end{bmatrix}
=
\begin{bmatrix}
	1&-1&1 \\
	0&0&0 \\
	0&0&0
\end{bmatrix}
+
\begin{bmatrix}
	1&-1&0 \\
	0&0&-1 \\
	0&0&0
\end{bmatrix}
+
\begin{bmatrix}
	1&0&0 \\
	0&0&0 \\
	0&1&0
\end{bmatrix}.
\]
That is, the components are matrices with entries being in $\set{-1,0,1}$, each column having at most one nonzero entry and the $k$th component picks ``the $k$th one of each column'' if it exists. Since each column has at most one nonzero entry, the supports of the rows of $X_2^{(k)}$ are disjoint.

Notice that the summands are eventually zero matrices as each entry of $X_2$ is finite almost surely. Let $K_{\rm max}$ denote the number of nonzero components (which is random as $X_2$ is random).
By \SPcite{absolute value}, the random variables $\abs{X_{2}(i,j) - \E X_{2}(i,j)}$ are upper sub-Poisson with variance proxy $2\sig^2\le3\sig^2$. Consequently, the sums $\sum_{i}\abs{X_{2}(i,j) - \E X_{2}(i,j)}$ are upper sub-Poisson with variance proxy $3n_2\sig^2\le 3n\sig^2$ by \SPcite{independent sum}. Now for any $k_{\rm max}\ge 3n\sig^2$, Bennett's inequality \eqref{eq:Bennett} gives a tail bound
\begin{equation}\label{eq:Kmax bound}
\begin{split}
	&\Pr\left(K_{\rm max}\ge k_{\rm max}\right) \\
	&\weq \Pr\left(\bigcup_{j=1}^{m}\set*{\sum_{i=1}^{n_2}\abs{X_2(i,j)}\ge k_{\rm max}}\right) \\
	&\wle \sum_{j=1}^{m}\Pr\left(\sum_{i=1}^{n_2}\abs{X_2(i,j) - \E X_{2}(i,j)} + \abs{\E X_{2}(i,j)}\ge k_{\rm max}\right) \\
	&\wle \sum_{j=1}^{m}\Pr\bigg(\sum_{i=1}^{n_2}\abs{X_2(i,j) - \E X_{2}(i,j)} - \E\abs{X_2(i,j) - \E X_{2}(i,j)} \ge k_{\rm max} \\&\quad- \sum_{i=1}^{n_2}\abs{\E X_{2}(i,j)} + \E\abs{X_2(i,j) - \E X_{2}(i,j)}\bigg) \\
	&\wle \sum_{j=1}^{m}\Pr\bigg(\sum_{i=1}^{n_2}\abs{X_2(i,j) - \E X_{2}(i,j)} - \E\abs{X_2(i,j) - \E X_{2}(i,j)} \ge k_{\rm max} -3n\sig^2\bigg) \\
	&\overset{\eqref{eq:Bennett}}{\wle} me^{-3n\sig^2}\left(\frac{e3n\sig^2}{3n\sig^2 + k_{\rm max} -3n\sig^2}\right)^{3n\sig^2 +  k_{\rm max} -3n\sig^2} \\
	&\wle m\left(\frac{3en\sig^2}{k_{\rm max}}\right)^{k_{\rm max}}.
\end{split}
\end{equation}
The resulting upper bound holds also for $0<k_{\rm max}\le 3n\sig^2$ as the upper bound becomes trivial.

Conditioned on $X_2$, the product $Y^{(k)}= X_1(X_2^{(k)})^\top$ consists of independent entries (since the rows of $X_2^{(k)}$ have disjoint supports) and by \SPcite{independent sum}, they are sub-Poisson with variance proxy $\max_j\norm{(X_2^{(k)})_{j:}}_1\sig^{2}\le \max_j\norm{(X_2)_{j:}}_1\sig^{2}$. Conditioned on $X_2$, the sum $\sum_{l:\exists j:(X_2^{(k)})_{jl}\ne 0}\abs{(X_1)_{il}}$ consists of at most $\sum_{j,l}\abs{(X_2^{(k)})_{jl}}\le n\max_j\norm{(X_2)_{j:}}_1$ independent random variables that are upper sub-Poisson with variance proxy $2\sig^2$ by \SPcite{absolute value}. Hence, the sum $\sum_{l:\exists j:(X_2^{(k)})_{jl}\ne 0}\abs{(X_1)_{il}}$ is upper sub-Poisson with variance proxy $2n\max_j\norm{(X_2)_{j:}}_1\sig^2$ by \SPcite{independent sum}. The next objective is to bound $\max_{j}\norm{(X_2)_{j:}}_1$.

For any $t\ge 1$, the assumptions $\E\abs{(X_2)_{ij}}\le\sig^2$ and $8\log en\le m\sig^2$ and Lemma \ref{lem:row L1 bound} give a tail bound
\begin{equation}\label{eq:X2 row bound}
\begin{split}
	&\Pr\left(\max_j\norm{(X_2)_{j:}}_1\ge 5tm\sig^2\right) \\ 
	&\le \Pr\left(\max_j\norm{(X_2 - \E X_2)_{j:}}_1 + \norm{(\E X_2)_{j:}}_1\ge 5tm\sig^2\right) \\ 
	&\le \Pr\left(\max_j\norm{(X_2 - \E X_2)_{j:}}_1\ge 4tm\sig^2\right) \\ 
	&\le \Pr\left(\bigcup_{j}\set*{\norm{(X_2 - \E X_2)_{j:}}_1\ge \E\norm{(X_2 - \E X_2)_{j:}}_1 + (1+t)\sqrt{8m\sig^2\log en}}\right) \\
	&\overset{\eqref{lem:row L1 bound}}{\le} (en)^{-t}.
\end{split}
\end{equation}
That is, conditioned on $X_2$ under a high-probability event $\bigcap_i\set{\norm{(X_2)_{j:}}_1< 5tm\sig^2}$, the entries $Y_{ij}^{(k)}$ are sub-Poisson with variance proxy $5tm\sig^4$, and the sums $\sum_{l:\exists j:(X_2^{(k)})_{jl}\ne 0}\abs{(X_1)_{il}}$ are upper sub-Poisson with variance proxy $10tnm\sig^4$.

(i) Fix $s\ge 1\vee \frac{4\log m}{\log(1/3en\sig^2)}$
and $k_{\rm max} = \floor{s}$. Recall that each column of the matrix $X_2^{(k)}$ has at most one nonzero entry. By inequalities \eqref{eq:Kmax bound} and \eqref{eq:X2 row bound} we have that
\[
\begin{split}
&\Pr\left(\onenorm{( \abs{X_1} \, \abs{X_2}^\top )_{i:}}\ge 6s^2nm\sig^4\right) \\
&\le \Pr\bigg(\sum_{k\le k_{\rm max}}\sum_{j,l}\abs{(X_1)_{il}(X_2^{(k)})_{jl}}\ge 6k_{\rm max}snm\sig^4,\ \max_{i}\norm{(X_2)_{i:}}_1< 5sm\sig^2,\ K_{\rm max} \le k_{\rm max}\bigg) \\&\quad+\Pr\left(\max_i\norm{(X_2)_{i:}}_1\ge 5sm\sig^2\right) + \Pr\left(K_{\rm max}\ge s\right) \\
&\le \sum_{k\le k_{\rm max}}\Pr\bigg(\sum_{l:\exists j:(X_2^{(k)})_{jl}\ne 0}\abs{(X_1)_{il}}\ge 6snm\sig^4,\ \max_i\norm{(X_2)_{i:}}_1 < 5sm\sig^2\bigg) +  (en)^{-s} + m\left(\frac{3en\sig^2}{s}\right)^{s}.
\end{split}
\]
Notice that if $\max_i\norm{(X_2)_{i:}}_1 < 5sm\sig^2$, then the sum $\sum_{l:\exists j:(X_2^{(k)})_{jl}\ne 0}\abs{(X_1)_{il}}$ consists of at most $\sum_i\norm{(X_2)_{i:}}_1\le 5snm\sig^2$ entries with expectations $\E\abs{(X_1)_{ij}}\le\sig^2$ by assumption. Hence Bernstein's inequality \eqref{eq:Bernstein2} allow us to bound the first term by
\[
\begin{split}
	&\sum_{k\le k_{\rm max}}\E\bigg[\Pr\bigg(\sum_{l:\exists j:(X_2^{(k)})_{jl}\ne 0}\abs{(X_1)_{il}}\ge \E\bigg[\sum_{l:\exists j:(X_2^{(k)})_{jl}\ne 0}\abs{(X_1)_{il}}\mid X_2\bigg] + snm\sig^4 \mid X_2\bigg)\\&\quad\times\I\set*{\max_i\norm{(X_2)_{i:}}_1 < 5sm\sig^2}\bigg] \\
	&\le k_{\rm max}\exp\left(-\left(\frac{s^2n^2m^2\sig^8}{4\cdot 10snm\sig^4}\wedge\frac{3snm\sig^4}{4}\right)\right) 
	\weq k_{\rm max}\exp\left(-\frac{snm\sig^4}{40}\right). 
\end{split}
\]
From combining the bounds and $k_{\rm max}\le s$ it follows that
\[
\Pr\left(\onenorm{( \abs{X_1} \, \abs{X_2}^\top )_{i:}}\ge 6s^2nm\sig^4\right)
\wle s\exp\left(-\frac{snm\sig^4}{40}\right) + (en)^{-s} + m\left(\frac{3en\sig^2}{s}\right)^{s}.
\]
Now it remains to simplify the probability bound. Because $s\ge \frac{4\log m}{\log(1/3en\sig^2)}$ and $3e^2n\sig^2\le 1$, we have that
\[
\begin{split}
m\left(\frac{3en\sig^2}{s}\right)^{s}
&\wle \exp\left(\log m - \frac{s}{2}\log\frac{1}{3en\sig^2 }\right)\left(3en\sig^2\right)^{s/2}
\wle m^{-1}e^{-s/2}.
\end{split} 
\]
From the assumptions $8\log en\le m\sig^2$ and $3e^2n\sig^2\le 1$ it follows that $m\ge 8\sig^{-2}\ge 24e^2n\ge n$. Because also $s\geq 1$, we may further bound $(en)^{-s}+m^{-1}e^{-s/2}\leq 2n^{-1}e^{-s/2}$. The claim follows by substituting $s=\sqrt{t/6}$ and writing the assumptions on $s$ in terms of $t=6s^2$.

(ii) Fix $s\geq 5\vee\frac{9\log^{2}m}{\log^2(1/3en\sigma^2)}$ and $k_{\rm max}
= \floor{\sqrt{s}}$. Let us apply Theorem~\ref{the:TrimmedMatrixConcentration} to the matrices $Y^{(k)}$ conditioned on $X_2$ under the event $\bigcap_i\set{\norm{(X_2)_{j:}}_1<5sm\sigma^2}$. 
The entries of $Y^{(k)}$ are conditionally independent with variance proxy $5sm\sigma^4$. The trimming mask $N$ guarantees that the $1$-norms of the rows and columns of $Y^{(k)}\odot N$ are bounded from above according to
\[
\begin{split}
	\norm{Y_{i:}^{(k)}}_1
	&= \sum_{j=1}^{n_2}\abs{Y_{ij}^{(k)}}
	= \sum_{j=1}^{n_2}\abs*{\sum_{l=1}^{m}(X_1)_{il}(X_2^{(k)})_{jl}}
	\le \sum_{j=1}^{n_2}\sum_{l=1}^{m}\abs{(X_1)_{il}(X_2)_{jl}}
	\le \Ctrim nm\sig^4, \\
	\norm{Y_{:j}^{(k)}}_1
	&= \sum_{i=1}^{n_1}\abs{Y_{ij}^{(k)}}
	= \sum_{i=1}^{n_1}\abs*{\sum_{l=1}^{m}(X_1)_{il}(X_2^{(k)})_{jl}}
	\le \sum_{i=1}^{n_1}\sum_{l=1}^{m}\abs{(X_1)_{il}(X_2)_{jl}}
	\le \Ctrim nm\sig^4.
\end{split}
\]
Define $\Ctrim'= \Ctrim/5s$
and a new trimming masks $N^{(k)}\in\set{0,1}^{n_1\times n_2}$ defined by
\[
N_{ij}^{(k)}
\weq \indic \Big(
 \onenorm{Y_{i:}^{(k)}} \vee
 \onenorm{Y_{:j}^{(k)}} \le \Ctrim' n(5sm\sigma^4)
 \Big).
\]
By noticing $\Ctrim nm\sigma^4 = \Ctrim' n(5sm\sigma^4)$ we see that $\Ctrim'$ is the trimming constant needed in Theorem~\ref{the:TrimmedMatrixConcentration} and $N^{(k)}$ is the corresponding trimming mask. Furthermore, notice that $N$ picks a smaller submatrix than $N^{(k)}$ implying $\spenorm{Y^{(k)}\odot N}\leq\spenorm{Y^{(k)}\odot N^{(k)}}$. Since $\sqrt{s/5}\geq 1$ by the assumption $s\geq 5$, there exists an absolute constant $C$ such that
\[
\begin{split}
    &\Pr\left(\spenorm{Y\odot N}\ge Ck_{\rm max}\left(s + \frac{\Ctrim}{\sqrt{5s}}\right)\sqrt{nm}\sig^2\right) \\
	&\le \Pr\left(\spenorm{Y\odot N}\ge Ck_{\rm max}\left(s + \frac{\Ctrim}{\sqrt{5s}}\right)\sqrt{nm}\sig^2,K_{\rm max}\le k_{\rm max},\max_i\norm{(X_2)_{i:}}_1\le 5sm\sig^2\right) \\&\quad+ \Pr\left(K_{\rm max}\ge \sqrt{s}\right) + \Pr\left(\max_i\norm{(X_2)_{i:}}_1\ge 5sm\sig^2\right) \\
	&\le \E\Bigg[\Pr\left(\bigcup_{k=1}^{k_{\rm max}}\set*{\spenorm{Y^{(k)}\odot N^{(k)}}\ge C\left(\sqrt{\frac{s}{5}} + \Ctrim'\right)\sqrt{n(5sm\sig^4)}}\mid X_2\right)\\&\quad\times \I\set*{\bigcap_i\set{\norm{(X_2)_{i:}}_1< 5sm\sig^2}}\Bigg] + m\left(\frac{3en\sig^2 }{\sqrt{s}}\right)^{\sqrt{s}} + (en)^{-s} \\
	&\le 2k_{\rm max}(en)^{-2\sqrt{s/5}} + m\left(\frac{3en\sig^2 }{\sqrt{s}}\right)^{\sqrt{s}} + (en)^{-s}. \\
\end{split}
\]
Substituting the bounds $k_{\rm max}\leq\sqrt{s}$ and $2\sqrt{s/5}\le s$ gives
\[
\Pr\left(\spenorm{Y\odot N}\ge C\left(s^{3/2} + \frac{\Ctrim}{\sqrt{5}}\right)\sqrt{nm}\sig^2\right)
\wle 3 \sqrt{s} (en)^{-2\sqrt{s/5}} + m \left(\frac{3en\sig^2}{\sqrt{s}}\right)^{\sqrt{s}}.
\]
It remains to simplify the probability bound. For the first term, recall the assumption $s\geq 5$. Now
\[
3 \sqrt{s} (en)^{-2\sqrt{s/5}}
\wle 3\sqrt{5}e^{\sqrt{s/5}}(en)^{-2\sqrt{s/5}}
\weq 3\sqrt{5}n^{-2\sqrt{s/5}}e^{-\sqrt{s/5}}
\wle 3\sqrt{5}n^{-1}e^{-\sqrt{s}/3}.
\]
For the second term, recall also the assumptions $3e^2n\sig^2\leq 1$ and $s\ge \frac{9\log^2m}{\log^2(1/3en\sig^2)}$. Now 
$\frac{3en\sig^2}{\sqrt{s}}\le \frac{e^{-1}}{\sqrt{5}}\le e^{-1}$
and consequently
\[
m \left(\frac{3en\sig^2}{\sqrt{s}}\right)^{\sqrt{s}}
\leq \exp\left(\log m - \frac{2\sqrt{s}}{3}\log\left(\frac{1}{3en\sig^2}\right) - \frac{\sqrt{s}}{3}\right)
\leq\exp\left(-\log m - \frac{\sqrt{s}}{3}\right).
\]
From the assumptions $8\log en\le m\sig^2$ and $3e^2n\sig^2\le 1$ it follows that $m\ge 8\sig^{-2}\ge 24e^2n\ge n$. Therefore, we may further bound $m^{-1}e^{-\sqrt{s}/3}\le n^{-1}e^{-\sqrt{s}/3}$. Combining these gives
\[
\Pr\left(\spenorm{Y\odot N}\ge C\left(s^{3/2} + \frac{\Ctrim}{\sqrt{5}}\right)\sqrt{nm}\sig^2\right)
\wle (3 \sqrt{5}+1)n^{-1}e^{-\sqrt{s}/3}
\wle 8n^{-1}e^{-\sqrt{s}/3}.
\]
The claim follows by substituting $t=s^{3/2}\geq 5^{3/2}\vee\frac{27\log^{3}m}{\log^{3}(1/3en\sig^2)}$.

\end{proof}

\begin{lemma}[Decoupling]
\label{lem:Decoupling}
Let $X \in \Z^{n\times m}$ be
a random matrix with independent sub-Poisson entries
with variance proxy $\sig^2$, and such that $\E\abs{X_{ij}} \le \sig^2$.
Define indicator matrices $M,N \in \{0,1\}^{n\times n}$ by $M_{ij} = \I( i \ne j)$ and
\[
 N_{ij}
 \weq \indic \Big(
 \onenorm{ ((\abs{X-\E X} \abs{X}^\top) \odot M)_{i:} }
 \vee \onenorm{ ((\abs{X-\E X} \abs{X}^\top) \odot M)_{:j} }
 \le \Ctrim n m \sig^4 \Big).
\]
There exists an absolute constant $C$ such that
\[
 \Pr\left( \spenorm{((X-\E X)X^\top) \odot M \odot N} \ge C(t+\Ctrim)\sqrt{nm}\sig^2 \right)
 \wle C n^{-1} e^{-t^{1/3}}
\]
whenever $4\log en\le m\sig^2$, $6 e^2 n \sig^2 \le 1$, 
and
$t \ge C\left( 1\vee\frac{\log^3 m}{\log^3(1/6en\sig^2)} \right)$.
\end{lemma}

In particular, if $n \sig^2 \lesssim m^{-c}$ for some constant $c > 0$,
then
$\spenorm{((X-\E X)X^\top) \odot M \odot N} \lesssim \sqrt{nm}\sig^2$ with high probability as $n\to\infty$,
and if $n\sig^2\lesssim 1$, then
$\spenorm{((X-\E X)X^\top) \odot M \odot N}\lesssim \log^3(m)\sqrt{nm}\sig^2$ with high probability as $n\to\infty$.

\begin{proof}
We start by deriving a decoupling inequality which generalizes \cite[Theorem 6.1.1]{Vershynin_2018} to random matrices. 
Let $\de_1, \dots, \de_n$ be independent Bernoulli random variables with success probability $1/2$, independent of the random matrix $X$. Define the random index sets
\[
I = \{ i : \de_i = 1 \}, \qquad I^c = \{ i : \de_i = 0 \},
\]
and construct the masked matrices $X_{I:}, X_{I^c:} \in \Z^{n \times m}$ by
\[
(X_{I:})_{ij} = \de_i X_{ij}, \qquad (X_{I^c:})_{ij} = (1 - \de_i) X_{ij}.
\]
In other words, $X_{I:}$ is obtained by zeroing out all rows of $X$ outside of $I$, and similarly for $X_{I^c:}$ with respect to $I^c$.  We will show that for any nondecreasing convex function
$f \colon \R_{\ge 0} \to \R_{\ge 0}$,
the matrix
$Y = ((X-\E X)X^\top) \odot M \odot N$ is bounded by
\begin{equation}
 \label{eq:DecouplingInequality}
 \E f ( \spenorm{Y} )
 \wle \E f ( 4 \spenorm{\tilde Y} ),
\end{equation}
where $\tilde Y = ((X_{I:} - \E_I X_{I:})X_{I^c:}^\top) \odot \tilde{N}$ and $\tilde{N}\in\set{0,1}^{n\times n}$ is defined by
\[
 \tilde{N}_{ij}
 \weq \indic \Big(
 \onenorm{ (\abs{X_{I:} - \E_I X_{I:}}\abs{X_{I^c:}}^\top)_{i:} }
 \vee \onenorm{ (\abs{X_{I:} - \E_I X_{I:}}\abs{X_{I^c:}}^\top)_{:j} }
 \le \Ctrim n m \sig^4 \Big).
\]
Inequality~\eqref{eq:DecouplingInequality} means that the random variable
$\spenorm{Y}$ is dominated by the random variable
$4 \spenorm{\tilde Y}$
in the increasing convex stochastic order \cite{Leskela_Vihola_2013, Muller_Stoyan_2002}.  The right side of \eqref{eq:DecouplingInequality}
provides a convenient decoupling because $X_{I:}$ and $X_{I^c:}$ are conditionally independent given $I$.

Let us prove \eqref{eq:DecouplingInequality}.
Abbreviating
$\E_X = \E( \cdot | X)$,
we find that or any $i\ne j$,
\begin{align*}
 ((X-\E X)X^\top )_{ij}
 &\weq 4 \E_X \delta_i(1-\delta_j)((X-\E X)X^\top )_{ij} \\
 &\weq 4 \E_X \sum_{k}\delta_i(X-\E X)_{ik}(1-\delta_j)X_{jk} \\
 &\weq 4 \E_X \sum_{k}(X_{I:} - \E_I X_{I:})_{ik} (X_{I^{c}:})_{jk} \\
 &\weq 4 \E_X ((X_{I:} - \E_I X_{I:}) (X_{I^{c}:})^\top )_{ij}.
\end{align*}
Because $((X_{I:} - \E_I X_{I:})(X_{I^{c}:})^\top )_{ii}=0$,
we conclude that
\[
 ((X- \E X) X^\top) \odot M
 \weq 4 \E_X (X_{I:}-\E_I X_{I:})(X_{I^{c}:})^\top.
\]
Because the spectral norm is a convex function on the real vector space $\R^{n \times n}$,
and $f$ is convex and nondecreasing, it follows that $A \mapsto f(4 \spenorm{A})$ is convex,
and Jensen's inequality gives
\begin{align*}
 f\left(\spenorm{((X-\E X)X^\top) \odot M \odot N}\right)
 &\weq f \left(4 \spenorm{\E_X((X_{I:} - \E_I X_{I:}) X_{I^c:}^\top) \odot N}\right) \\
 &\wle \E_X f \left( 4 \spenorm{((X_{I:}-\E_I X_{I:})X_{I^c:}^\top) \odot N} \right).
\end{align*}
Notice that $N_{ij}=1$ implies $\tilde{N}_{ij}=1$. Consequently, $N$ picks a smaller submatrix resulting in an upper bound $\E_X f(4\spenorm{\tilde{Y}})$. By taking expectations, we conclude that \eqref{eq:DecouplingInequality} is valid.

Now Lemma~\ref{lem:product of centered and integer matrices} can be applied to the product
$(X_{I:}-\E_I X_{I:})X_{I^c:}^\top $ with variance proxy $2\sig^2$ (the number $2$ is needed
to have $\E\abs{(X_{I:}-\E X_{I:})_{ij}}\le 2\sig^2$).
Consequently, the trimming constant in Lemma~\ref{lem:product of centered and integer matrices}
becomes $\Ctrim'=\Ctrim/4$. Define
\[
 Z
 \weq \left(\frac{\spenorm{Y}}{8C\sqrt{nm}\sig^2} - \Ctrim'
   - \bigg (5^{3/2}\vee\frac{27\log^3 m}{\log^3(1/6en\sig^2)} \bigg) \right)_+,
\]
where $C$ is the universal constant from
Lemma~\ref{lem:product of centered and integer matrices}:(ii).
Since $f(x)=(x-1)_+^p$ is a nondecreasing convex function on $\R_{\ge 0}$ for $p \ge 1$,
by Lemma~\ref{lem:product of centered and integer matrices},
the absolute moments of $Z$ can be estimated from above as
\begin{align*}
 \E Z^p
 &\wle \E \E_I \left(\frac{4\spenorm{\tilde Y}}{8C\sqrt{nm}\sig^2}
   - \Ctrim' - \bigg( 5^{3/2}\vee\frac{27\log^3 m}{\log^3(1/6en\sig^2)} \bigg) \right)_+^p \\
 &\weq \E \int_0^{\infty} pt^{p-1}
   \Pr_I \left( \frac{\spenorm{\tilde Y}}{C\sqrt{nm}2\sig^2} \ge
   \bigg(5^{3/2} \vee \frac{27\log^3 m}{\log^3(1/3en2\sig^2)} \bigg) + t + \Ctrim'\right) \intdiff t \\
 &\overset{\text{Lemma~\ref{lem:product of centered and integer matrices}}}{\le}
   8 n^{-1} \int_0^{\infty} pt^{p-1}e^{-t^{1/3}/3}\intdiff t.
\end{align*}
By noting that
\begin{align*}
 \int_0^{\infty} pt^{p-1}e^{-t^{1/3}/3}\intdiff t
 \weq p3^{3(p-1)} \int_0^{\infty}u^{3(p-1)}e^{-u}3^4u^2 \intdiff u
 \weq 3^{3p} 3p\Gamma(3p)
 \weq 3^{3p} \Gamma(3p+1),
\end{align*}
we conclude that
\[
 \E Z^p \wle \frac{8}{n} 3^{3p} \Gamma(3p+1)
 \qquad \text{for all $p \ge 1$}.
\]
For $1/3\le p\le 1$, the above argument does not work as $x\mapsto x^p$ is not convex.
Now for any $0<\la<3^{-1}$, the geometric summation formula gives
\begin{align*}
 \E \left( e^{\la Z^{1/3}} - \sum_{k=0}^{2}\frac{\la^{k}Z^{k/3}}{k!} \right)
 &\weq \sum_{k=3}^{\infty}\frac{\la^{k}}{k!} \E Z^{k/3} \\
 &\wle \frac{8}{n} \sum_{k=3}^{\infty} \frac{\la^{k}}{k!}3^{k}\Gamma(k+1) \\
 &\weq \frac{8}{n} \sum_{k=3}^{\infty} (3\la)^{k} \\
 &\weq \frac{8}{n} \frac{(3\la)^3}{1-3\la}.
\end{align*}
By Markov's inequality, this gives a tail bound
\begin{align*}
 \Pr(Z \ge t)
 &\wle \Pr\left( \sum_{k=3}^{\infty} \frac{\la^{k}}{k!} Z^{k/3}
    \ge \sum_{k=3}^{\infty} \frac{\la^{k}}{k!} t^{k/3} \right) \\
 &\wle \frac{8}{n} \frac{(3\la)^3}{1-3\la}
    \left( \sum_{k=3}^{\infty}\frac{\la^{k}}{k!}t^{k/3} \right)^{-1}.
\end{align*}
To simplify the tail bound, let us study the relation between the series $\sum_{k=3}^{\infty}\frac{x^{k}}{k!}$ and the exponential function $e^{x}$. Denote the Taylor polynomials of the exponential function by $p_k(x) = \sum_{i=0}^{k}\frac{x^i}{i!}$. Since the derivatives of the polynomials satisfy $p_{k+1}'=p_k$, for any $x\ge 0$ and $k\ge 1$, we have
\[
\frac{\diff}{\diff x}\frac{e^x}{p_k(x)}
= \frac{e^xp_k(x)-e^xp_k'(x)}{p_k(x)^2}
= \frac{(p_k(x) - p_{k-1}(x))e^x}{p_k(x)^2}
= \frac{x^ke^x}{k!p_k(x)^2}
\ge 0.
\]
This implies that for any $x\ge x_0\ge 0$, we have $e^x/p_k(x)\ge e^{x_0}/p_k(x_0)$, or equivalently, $p_k(x)\le \frac{p_k(x_0)}{e^{x_0}}e^x$. Therefore, for any $t\ge t_0$, we have
\[
\begin{split}
	\left(\sum_{k=3}^{\infty}\frac{\la^{k}}{k!}t^{k/3}\right)^{-1}
	&= (e^{\la t^{1/3}} - p_2(\la t^{1/3}))^{-1} 
	\le \left(e^{\la t^{1/3}} - \frac{p_2(\la t_0^{1/3})}{e^{\la t_0^{1/3}}}e^{\la t^{1/3}}\right)^{-1} \\
	&= \left(1 - \frac{p_2(\la t_0^{1/3})}{e^{\la t_0^{1/3}}}\right)^{-1}e^{-\la t^{1/3}}.
\end{split}
\]
Choose $\la=6^{-1}$ and $t_0=\la^{-3}$. Then for all $t\ge 6^3=216$, we have
\[
\begin{split}
	\Pr(Z\ge t)
	&\wle\frac{8}{n}\frac{(3/6)^3}{1-3/6}\left(1 - \frac{p_2(1)}{e}\right)^{-1}e^{- t^{1/3}/6} 
    \wle \frac{25}{n}e^{- t^{1/3}/6}. 
\end{split}
\]
The claim follows now from
\begin{align*}
 \set{Z\ge t}
 &= \set*{\spenorm{Y} \ge
   8 C \left(\bigg(5^{3/2}\vee\frac{27\log^3 m}{\log^3(1/6en\sig^2)}\right)+t+\Ctrim' \bigg) \sqrt{nm}\sig^2} \\
 &\supset \set{\spenorm{Y} \ge 8C(2t+\Ctrim')\sqrt{nm}\sig^2} \\
 &\supset \set{\spenorm{Y} \ge 16C(t+\Ctrim)\sqrt{nm}\sig^2},
\end{align*}
when $t \ge 5^{3/2} \vee \frac{27\log^3 m}{\log^3(1/6en\sig^2)}$.
\end{proof}

\subsection{Proof of Theorem~\ref{thm:hollow gram matrix bound}}
\label{sec:hollow gram matrix bound}

Let us write the error term as
\[
 \Delta
 \weq \big((XX^\top) \odot M - \E X \E X^\top \big) \odot N
\] 
where the indicator matrices $M,N \in \{0,1\}^{n \times n}$ are given by
$M_{ij} = \indic\set{ i \ne j}$ and
\[
N=\xi\xi^{\top},
\qquad
\xi_i
= \indic\set*{\onenorm{ ((\abs{X} \abs{X}^\top) \odot M)_{i:} }\le \Ctrim nm \sig^4 }.
\]
We note that the error term can be decomposed according to
$
\Delta
= \Delta_1 + \Delta_2 - \Delta_3,
$
where
\begin{align*}
 \Delta_1 &\weq ((X-\E X)X^\top) \odot M \odot N, \\
 \Delta_2 &\weq ((\E X)(X-\E X)^\top) \odot M \odot N, \\
 \Delta_3 &\weq (\E X \, \E X^\top ) \odot I \odot N.
\end{align*}
Therefore, it suffices to derive upper bounds for the
spectral norms of $\Delta_1$, $\Delta_2$, $\Delta_3$.

(i) We start by analyzing $\spenorm{\Delta_1}$.
Let us first show that the trimming mask $N$ regularizes the rows and columns of
\[
\bar\Delta_1=(\abs{X-\E X}\abs{X}^\top) \odot M
\]
properly. Note that
\begin{align*}
 \abs{ (X_{ik}-\E X_{ik}) X_{jk} }
 &\wle \abs{X_{ik}X_{jk}} + \abs{ (\E X_{ik}) (X_{jk} - \E X_{jk})} + \abs{ (\E X_{ik}) (\E X_{jk}) } \\
 &\wle \abs{X_{ik}X_{jk}} + \maxnorm{\E X} \abs{ (X_{jk} - \E X_{jk})} + \maxnorm{\E X}^2.
\end{align*}
By summing the above inequality with respect to $k$,
we find that
\begin{align*}
 (\abs{X - \E X} \, \abs{X}^\top)_{ij}
 &\wle (\abs{X} \abs{X}^\top)_{ij} + \maxnorm{\E X} \inftynorm{ X - \E X} + m\maxnorm{\E X}^2.
\end{align*}
Let $t\ge 0$.
Since $\E\abs{X_{ij} - \E X_{ij}}\le 2\E\abs{X_{ij}}\le 2\sig^2$
and
$8\log en\le m \sig^2$, Lemma~\ref{lem:row L1 bound} implies that
the event
\[
\cA_1
\weq \Big\{ \inftynorm{X-\E X} < 2m\sig^2 + (1+t)\sqrt{8m\sig^2\log en} \, \Big\}
\]
occurs with probability  $\pr(\cA_1) \ge 1 - (en)^{-t}$. By recalling the assumptions  $\maxnorm{\E X} \le \sig^2$ and $8\log en\le m \sig^2$, we see that
\begin{align*}
(\abs{X - \E X} \, \abs{X}^\top)_{ij}
&\ <\  (\abs{X} \abs{X}^\top)_{ij} + \sigma^2\left(2m\sig^2 + (1+t)\sqrt{8m\sig^2\log en}\right) + m\sigma^4 \\
&\weq (\abs{X} \abs{X}^\top)_{ij} + \left(3 + (1+t)\sqrt{\frac{8\log en}{m\sig^2}}\right)m\sig^4 \\
&\wle (\abs{X} \abs{X}^\top)_{ij} + \left(4 + t\right)m\sig^4
\end{align*}
on $\cA_1$. By multiplying both sides above by $M_{ij}$, and summing with respect to $j$,
we find that
\begin{align*}
\onenorm{(\bar\Delta_1)_{i:}}
\wle \onenorm{ (\abs{X} \abs{X}^\top \odot M)_{i:}}
+ (4+t)nm\sig^4.
\end{align*}
If $\xi_i=1$, then by recalling the trimming threshold $\Ctrim n m \sig^4$ we conclude that
\begin{equation}
 \label{eq:Deltabar1UB}
 \onenorm{(\bar\Delta_1)_{i:}}
 \wle \underbrace{(4+\Ctrim+t)}_{\Ctrim'} n m \sig^4
\end{equation}
on $\cA_1$. Similarly, we obtain an analogous bound for the columns of $\bar\Delta_1$. Let us combine this with Lemma~\ref{lem:Decoupling} stating that the event
\[
\cA_2\weq \set{\spenorm{((X-\E X)X^\top) \odot M \odot N'} < C_1 (t+\Ctrim') \sqrt{nm} \sig^2}
\]
occurs with probability $\Pr(\cA_2)\geq 1 - C_1n^{-1}e^{-t^{1/3}}$ whenever
$4\log en\le m\sig^2$,
$6 e^2 n \sig^2 \le 1$,
and
$t \ge C_1 \left( 1\vee\frac{\log^3 m}{\log^3(1/6en\sig^2)} \right)$,
where $N'=\xi'\zeta'^{\top}$ is defined by
\[
\xi_i'
\weq \indic \set*{
\onenorm{ (\bar\Delta_1)_{i:} } 
\le \Ctrim' n m \sig^4 },
\qquad
\zeta_j'
\weq \indic \set*{
\onenorm{ (\bar\Delta_1)_{:j} } 
\le \Ctrim' n m \sig^4 }.
\]
On $\cA_1\cap\cA_2$ we saw that $\xi_i\leq\xi_i'$ and $\xi_i\leq\zeta_i'$. Therefore $N$ picks a smaller submatrix than $N'$. In particular,
\[
\spenorm{\Delta_1}
\wle\spenorm{((X - \E X)X^{\top})\odot M\odot N'}
\ <\ C_1(2t+4+\Ctrim)\sqrt{nm}\sigma^2
\]
with probability $\Pr(\cA_1\cap\cA_2)\geq 1- (en)^{-t}-C_1n^{-1}e^{-t^{1/3}}$.

(ii) For $\Delta_2$, we use the bound
\[
 \spenorm{\Delta_2}
 \wle \fronorm{\Delta_2}
 \wle \fronorm{(\E X)(X-\E X)^\top}.
\]
Now Lemma~\ref{lem:AX^T  bound} and the assumption $n\sig^2\le 6^{-1}e^{-2}$ imply that
\[
 \fronorm{(\E X)(X-\E X)^\top}
 \ <\ C_2tn\sig^2\sqrt{m\sig^2}
 \wle C_2t\sqrt{nm}\sig^2
\]
with probability at least $1 - n(en)^{-t}$ for all $t\ge 1$.

(iii) For $\Delta_3$, we note that
\begin{align*}
 \spenorm{\Delta_3}
 &\weq \max_{1 \le i \le n} \abs{ ( (\E X \, \E X^\top ) \odot N )_{ii} },
\end{align*}
which implies that
\begin{align*}
 \spenorm{\Delta_3}
 \wle \max_{1 \le i \le n} \sum_{j=1}^m (\E X_{ij})^2
 \wle m \maxnorm{\E X}^2
 \wle m \sig^4.
\end{align*}

(iv) By collecting the above bounds,
we may conclude that
\[
 \spenorm{\Delta}
 \ <\ (4C_1+C_2)(t+\Ctrim)\sqrt{nm}\sig^2 + m \sig^4
\]
with probability at least
\[
1 - (en)^{-t}-C_1n^{-1}e^{-t^{1/3}}-n(en)^{-t}
\wge 1 - (1+C_1+1)n^{-1}e^{-t^{1/3}}
\]
for all $t\ge \big(C_1 \big( 1\vee\frac{\log^3 m}{\log^3(1/6en\sig^2)} \big)\big)\vee 2$.
\qed

\subsection{Proof of Lemma~\ref{lem:hollow gram matrix row L1 bound}}
\label{sec:hollow gram matrix row L1 bound}

By noting that
$( \abs{X} \, \abs{X}^\top )_{ij} = \sum_k \abs{X_{ik} X_{jk}}$
and applying the triangle inequality to the decomposition
\[
 X_{ik} X_{jk}
 \weq (X_{ik} - \E X_{ik}) X_{jk}
 +  (\E X_{ik}) (X_{jk} - \E X_{jk})
 +  (\E X_{ik}) (\E X_{jk}),
\]
we obtain the entrywise matrix inequality
\begin{equation}
 \label{eq:AbsGram1}
 \abs{X} \, \abs{X}^\top
 \wle \abs{X - \E X} \, \abs{X}^\top
 + \abs{\E X} \, \abs{X - \E X}^\top
 + \abs{\E X} \, \abs{\E X}^\top.
\end{equation}
We also note that
\[
 \onenorm{(\abs{\E X} \, \abs{X - \E X}^\top)_{i:}}
 \weq \sum_j \sum_k \abs{\E X_{ik}} \, \abs{X_{jk} - \E X_{jk}}
 \wle n \maxnorm{\E X} \inftynorm{X-\E X},
\]
and
$
 \onenorm{(\abs{\E X} \, \abs{\E X}^\top)_{i:}}
 \le n m \maxnorm{\E X}^2.
$
By multiplying both sides of \eqref{eq:AbsGram1} entrywise by $M$,
it follows that
\begin{equation}
 \label{eq:AbsGram2}
 \begin{aligned}
 &\onenorm{ ((\abs{X} \abs{X}^\top) \odot M)_{i:} } \\
 &\wle \onenorm{ (\abs{X - \E X} \, \abs{X}^\top \odot M)_{i:} }
  +n \maxnorm{\E X} \inftynorm{X-\E X} + n m \maxnorm{\E X}^2.
 \end{aligned}
\end{equation}

Let us derive an upper bound for the first term on the right side of \eqref{eq:AbsGram2}. Assume that $t\geq 6\vee \frac{96\log^2 m}{\log^2(1/6en\sig^2)}$.
Fix an index $1 \le i \le n$,
and define matrices $X_1 \in \R^{1 \times m}$ and $X_2 \in \Z^{(n-1) \times m}$ by
\[
 (X_1)_{1k} = (X-\E X)_{ik}, \qquad 
 (X_2)_{jk} = X_{j'k},
\]
for $j=1,\dots,n-1$ and $k=1,\dots,m$, 
where $j'$ is the $j$th element of $\{1,\dots,n\} \setminus \{i\}$.
Then
\begin{align*}
 \onenorm{ ((\abs{X- \E X} \, \abs{X}^\top) \odot M)_{i:} } 
 &\weq \sum_{j=1}^n \sum_{k=1}^m \abs{ (X_{ik} - \E X_{ik}) X_{jk} M_{ij} } \\
 &\weq \sum_{j=1}^{n-1} \sum_{k=1}^m \abs{ (X_1)_{1k} (X_2)_{jk} }
 \weq \onenorm{( \abs{X_1} \, \abs{X_2}^\top )_{1:}}.
\end{align*}
The random matrices $X_1,X_2$ are independent,
have independent sub-Poisson entries with variance proxy $\sig^2 \le 2 \sig^2$,
and satisfy
$\E \abs{(X_1)_{1k}} \le 2 \sig^2$ and $\E \abs{(X_2)_{jk}} \le 2 \sig^2$ for all $j,k$.
Lemma~\ref{lem:product of centered and integer matrices}:(i) (with $2\sig^2$)
then implies that
\begin{equation}
 \label{eq:AbsGram3}
 \onenorm{ ((\abs{X- \E X} \, \abs{X}^\top) \odot M)_{i:} } 
 \weq \onenorm{( \abs{X_1} \, \abs{X_2}^\top )_{1:}}
 \ <\ 4 t n m \sig^4
\end{equation}
with probability at least
$1 - \frac{\sqrt{t}}{\sqrt{6}} \exp \left( -\frac{\sqrt{t}nm\sig^4}{10\sqrt{6}} \right)
- 2n^{-\sqrt{t}/2\sqrt{6}}$.

Let us bound the second term on the right side of \eqref{eq:AbsGram2}.
Let us assume that $t \ge 0$.
Because $\E\abs{X_{ij} - \E X_{ij}}\le 2\E\abs{X_{ij}}\le 2\sig^2$ and $m\sig^2\ge 8\log en$,
Lemma~\ref{lem:row L1 bound} implies that
\begin{equation}
 \label{eq:AbsGram4}
 \inftynorm{X-\E X}
 \ <\ 2 m \sig^2 + (1+t) \sqrt{8m\sig^2\log en}
 \wle (3+t)m\sig^2
\end{equation}
with probability at least $1-(en)^{-t}$.

Finally, let us substitute in \eqref{eq:AbsGram3}--\eqref{eq:AbsGram4} into \eqref{eq:AbsGram2} for $t\geq 6\vee \frac{96\log^2 m}{\log^2(1/6en\sig^2)}$. By recalling that $\maxnorm{\E X} \le \sig^2$, we may conclude that
\begin{align*}
 \onenorm{ ((\abs{X} \abs{X}^\top) \odot M)_{i:} }
 \wle (4+5 t) n m \sig^4
\end{align*}
with probability at least
\[
 1 - \frac{\sqrt{t}}{\sqrt{6}} \exp \left( -\frac{\sqrt{t}nm\sig^4}{10\sqrt{6}} \right)
 - 2 n^{-\sqrt{t}/2\sqrt{6}} - (en)^{-t}
 \wge 1 - \sqrt{t} \exp \left( -\frac{\sqrt{t}nm\sig^4}{10\sqrt{6}} \right)
 - 3 n^{-\sqrt{t}/2\sqrt{6}}.
\]
In terms of the indicator variable
\[
 \xi_i
 \weq \indic \Big( \onenorm{( \abs{X} \abs{X}^\top \odot M)_{i:}} \le (4+5t) n m \sig^4 \Big),
\]
this is equivalent to
\[
 \E (1-\xi_i)
 \wle \sqrt{t} \exp \left( -\frac{\sqrt{t}nm\sig^4}{10\sqrt{6}} \right) + 3 n^{-\sqrt{t}/2\sqrt{6}}.
\]
Because the above inequality holds for all $i$, it also holds for
$\E n^{-1}\sum_{i}(1-\xi_i)$.
Now the claim follows by Markov's inequality.
\qed

\end{appendix}

\newcommand{\acktext}{We thank Bogumi{\l} Kami\'nski and Paul Van Dooren for stimulating discussions and helpful comments.}

\newcommand{\fundingtext}{
Ian V\"alimaa’s research was partly supported by a doctoral research grant from the Emil Aaltonen Foundation.}

\ifarxiv
\paragraph{Acknowledgments}
\acktext \fundingtext
\fi

\end{document}